\documentclass[reqno,a4paper]{article}

\usepackage{amssymb,amsmath,epsfig,graphics,graphicx,color}

\definecolor{lightblue}{rgb}{0.22,0.45,0.70}
\usepackage[colorlinks=true,breaklinks=true,linkcolor=lightblue,citecolor=lightblue]{hyperref}
\hypersetup{hypertexnames=false}
\usepackage{subfig}
\usepackage{float}
\usepackage{mathrsfs}
\usepackage{comment}
\usepackage{appendix}
\usepackage{tikz}
\usetikzlibrary{arrows.meta,calc}
\tikzset{
	vertex/.style={circle, fill=black, inner sep=1.5pt},
	edgedof/.style={-Stealth, thick},
	edgedofsmall/.style={-Stealth, thin},
	mid/.style={rectangle, fill=black, inner sep=2pt}
}

\usetikzlibrary{arrows,positioning,shapes,fit,calc}
\usepackage{cancel}

\newtheorem{remark}{Remark}[section]
\newtheorem{lemma}{Lemma}[section]
\newtheorem{theorem}{Theorem}[section]
\newtheorem{proposition}{Proposition}[section]
\newtheorem{assumption}{Assumption}[section]

\newtheorem{definition}{Definition}[section]
\newtheorem{example}{Example}[section]

\usepackage{soul}

\renewcommand{\O}{\Omega}

\def\Jh{\mathcal{J}_h}

\def\bx{\boldsymbol{x}}

\newcommand\bv{\boldsymbol{v}}
\newcommand\bn{\boldsymbol{n}}

\newcommand\bt{\boldsymbol{t}}

\def\HdoK{{H^{2}(\E)}}

\def\EE{\mathcal{E}}
\def\EEh{\EE_h}

\def\PiK{\Pi_{\E}^{\bD}}

\def\HdoK{{H^{2}(\E)}}

\def\inte{{\rm int}}
\def\bdry{{\rm bdry}}
\def\EE{\mathcal{E}}
\def\EEh{\EE_h}
\def\EEi{\EE_h^{\inte}}
\def\EEbdry{\EE_h^{\bdry}}
\def\VV{\mathcal{V}}
\def\VVi{\VV_h^{\inte}}
\def\VVbdry{\VV_h^{\bdry}}
\def\HdoK{{H^{2}(\E)}}

\newcommand\Vh{{\mathcal V}_h}

\def\D{\mathcal D}

\def\vb{{\bf v}}
\def\PiK{\Pi_{\E}^{\D}}
\def\qon{{\quad\hbox{on}\quad}}

\def\CH{\mathcal{H}}

\def\CP{\mathcal{P}}

\def\dim{\mathop{\mathrm{\,dim}}\nolimits}

\def\div{\mathop{\mathrm{div}}\nolimits}

\def\dn{\partial_{\bn}} 
\def\E{\mathrm{K}}

\def\ts{\tilde{s}}

\def\HdoO{{H_0^2(\O)}}

\def\HuoO{{H_0^1(\O)}}

\def\M{{\mathbb M}}

\def\O{\Omega}
\def\PiK{\Pi_k^{\E}}

\def\P{{\mathbb P}}
\newcommand{\PP}{ \mathbf{P} }  
\newcommand{\LL}{ \boldsymbol{L} }  

\def\Q{{\mathcal Q}}
\def\R{{\mathbb{R}}}
\def\rot{\mathop{\mathrm{rot}}\nolimits}

\def\VK{V_k^h(\E)}
\def\Vh{V_k^h}

\def\Q{{\mathcal Q}}
\def\nc{{\rm nc}}
\newcommand{\vertiii}[1]{{\left\vert\kern-0.25ex\left\vert\kern-0.25ex\left\vert #1 
		\right\vert\kern-0.25ex\right\vert\kern-0.25ex\right\vert}}

\newcommand\0{\boldsymbol{0}}

\newcommand*{\jump}[1]{\lbrack\hspace{-1.5pt}\lbrack #1\rbrack\hspace{-1.5pt}\rbrack}
\newcommand*{\Jump}[1]{\Big\lbrack\hspace{-3pt}\Big\lbrack #1\Big\rbrack\hspace{-3pt}\Big\rbrack}
\def\D{{\nabla}}

\def\PiK{\Pi_{\E}^{k,\D^2}}
\def\PinablaK{\Pi_{\E}^{k,\nabla}}

\def\PimunoK{\boldsymbol{\Pi}_{k-1}^{K}}

\def\PioKk{\Pi^{k}_{\E}}

\def\PimunoK{\boldsymbol{\Pi}_{\E}^{k-1}}

\def\dof{{\rm dof}}

\def\CP{{\bf CP}}
\def\SSP{{\bf NC}}
\def\CH{{\bf CH}}

\def\HncT{H_{k}^{2,\nc}(\O_h)}

\newcommand{\vertexValueDofs}{d_0^{\vb}}

\newcommand{\edgeValueDofs}{d_0^e}
\newcommand{\edgeNormalDofs}{d_1^e}
\newcommand{\innerDofs}{d_0^{\E}}

\newenvironment{proof}{\noindent{\it Proof.}}{\hfill$\square$}

\begin{document}

\title{Nonconforming virtual element methods for fourth-order nonlinear reaction-diffusion systems: a unified  framework and analysis}
\author{
D. Adak\thanks{{Department of Mathematics, Indian Institute of Technology-Kharagpur, WB, India, E-mail:
{\tt dibyenduadak25t@kgpian.iitkgp.ac.in}}},
 \quad D. Mora\thanks{GIMNAP, Departamento de Matem\'atica, Universidad
del B\'io-B\'io, Concepci\'on, Chile and
CI$^2$MA, Universidad de Concepci\'on, Concepci\'on, Chile.
E-mail: {\tt dmora@ubiobio.cl}}, \quad A. Silgado\thanks{GIMNAP, Departamento de Ciencias B\'asicas, Universidad
del B\'io-B\'io, Chill\'an, Chile. E-mail:
{\tt asilgado@ubiobio.cl}.} 
}  

\date{}
\maketitle

\begin{abstract}
We develop a unified framework for the design and analysis of high-order nonconforming
virtual element methods for nonlinear fourth-order reaction--diffusion problems
in two dimensions, with emphasis on \emph{clamped, Navier, and
Cahn--Hilliard-type} boundary conditions. Time discretization is performed using
the backward Euler scheme, while the spatial approximation relies on
nonconforming virtual element spaces of arbitrary order $k \ge 2$, encompassing
both $C^0$-nonconforming and Morley-type methods. A key contribution of this work is the development of a novel and rigorous unified error analysis for these numerical schemes, applicable to domains that are not necessarily convex, differing from the existing literature. By introducing a class of \emph{Companion operators}, we construct novel \emph{Ritz-type projections} and derive a new error equation that enables us to obtain optimal error estimates for the scheme under a \emph{minimal spatial regularity} assumption on the weak solution. Finally, we present numerical experiments on polygonal meshes as applications of the proposed framework, including the extended Fisher--Kolmogorov equation, and a fourth-order model with Cahn--Hilliard-type boundary conditions, which validate the theoretical results
and illustrate the performance of the method for the three classes of boundary
conditions.

\end{abstract}



\maketitle

\section{Introduction}\label{SECTION:INTRO}
\paragraph{Scope and contributions.} 
The Virtual Element Method~\cite{Beirao2013_M3AS} (abbreviated as VEM) belongs to the class of polytopal
Galerkin schemes for the numerical approximation of Partial Differential
Equations (PDEs). These methods provide a flexible framework for discretizations on
general polytopal meshes, allowing the use of highly irregular polygonal and
polyhedral elements, which are not easily accommodated by classical finite
element methods. This flexibility allows for better adaptability to complex geometries, mesh adaptivity, and the efficient handling of hanging nodes, making them attractive for engineering and scientific applications. 

In particular, one of the key features of VEM is that it does not require explicit knowledge of basis functions inside the elements. Instead, it constructs approximation spaces that satisfy key mathematical properties, such as stability and consistency. Thus, this approach allows VEM to construct schemes of high polynomial degrees and work efficiently with elements of arbitrary shape (including non-convex polygons) without the need for complex numerical integration, which is a common challenge in standard FEMs on non-standard meshes. We refer to \cite{Beirao2013_M3AS}
for further details on these aspects.

The versatility of the VEM has been demonstrated through its application to a wide range of problems in recent years, in both conforming and nonconforming settings (see, for instance, \cite{Beirao2013_M3AS,Ahmad2013_CMA,ALKM2016,CMS2016} for some of the earliest works). For a current state of the art on VEM, we refer to~\cite{Book_VEM2022,Beirao2023_ActaNumer}. Furthermore, significant progress has been made in designing VEM schemes for solving fourth-order PDEs. Constructing $C^1$-conforming spaces remains a well-known challenge within the standard finite element framework, and higher-order nonconforming spaces for fourth-order problems are still relatively scarce, even on simplicial meshes. In contrast, the VEM enables the construction of both conforming and nonconforming schemes of arbitrary order through an appealing and computationally efficient approach. In particular, the authors in~\cite{Brezzi2013_CMAME,Chinosi2016_CMA} have developed families of conforming VE of high-order $ k \geq 2 $, offering greater flexibility compared to classical FEM approaches. One notable example of a $C^1$-conforming element is the application of the lowest-order VEM space to the Cahn--Hilliard equation, as studied in~\cite{Antonietti2016_SINUM}.

On the other hand, in~\cite{Zhao2016_M3AS}, a \( C^0 \)-nonconforming virtual element method (NCVEM) was developed without the need for degrees of freedom associated with the values of first-order derivatives at the vertices, contrasting with the conforming VE that maintains the same polynomial order \( k \). In~\cite{Zhao2018_JSC}, a Morley-type nonconforming virtual element was constructed, which imposed restrictions on the \( C^0 \)-nonconforming function space along each edge to eliminate the \( (k - 2) \)-order edge moments.  Independently, the authors in~\cite{Antonietti2018_M3AS} introduced a fully nonconforming virtual element that posses the same degrees of freedom as the Morley-type VE in \cite{Zhao2018_JSC}, but approached the problem differently by defining the local virtual space as the set of solutions to a biharmonic problem with specific boundary conditions. These VEMs, can be consider as the extension of the popular Morley finite element~\cite{Morley} to general polygonal meshes with high-order polynomial degrees. Other contributions on nonconforming VE schemes for fourth-order problems can be found, for instance, in \cite{CH2020,CKP-SINUM2023,Zhao2023_MathComp}.

Fourth-order models arising in a wide range of applications have been
extensively studied within both conforming and nonconforming VE frameworks.
In fluid mechanics, representative examples include the Brinkman and
Navier--Stokes equations in stream function formulation
\cite{Mora2021_IMA,Adak2024_CMAME,Mora2025_SISC}, while in solid mechanics
applications encompass Kirchhoff--Love and von Kárm\'an plate models
\cite{Mora2018_M2AN,Lovadina2021_M2AN,Shylaja2024_ACM}. Other classes of
problems, such as the reaction--diffusion Fisher--Kolmogorov model,
sub-diffusion equations, and nonlocal reaction systems, have also been
addressed within the VEM framework
\cite{Pei2023_CMA,Li2021_IMA,Adak2023_M3AS}. More recently, a general strategy
for constructing high-order NCVEMs for fourth-order problems
with variable coefficients has been proposed in~\cite{Dedner2022_IMA},
together with a high-order Morley-type VEM for the Cahn--Hilliard equation
\cite{Dedner2024_JSC}.

Beyond the virtual element setting, a broad class of numerical methods has been
developed for the approximation of biharmonic and more general fourth-order
problems. These include nonconforming and mixed formulations \cite{Ciarlet2002_FEM,Das2024_CMA}, discontinuous Galerkin methods, and $C^0$ interior penalty schemes \cite{Mozolevski2003_CMAM,Brenner2012_SINUM,Georgoulis2009_IMA,Dong2023_JSC,DE_IMA2024}. Polygonal discretization techniques, such as hybrid high-order and weak Galerkin methods, have also been investigated \cite{Georgoulis2009_IMA,DE2022}. Owing to their relevance in a wide range of applications, significant effort has been devoted to the numerical analysis of complex fourth-order models, including the extended Fisher--Kolmogorov equation, linear and nonlinear Cahn--Hilliard phase-field systems, micro-electromechanical devices, plate bending models, and fluid flow formulations \cite{Danumjaya-Pani_2006,Gudi-Gupta_camwa2013,Danumjaya2021_CMA,Elliott1989_SINUM,Kay2009_SINUM,GR}. These contributions confirm that the numerical analysis of fourth-order PDEs continues to be an active area of research.

Motivated by the above discussion and the sustained interest in numerical discretisations for fourth-order problems, in this work we focus on the development and analysis of NCVEMs for a class of nonlinear, time-dependent reaction--diffusion systems. The proposed framework is deliberately flexible and can be applied to different nonconforming virtual element spaces under various boundary conditions (BCs). In particular, the methodology encompasses several relevant models discussed above, namely, the extended Fisher--Kolmogorov equation, reaction--diffusion systems with Cahn--Hilliard-type BCs, and fluid flow problems in pure stream-function formulation.

Among these settings, a particularly challenging and relevant case arises when Cahn--Hilliard-type BCs are imposed. For the BCs described in~\eqref{Cahn-Hilliard:BCs}, a no-flux constraint is enforced on both $u$ and $\Delta u$ along the boundary of the domain. Such conditions play a fundamental role in the modelling of systems that preserve mass and energy (see, e.g., Remark~\ref{remark:boundary-apps}) and naturally appear in the analysis of Cahn--Hilliard-type problems; see, for instance,~\cite{Elliott1989_SINUM,Antonietti2016_SINUM,Feng2007_MathComp}. Despite their relevance, these BCs have received comparatively limited attention within the NCVEM framework. To the best of our knowledge, the recent contribution~\cite{Dedner2024_JSC} is among the few works addressing this case. This gap motivates the need for a unified variational and numerical framework capable of incorporating Cahn--Hilliard-type BCs alongside more classical configurations.

To this end, we present the continuous variational formulation under three distinct types of BCs: the less-studied Cahn--Hilliard case and, as by-products, the clamped and  Navier (also known as simply supported; see~\cite{Li2023_IMA}) cases. For each configuration, we establish an unified well-posedness result. Building upon the abstract framework introduced in~\cite{Dedner2022_IMA}, which relies on the notion of $\dof$-tuples to express NCVEMs in a unified manner, we further extend this approach to nonlinear, time-dependent regimes under different BCs. Moreover, we provide a unified well-posedness analysis for both the semi-discrete and fully discrete problems. In particular, existence and uniqueness of the fully-discrete scheme are obtained via a fixed-point argument under standard small time-step assumptions.

We would like to emphasize a challenging aspect that arises from the nonconformity of the discrete spaces and the approximation of the associated multilinear forms. These factors significantly increase the complexity of deriving rigorous error estimates for nonconforming Galerkin schemes. The challenge is further amplified within the VEM framework, where the trial and test functions are not explicitly known and suitable projection operators must be employed. Moreover, realistic scenarios involving non-convex domains, non-homogeneous boundary data, Navier or Cahn--Hilliard type often require high regularity of the continuous solution to guarantee optimal theoretical convergence rates—regularity that may not be satisfied in practice. In particular, we recall that, unlike the problem with homogeneous clamped BCs, the additional regularity index of the biharmonic problem can be less than $1/2$. In fact, under Cahn--Hilliard BCs, the regularity can be arbitrarily close to zero, even when the domain $\O$ is convex~\cite{Brenner2012_SINUM}. These situations further complicate the analysis but underscore the importance of developing a rigorous theoretical framework to address such challenges.

Early studies (see, e.g., \cite{Zhao2016_M3AS,Zhao2018_JSC,Antonietti2018_M3AS}) introduced NCVEMs for steady fourth-order problems in convex domains; however, the corresponding analyses relied on strong regularity assumptions on the exact solution (at least $H^4$-regularity), which restricted their practical applicability,  as discussed above. Similar limitations can be observed in the context of time-dependent problems; see, for instance, \cite{Li2021_IMA,Pei2023_CMA}. To relax these assumptions, several contributions have focused on the stationary setting, proposing enriching and companion operators that map nonconforming VE functions into suitable conforming spaces. This strategy has enabled sharper error analyses under weaker regularity requirements (cf.~\cite{Huang2021_JCAM,CKP-SINUM2023}). More recently, \cite{Khot2025_MathComp} introduced a new companion operator for general-degree Morley-type VE spaces, establishing enhanced orthogonality properties and best-approximation estimates, and successfully applying the approach to a steady Biot-plate model coupling fourth- and second-order equations.

Despite these approaches, linear and nonlinear time-dependent fourth-order problems within the NCVEM framework have received comparatively little attention. Motivated by this gap, we introduce a new class of Ritz-type operators for NCVEMs within a unified analytical framework, specifically designed for the error analysis of evolutionary biharmonic problems.

The construction of these operators follows a systematic three-step procedure. First, we design a suitable companion operator $\mathcal{J}_h$, endowed with orthogonality properties and best-approximation estimates, extending the companion mappings previously introduced in~\cite{Khot2025_MathComp}. Second, we introduce an auxiliary bilinear form $\widehat{A}(\cdot,\cdot)$ that consistently combines contributions at the $H^2$, $H^1$, and $L^2$ levels. Composing this form with the companion operator yields, for any $\varphi \in V$, a linear functional $\widehat{A}_{\varphi}(\cdot)$ (see Figure~\ref{fig:functional} for a schematic illustration). Finally, the Ritz operator is defined as the solution of an associated elliptic problem driven by $\widehat{A}(\cdot,\cdot)$ and the functional $\widehat{A}_{\varphi}(\cdot)$; see \eqref{Ritz:operator}.

The resulting Ritz-type projections are well defined under minimal regularity assumptions and exhibit optimal approximation properties. In particular, we establish optimal error estimates in the broken $H^2$ norm and in the broken $H^1$ seminorm. These estimates play a central role in the analysis of the nonlinear time-dependent problem and are key to proving optimal convergence rates in the energy norm, under the assumption of quasi-uniform meshes. In contrast to existing Ritz-type operators for NCVEMs (see, e.g., \cite{Zhao2019_parabolic,Li2021_IMA,Pei2023_CMA,Adak2023_M3AS}), the proposed construction requires only $H^2$-regularity of the exact solution and attains optimal rates under the mild assumption $u \in H^{2+s}(\Omega)$ for some $s>0$ (cf. Assumption~\ref{assump:reg:add}, Remarks~\ref{remark:Ritz:min:regu} and~\ref{remark:minimal:reg}).

As a result, the present work establishes a new setting for the error analysis of NCVEMs applied to time-
dependent biharmonic problems, encompassing more general and realistic scenarios, including non-convex
domains and a broad class of boundary conditions.

While a related approach was recently proposed in~\cite{DNR2025_JSC} within a lowest-order finite element framework—where the authors employ a companion-based Ritz projection to establish optimal convergence rates under minimal regularity assumptions for the extended Fisher–Kolmogorov equation with clamped boundary conditions—our setting entails additional and fundamental challenges. These stem from the consideration of different BCs, the use of high-order schemes, and the intrinsic construction of the VEM. In particular, the VEM framework introduces further difficulties due to the absence of explicit representations of the basis functions, which requires a careful definition of the discrete bilinear forms through suitable polynomial projection operators. As a consequence, a new type of error equation arises in the analysis (cf. identity~\eqref{main:semi:dis:err}), requiring the development of techniques that differ from those typically employed in the FEM context (see, e.g., Theorem~\ref{Error:Ellip}).

The construction of the proposed Ritz-type projection operators extends existing
companion-operator techniques to a more general and flexible setting. Although
inspired by the ideas introduced in~\cite{Khot2025_MathComp}, the present approach
introduces a new class of companion operators that map \(C^0\)-nonconforming
virtual element spaces into conforming continuous function spaces and allow for
a unified treatment of both \(C^0\)- and Morley-type VEMs. A key feature of this
construction is its ability to accommodate different classes of boundary
conditions within a single framework, as the operator is intrinsically tied to
the choice of degrees of freedom defining each virtual element space. These
companion operators are central to the definition of the Ritz-type projections
used in the unified error analysis (cf.~\eqref{Ritz:operator}).

Based on the above discussion, we summarize the main contributions of this work
to the development and analysis of NCVEMs as follows:
\begin{itemize}
	\item \textbf{Unified treatment of boundary conditions.}
	A unified analysis of time-dependent and nonlinear fourth-order problems
	under different boundary conditions, including nonstandard Cahn--Hilliard
	BCs as well as the classical clamped and Navier types.
	
	\item \textbf{Unified VEM framework.}
	A unified formulation of NCVEMs supporting the construction and analysis of
	high-order schemes on polygonal meshes, including $C^0$-nonconforming and
	Morley-type VEMs, for evolutionary nonlinear fourth-order problems posed on
	general polygonal (not necessarily convex) domains.
	
	\item \textbf{New error analysis via companion and Ritz-type operators.}
	The design and rigorous error analysis of new companion operators and
	Ritz-type projections, yielding a unified convergence theory under minimal
	regularity assumptions on the weak solution.
	
	\item \textbf{Applications and numerical validation.}
	The application of the proposed framework to a broad class of fourth-order
	models, together with numerical experiments that corroborate the theoretical
	error estimates and illustrate the robustness and performance of the
	resulting schemes; see Sections~\ref{SECTION:APPLICATIONS} and \ref{SECTION:NUMERICAL:EXPERIMENTS}.
\end{itemize}

\paragraph{Outline.}  The paper is organized as follows. Section~\ref{SECTION:MODEL:PROBLEM} introduces the notation, functional setting, and continuous problem, together with its variational formulation and well-posedness. An abstract VEM framework for time-dependent fourth-order problems and corresponding a priori error estimates for semi- and fully-discrete schemes are developed in Sections~\ref{SECTION:VEM:ABSTRACT} and~\ref{SECTION:ERROR:UNIFIED}. Specific $C^0$-nonconforming and fully nonconforming VEM spaces for biharmonic problems, along with the construction of the associated companion operators, are presented in Section~\ref{SECTION:VEM:SPACES}. Applications to different model problems, including the Extended Fisher–Kolmogorov equation, a fourth-order model with Cahn–Hilliard BCs, and Stokes flow in stream function form, are discussed in Section~\ref{SECTION:APPLICATIONS}. Numerical experiments validating the theoretical results are reported in Section~\ref{SECTION:NUMERICAL:EXPERIMENTS}, and concluding remarks and future perspectives are given in Section~\ref{SECTION:CONCLUSIONS}.

\setcounter{equation}{0}

\section{Preliminaries and continuous formulation}\label{SECTION:MODEL:PROBLEM}

In this section we shall introduce the notation of this paper and describe the model problem.
\subsection{Notation}
Throughout the paper, we have dealt with the following notations. Henceforth, $\O\subset \R^2$ be polygonal domain (non necessary convex) with boundary $\Gamma:=\partial\O$.  We denote the outward normal to $\Gamma$ by $\bn=(n_1,n_2)^\textrm{t}$, and the unit tangent vector of $\Gamma$ by $\bt= (t_1,t_2)^\textrm{t}$ and oriented such that $t_1=-n_2$, and $t_2=n_1$.

For each $\mathcal{D}$ open subdomain of $\Omega$, the space of square integrable
scalar functions will be denoted by $L^2(\mathcal{D})$ and endowed with the standard inner product
$(v,w)_{\mathcal{D}}:=\int_{\mathcal{D}}vw$. For each $s$, we define $H^s(\mathcal{D})$, the Sobolev space with standard seminorm and norm $|\cdot|_{s,\mathcal{D}}$  and $\|\cdot\|_{s,\mathcal{D}}$.

Let us denote by $t$ the time variable, which takes values in the interval $I:=(0,T]$, where $T$ is a given final time. Then, we define the time derivative
$\partial_t v:=\frac{d v}{dt}$, $\partial_{tt} v:=\frac{d^2 v}{dt^2}$. Moreover, we define the Bochner space $L^2(0,T;H^s(\mathcal{D}))$ and $L^{\infty}(0,T;H^s(\mathcal{D}))$ with the norm 
\begin{equation*}
	\begin{aligned}
		\|v\|_{L^2(0,t;H^s(\mathcal{D}))}:=\Bigg(\int_0^t \|v(\cdot, \sigma)\|_{s,\mathcal{D}}^2 d\sigma \Bigg)^{1/2}; \quad
		\|v\|_{L^{\infty}(0,t;H^s(\mathcal{D}))}:=\underset{0 \leq \sigma \leq t}{\text{ess sup}}\|v(\cdot, \sigma)\|_{s,\mathcal{D}}.
	\end{aligned}
\end{equation*} 
In addition, given any Hilbert space $V$,
we will denote by $\boldsymbol{V}:=[V]^2$ the space of
vectors functions with entries in $V$ and  we denote by $V'$ the dual space of $V$.

On the other hand, for each scalar and vector functions $v$ and $\bv$, we recall the following differential operators; $\div \bv:= \partial_1 v_1  + \partial_2 v_2$, 
$\rot \bv = \partial_1 v_2  - \partial_2 v_1,$
$\nabla v:=(\partial_1 v ,  \partial_2 v )^\textrm{t}$.
 Moreover,  $\D^2 v:=(\partial_{ij}v)_{1\le i,j\le2}$ and $\partial_{\bn} v:= \nabla v \cdot \bn$ denote the Hessian
 matrix and the normal derivative of $v$, respectively.
 
Throughout the paper, we also will employ the following spaces
\begin{equation*}
	\begin{split}
		H_0^1(\O) &:= \left\{ v \in H^1(\O): v= 0  \qon \Gamma\right\}, \\
		H_0^2(\O) &:= \left\{ v \in H^2(\O): v = \dn v = 0  \qon \Gamma\right\}, \\
		\widetilde{H}^2(\O)	&:= \left\{ v \in H^2(\O) :  \dn v = 0 \qon \Gamma \right\}.
	\end{split}	
\end{equation*}

Finally, we note that the symbol $C$, possibly endowed with subscripts,
superscripts, or modifiers such as $\widehat{C}$ or $\widetilde{C}$,
denotes a generic positive constant whose value may change from line to line.
Unless otherwise stated, this constant is independent of the mesh size and
time-step parameters, but it may depend on other quantities related to the
continuous problem or on mesh regularity parameters.

Moreover, for two nonnegative quantities $a$ and $b$, we write
$a \lesssim b$ and $a \gtrsim b$ to mean $a \leq C b$ and $a \geq C b$,
respectively, and we write $a \approx b$ when both inequalities hold.
	
\subsection{The model problem}
We consider the following fourth-order nonlinear reaction-diffusion equation: 
\begin{equation}\label{reac:diffusion-problem}
	\left\{
\begin{aligned}			
		\partial_t u + \alpha_1 \Delta^2 u  - \alpha_2 \Delta u 
		&=f(u) \:\:  \qquad\textrm{in} \quad\Omega \times (0,T],\\
		u(\bx,0)&=u_0(\bx)  \qquad\textrm{in} \quad\Omega,
		\end{aligned}
\right.
\end{equation}
associated with one of the following three BCs~\cite{Brenner-Monk-Sun2014}:
\begin{align}
	\text{Clamped BCs ($\CP$):}& \quad  u=\dn u=0 \quad\quad  \qquad \text{on} \quad \Gamma \times (0,T]; \label{Campled:BCs}\\
	\text{Navier/Simply supported BCs ($\SSP$):}& \quad u=\Delta u=0  \qquad  \: \quad \quad \text{on} \quad  \Gamma \times (0,T];\label{supported:BCs}\\
	\text{Cahn--Hilliard BCs ($\CH$):}& \quad \dn u=\dn \Delta u=0   \quad \quad \text{on} \quad  \Gamma \times (0,T]. \label{Cahn-Hilliard:BCs} 
\end{align}
Here, the parameters $\alpha_i$, with $i \in \{1,2\}$, are positive constants and the nonlinear force function $f$ satisfies the Lipschitz continuity condition, i.e., there exists $L_{f}>0$ such that
\begin{equation}\label{ineq:Lipt:f}
	|f(x)-f(y)| \leq L_{f} |x-y| \quad \forall x,y \in \R.
\end{equation}

The main goal of this section is to set up a unified primal framework with any one of the BCs
stated in \eqref{Campled:BCs}-\eqref{Cahn-Hilliard:BCs}. Thus, using twice integration by parts and the BCs the weak formulation that corresponds to \eqref{reac:diffusion-problem}, read as: seek 
$$u\in L^2(0,T;V),$$
 such that for a.e. $t \in (0,T]$
\begin{equation}\label{conti:weak:form}
	\left\{
	\begin{aligned}	
	M(\partial_t u,v)+ \alpha_1 A(u,v)+ \alpha_2 B(u,v)&= F(u;v) \qquad \forall v \in V,\\
	u(0) &= u_0, 
		\end{aligned}
	\right.
\end{equation}	
where the  forms are defined by
\begin{align}
	M(\cdot,\cdot): V \times V&\to \R, \qquad
	M(v, w) :=  (v, w)_{\O} \label{bilinear-contM}
	\\
	A(\cdot,\cdot): V \times V &\to \R, \qquad 
	A(v, w):=  (\D^2 v,\D^2 w)_{\O}  \label{bilinear-contA} \\
	B(\cdot,\cdot): V \times V &\to \R, \qquad 
	B(v, w):=(\nabla v,  \nabla w)_{\O}, \label{bilinear-contB}\\
	F(v;\cdot): V  &\to \R, \qquad  F(v;w):= (f(v),w)_{\O}  \label{contF}
\end{align}
and 
\begin{equation}\label{cont:space:V}
	V := \begin{cases}
		\HdoO& \quad \text{for \CP};\\
		H^2(\O) \cap \HuoO& \quad \text{for \SSP};\\
		\widetilde{H}^2(\O)& \quad \text{for \CH}.
	\end{cases}
\end{equation}

The following result establishes the well-posedness of the continuous weak formulation~\eqref{conti:weak:form}.
\begin{theorem}
	Assume that $f$ satisfies the Lipschitz continuity condition \eqref{ineq:Lipt:f} and $u_0 \in V$. Then, there exists a unique solution $u(\cdot, t) \in V$ to problem~\eqref{conti:weak:form}.
\end{theorem}

\begin{proof}
	The proof follows by combining the arguments presented in~\cite{Gudi-Gupta_camwa2013,Danumjaya-Pani_2006}. 
	
\end{proof}

We conclude this section with the following remark concerning the boundary conditions of the problem considered in this work.

\begin{remark}\label{remark:boundary-apps}
For the {\bf CH} boundary conditions in~\eqref{Cahn-Hilliard:BCs}, no-flux constraints are imposed on both $u$ and $\Delta u$ along the boundary of the domain, ensuring the conservation of mass and energy. This setting is central to the analysis of Cahn--Hilliard-type problems; see, for instance,~\cite{Elliott1989_SINUM,Dedner2024_JSC}.

In contrast, the boundary conditions~\eqref{Campled:BCs} and~\eqref{supported:BCs}, which correspond to clamped and Navier configurations, respectively. These classical conditions arise in structural vibrations, frequency analysis, and a broad range of applications, including wave propagation, reaction--diffusion systems, and fluid flow models formulated in terms of the stream function~\cite{Li2021_IMA,Dedner2022_IMA,Pei2023_CMA,Li2023_IMA,GR,Mora2021_IMA}.

\end{remark}
\setcounter{equation}{0}
\section{The virtual element method in an abstract framework}\label{SECTION:VEM:ABSTRACT}
In this section, we will introduce the main ingredients and basic setting needed for the construction of NCVEMs in a unified way. We extend the approach presented in~\cite{Dedner2022_IMA} to the time dependent problems with the three different 
BCs $\CP, \SSP$ and $\CH$ (cf. \eqref{Campled:BCs}-\eqref{Cahn-Hilliard:BCs}).

\subsection{Discrete setting and notation}
From now on, we denote by $\E$ a general polygon, $e$ a general edge of $\partial\E$, 
while the symbols $h_\E$ and $h_e$ will denote the diameter of $\E$ and the length of $e$, respectively.  
Let $\{\O_h\}_{h>0}$ be a sequence of decompositions of $\O$ into general non-overlapping polygons $\E$, where 
$$h:=\max_{\E\in\O_h}h_\E.$$
For each element $\E$ we denote by $\EEh^\E$ the set of its edges, while  the set of  all the 
edges in $\O_h$ will be denote by $\EE^{{\rm tot}}_h$. We decompose this set as
the following union: $\EE^{{\rm tot}}_h := \EEi \cup \EEbdry$, where $\EEi$ and $\EEbdry$ are the set of
interior and boundary edges, respectively.  For the set of all the vertices we have an analogous 
notation, i.e., we  denote  by  $\VV^{{\rm tot}}_h := \VVi \cup \VVbdry$ the set of vertices
in $\O_h$, where  $\VVi$ and $\VVbdry$ are the set of interior and boundary vertices, respectively. 
For each element $\E$ we denote by $\VV_h^\E$ the set of its vertices.

For each element $\E \in \O_h$, we denote by $\bn_{\E}$ its unit outward normal vector 
and by $\bt_{\E}$ its tangential vector along the boundary $\partial \E$. We assume a local
orientation of $\partial \E $ so that $\bn_{\E}$ point out of $\E$.  Moreover, we adopt the notation $\bn_e$ and $\bt_{e}$ for a unit normal and tangential vector of an edge $e \in \EE^{{\rm tot}}_h$, respectively.
Furthermore, for each $\E$ and  any integer $\ell\geq 0$, we introduce the following spaces:
\begin{itemize}
	\item For any  open bounded subdomain $\mathcal{D} \subset\R^2$, we define the space $\P_{\ell}(\mathcal{D})$ as the space of polynomials of degree  up to $\ell$ defined on $\mathcal{D}$  and we denote by $\PP_{\ell}(\mathcal{D})$ 
	its vectorial version, i.e.,  $\PP_{\ell}(\mathcal{D}):=[\P_{\ell}(\mathcal{D})]^2$.  For $\ell=-1,-2$, we conventionally assume that $\mathbb{P}_{\ell}(\mathcal{D})=\{ 0 \}$; 
	\item  We define the discontinuous piecewise $\ell$-order polynomial by: 
	\begin{equation*}
		\P_{\ell}(\O_h):= \big\{ q \in L^2(\O): q|_{\E} \in \P_{\ell}(\E)  \quad \forall \E \in \O_h \big\}.
	\end{equation*}
\end{itemize}
In this work  the degree of accuracy of the VEM will be $k \geq 2$, then  
we need to introduce an appropriately scaled basis for the space of polynomials $\P_{k}(\E)$. 
First, we denote by $\M^{\ast}_{\ell}(\E)$  the set of
the two-dimensional scaled monomials defined on each polygon $\E$ as follows:
\begin{equation*}
	\M^{\ast}_{\ell}(\E):=\left\{ \Big(\frac{\bx-\bx_{\E}}{h_\E}\Big)^{\boldsymbol{\beta}} : |\boldsymbol{\beta}| = \ell  \right\},
\end{equation*}
where we have used the classical notation for a multi-index $\boldsymbol{\beta}$ and $\bx_{\E}$ is the barycenter of $\E$. Then, we define  
$$\M_k(\E):= \bigcup_{\ell \leq k} \M^{\ast}_{\ell}(\E)=:\{m_{j} \}_{j=1}^{d_{k}},$$ as a basis of $\P_{k}(\E)$, where  $d_k= \dim(\P_k(\E))$.  For $\ell=-1,-2$, we conventionally assume that $\M_{\ell}(\mathcal{D})=\{ 0 \}$. 

Analogously, we consider the set of the scaled monomials defined on each edge $e$:
\[
\M_{\ell}(e):=\left\{1, \frac{x-x_e}{h_e}, \Big(\frac{x-x_e}{h_e}\Big)^2, \ldots, \Big(\frac{x-x_e}{h_e}\Big)^{\ell}\right\},
\] 
where $x_e$ is the midpoint of $e$.

For every $\delta>0$, we define the following broken Sobolev spaces
\begin{equation*}
	H^{\delta}(\O_h):= \{ v \in L^2(\O): v|_{\E} \in H^{\delta}(\E)  \quad \forall \E \in \O_h\},
\end{equation*}
and we endow these spaces with the following broken seminorm:
\begin{equation*}
	|v|_{\delta,h}:=\Big( \: \sum_{\E\in\O_h}|v|_{\delta,\E}^{2} \Big)^{1/2},
\end{equation*}
where $|\cdot|_{\delta,\E}$ is the usual seminorm in $H^{\delta}(\E)$.

Next, we will define the jump operator. First, for each $v_h \in H^2(\O_h)$,  
we denote by $v_h^{\pm}$  the trace of $v_h|_{\E^{\pm}}$, with 
$e \subset \partial \E^{+} \cap \partial \E^{-}$. 
Then, the jump operator $\jump{\cdot}$ is defined as follows:
\begin{equation*}
	\jump{v_h} 
	:= \begin{cases}
		v_h^{+}-v_h^{-} & \text{for every  $e \in \EEi$,}\\
		v_h|_e  &  \text{for every  $e \in \EEbdry$}.
	\end{cases}
\end{equation*}

Let $\mathcal{C}^0(\VVi)$ and $\mathcal{C}^0(\VV^{{\rm tot}}_h)$ be the sets of functions
continuous at internal vertices and at all the vertices of $\O_h$, respectively. Then, we consider the space
\begin{equation*}
	\mathcal{H}^0(\O_h) := \begin{cases}
		\Big\{ v_h \in H^2(\O_h) : v_h  \in \mathcal{C}^0(\VVi), \: \: v_h(\vb_i)=0 \quad  \forall \vb_i \in \VVbdry	\Big\}	& \quad \text{for \CP \: and \SSP};\vspace*{2mm}\\ 
		\Big\{ v_h \in H^2(\O_h) : v_h  \in \mathcal{C}^0(\VV^{{\rm tot}}_h)  	\Big\}& \quad \text{for \CH}.
	\end{cases}
\end{equation*}

Now, with the space define above  we shall introduce the following subspaces of $H^2(\O_h)$: 
	\begin{equation}\label{H2:nonconf}
		\begin{split}
	\HncT := \Big\{ v_h  \in \mathcal{H}^0(\O_h) : & \quad (\jump{ v_h},p)_{0,e}=0 \quad \: \forall p \in \P_{k-3}(e) \quad \forall  e \in \EE^{\ast}_h, \\
	& \quad  (\jump{\partial_{\bn_e} v_h},q)_{0,e}=0 \quad \: \forall q \in \P_{k-2}(e) \quad \forall  e \in \EE^{{\rm tot}}_h \Big\},
		\end{split}
	\end{equation}
where 
\begin{equation*}
\EE^{\ast}_h = 	 
\begin{cases}
\EE^{{\rm tot}}_h   & \quad \text{for \CP \: and \SSP};\vspace*{2mm}\\	
\EE^{\inte}_h   & \quad \text{for \CH}.
\end{cases}	
\end{equation*}

We assume that the NCVEMs we want to construct, satisfy  the following fundamental ingredients:
\begin{assumption}\label{assump:VEM}
	Assume the following holds for any fixed $h > 0$ and for a fixed integer $k \geq 2$.	
	
	\begin{itemize}
		\item[${\bf (A1)}$] The mesh $\O_h$ consists only of simple polygons. More precisely, a simple polygon refers to the fact that the boundary of each element must not intersect itself and is made up of a finite number of straight line segments.
		
		\item[${\bf (A2)}$] The finite dimensional function space, which will be denote by $V_{k}^h$, satisfies $\P_{k}(\O_h) \subseteq V_{k}^h  \subset \HncT$.
		Moreover, we denote the restriction of this global VEM space  $V_{k}^h$  to an element $\E$ by $V_{k}^h(\E):= V_{k}^h|_{\E}$ and this space must contain the polynomial space $\P_k(\E)$. 
		
		More precisely, the global virtual element space is defined as: 
		\begin{equation}\label{global:VEM}
			\Vh := \left\{ v_h \in \HncT : v_h|_{\E} \in \VK \quad   \forall \E \in \O_h  \right\}.
		\end{equation} 
		
		\item[${\bf (A3)}$]  There exist discrete bilinear forms $M_h(\cdot,\cdot), A_h(\cdot,\cdot)$ and $B_h(\cdot,\cdot)$ approximating the continuous  bilinear forms
		$M(\cdot,\cdot), A(\cdot,\cdot)$, and $B(\cdot,\cdot)$ defined in \eqref{bilinear-contM}, \eqref{bilinear-contA}, and \eqref{bilinear-contB}, respectively. These  discrete forms may be split over the elements in the mesh $\O_h$ as:
		\begin{align}
			M_h: V_{k}^h  \times V_{k}^h  \to \R,& \qquad M_h(v_h,w_h) := \sum_{\E \in \O_h} M^{\E}_h(v_h,w_h), \label{general:discrete:Mh}\\
			A_h: V_{k}^h  \times V_{k}^h  \to \R,&  \qquad A_h(v_h,w_h) := \sum_{\E \in \O_h} A^{\E}_h(v_h,w_h), \label{general:discrete:Ah}\\
			B_h: V_{k}^h  \times V_{k}^h  \to \R,&  \qquad B_h(v_h,w_h) := \sum_{\E \in \O_h} B^{\E}_h(v_h,w_h), \label{general:discrete:Bh}
		\end{align}
	where the local bilinear forms
	\( M_h^{\E}: V_k^h(\E) \times V_k^h(\E) \to \mathbb{R} \),
	\( A_h^{\E}: V_k^h(\E) \times V_k^h(\E) \to \mathbb{R} \), and
	\( B_h^{\E}: V_k^h(\E) \times V_k^h(\E) \to \mathbb{R} \)
	are suitable approximations of the corresponding continuous bilinear forms
	restricted to the element \( \E \). 
		
		\item[${\bf (A4)}$] There exists  $F_h(\cdot; \cdot): V_{k}^h \to \R$, which approximates the forcing term $F(\cdot; \cdot)$ defined in~\eqref{contF}.
		
		\item[${\bf (A5)}$]  The local forms $M_h^{\E}(\cdot,\cdot)$, $A_{h}^{\E}(\cdot,\cdot)$ and  $B^{\E}_h(\cdot,\cdot)$ satisfy the following properties:
		\begin{itemize}
			\item \textit{$k$-consistency}: for all $\E\in\O_h$, we have that
		\begin{equation*}
			\begin{alignedat}{3}
				M_{h}^{\E}(q,v_h) &= M^{\E}(q,v_h)
				&\quad& \forall q\in\P_k(\E)
				&\quad& \forall v_h\in V_{k}^h(\E),\\
				A_{h}^{\E}(q,v_h) &= A^{\E}(q,v_h)
				&& \forall q\in\P_k(\E)
				&& \forall v_h\in V_{k}^h(\E),\\
				B_{h}^{\E}(q,v_h) &= B^{\E}(q,v_h)
				&& \forall q\in\P_k(\E)
				&& \forall v_h\in V_{k}^h(\E).
			\end{alignedat}
		\end{equation*}
			\item \textit{Stability}: there exist positive constants independent of $\E$, such that:
		\begin{equation*}
			\begin{alignedat}{3}
				M_{h}^{\E}(v_h,v_h) &\approx M^{\E}(v_h,v_h)
				&\quad& \forall v_h\in V_{k}^h(\E),\\
				A_{h}^{\E}(v_h,v_h) &\approx A^{\E}(v_h,v_h)
				&& \forall v_h\in V_{k}^h(\E),\\
				B_{h}^{\E}(v_h,v_h) &\approx B^{\E}(v_h,v_h)
				&& \forall v_h\in V_{k}^h(\E).
			\end{alignedat}
		\end{equation*}
		\end{itemize}
		
	\end{itemize}	
\end{assumption}

	An explicit representation for the multilinear forms will be provided in subsection~\ref{polynomialproy} by using  the local forms defined in \eqref{local:formM}, \eqref{local:formA}, \eqref{local:formB} and \eqref{local:formF}.

\subsection{The degrees of freedom}
In this subsection, we recall the concept of a degrees of freedom tuple (DoFs-tuple), used to generically describe the degrees of freedom (DoFs) relating to a VE space on each element of the polygonal decomposition. For this purpose, we will follow some tools of the approach recently presented in~\cite{Dedner2022_IMA}. 
In what follows, let $\E \in \O_h$ be a fixed polygon. 
\begin{definition}[DoFs-tuple]\label{def:dof:tuple}  
	For $H^2$-nonconforming virtual element spaces, let the DoFs-tuple, $\boldsymbol{\xi} \in \mathbb{Z}^{4}$, be defined as  
	\begin{align}\label{dof:tuple}
		\boldsymbol{\xi} = \big( \vertexValueDofs, \edgeValueDofs,\edgeNormalDofs ,\innerDofs \big).
	\end{align}
	The entries correspond to the number of moments used in the definition of our DoFs (see below in Definitions~\ref{def:local:dofs} and~\ref{def:global:dofs}). The indexes $d^{\vb}_0$ and $d^e_i$, for $i=0,1$, encoding the information for the vertex and edge moments, respectively,  while the index $\innerDofs$ have the information of the inner moments. The subscript $i=0$ corresponds to moments of the function values and $i=1$ to derivative moments on  edges.
\end{definition}
From the DoFs-tuple we are able to define the corresponding local DoFs.
\begin{definition}[Associated local DoFs]\label{def:local:dofs}
	For a function $v_h \in H^2(\E)$, the local DoFs $\Lambda^{\E}_{\boldsymbol{\xi}}$ associated to the DoFs-tuple $\boldsymbol{\xi}$ defined in \eqref{dof:tuple} are given by the following set of linear operators:
	\begin{itemize}
		\item[${\bf (D1)}$] The values of $ v_h$ at each vertex $\vb \in \VV_h^{\E}$.
		\item[${\bf (D2)}$] The moments of $\partial_{\bn_e}^iv_h$ up to order $d^e_i$ on each $e \in \EE_h^{\E}$, for $i =0,1,$
		\begin{align*}
			h_{e}^{-1+i} ( \partial_{\bn_e}^{i} v_h, p)_{e}  \quad \forall p \in \M_{d^e_i}(e).
		\end{align*} 
		\item[${\bf (D3)}$] The  moments of $v_h$ up to order $\innerDofs$ inside $\E$,
		\begin{align*}
			h^{-2}_{\E} ( v_h, p )_{\E} \quad \forall p \in \M_{\innerDofs}(\E).
		\end{align*} 
		\end{itemize}		
\end{definition}

To fix ideas and avoid unnecessary technical difficulties, we introduce the following definition.

\begin{definition}
	Let $ k \geq 2 $. The entries of the DoFs tuple associated with an
	$ H^{2} $-nonconforming discretization are required to satisfy
	\begin{equation*}
		d^{\vb}_0 = 0, \qquad
		d^e_0 \geq k-3, \qquad
		d^e_1 \geq k-3, \qquad
		\text{and} \qquad
		\innerDofs \geq k-4.
	\end{equation*}
\end{definition}

\begin{remark}
	For notational consistency, we adopt the standard convention
	$ \partial_{\bn_e}^0 v_h = v_h $. 
	Moreover, whenever an entry of the DoFs tuple $ \boldsymbol{\xi} $ is equal to zero,
	this indicates that constant moments are employed.
\end{remark}

Under the above conventions, we now state the following fundamental assumption.

\begin{assumption}[Unisolvency]\label{assump:unisolvency}
	The set of degrees of freedom $ \mathbf{(D1)}\text{--}\mathbf{(D3)} $
	is unisolvent in the space $ \VK $.
\end{assumption}

As a consequence of Assumptions~\ref{assump:VEM} and~\ref{assump:unisolvency},
the local virtual element space $ \VK $ is well defined and finite dimensional.
We denote by $ d_k $ the dimension of $ \VK $, and by
$ \dof_i(\cdot) $, $ 1 \leq i \leq d_k $, the linear functional that associates to any sufficiently smooth function $ v $
its $ i $-th local degree of freedom $ \dof_i(v) $.

To illustrate these concepts, we next present examples of DoFs tuples arising from commonly used VEM spaces.
Further discussion and motivation for these choices are deferred to
Section~\ref{SECTION:VEM:SPACES}.

\begin{example}\label{example:dofs}
	Let $ v_h \in H^2(\E) $. 
	The DoFs tuple $ \boldsymbol{\xi}^{nc} $ associated with the
	$ C^0 $-nonconforming VEM introduced in~\cite{Zhao2016_M3AS} is given by
	\begin{equation*}
		\boldsymbol{\xi}^{nc} := (0,k-2,k-2,k-4).
	\end{equation*}
	On the other hand, the DoFs tuple $ \boldsymbol{\xi}^{fnc} $ corresponding to the
	Morley-type VEM considered in~\cite{Antonietti2018_M3AS,Zhao2018_JSC}
	is defined as
	\begin{equation*}
		\boldsymbol{\xi}^{fnc} := (0,k-3,k-2,k-4).
	\end{equation*}
	A visualization of these DoFs for $k=2,3,4$ on a pentagonal element
	is shown in Figure~\ref{fig:Dofs:pentagon}.
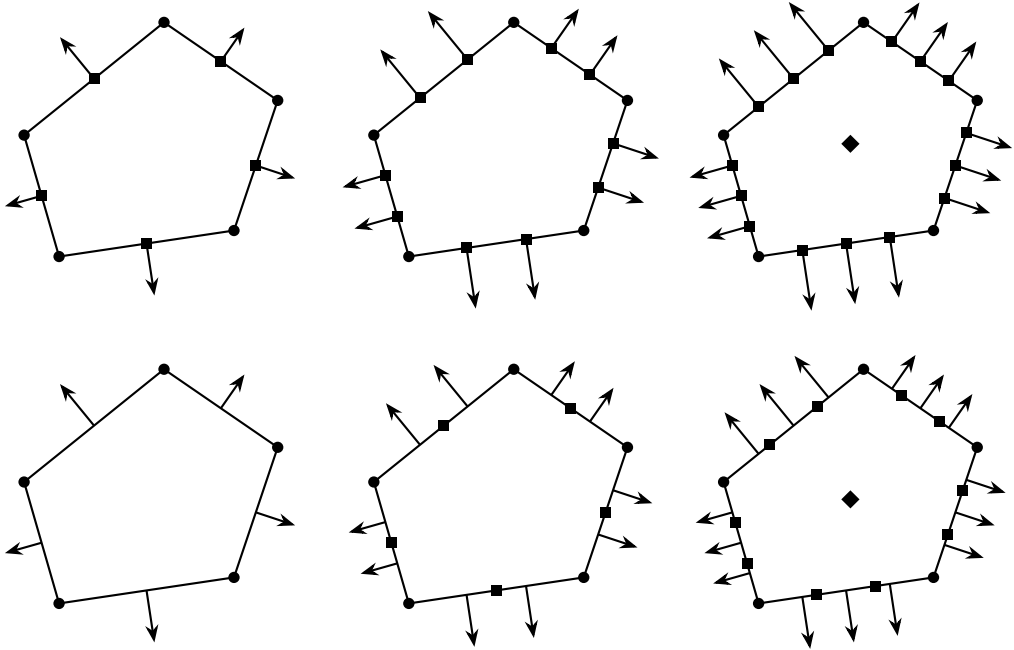
\begin{figure}[ht!]
	\centering
	\begin{tikzpicture}[scale=1.15]
		
		\tikzset{
			vertex/.style={circle, fill=black, inner sep=1.5pt},
			mid/.style={rectangle, fill=black, inner sep=2pt},
			edgedofpoint/.style={rectangle, fill=black, inner sep=2pt},
			edgedof/.style={-Stealth, thick},
			poly/.style={black, thick},
            interiordof/.style={diamond, fill=black, inner sep=1.75pt}
		}
		
		\def\Pentagon{
			\coordinate (A) at (0,0);
			\coordinate (B) at (2.0,0.3);
			\coordinate (C) at (2.5,1.8);
			\coordinate (D) at (1.2,2.7);
			\coordinate (E) at (-0.4,1.4);
			\draw[poly] (A)--(B)--(C)--(D)--(E)--cycle;
			\foreach \p in {A,B,C,D,E}
			\node[vertex] at (\p) {};
		}
		
		
		\begin{scope}[xshift=-2cm,yshift=2cm]
			\Pentagon
			\foreach \P/\Q in {A/B,B/C,C/D,D/E,E/A}
			\node[mid] at ($(\P)!0.5!(\Q)$) {};
			\foreach \P/\Q in {A/B,B/C,C/D,D/E,E/A}{
				\draw[edgedof]
				let \p1=(\P), \p2=(\Q) in
				($(\P)!0.5!(\Q)$)
				-- ++({(\y2-\y1)*0.3},{-(\x2-\x1)*0.3});
			}
		\end{scope}
		
		\begin{scope}[xshift=2cm,yshift=2cm]
			\Pentagon
			\foreach \P/\Q in {A/B,B/C,C/D,D/E,E/A}{
				\node[edgedofpoint] at ($(\P)!0.33!(\Q)$) {};
				\node[edgedofpoint] at ($(\P)!0.67!(\Q)$) {};
				\draw[edgedof]
				let \p1=(\P), \p2=(\Q) in
				($(\P)!0.33!(\Q)$)
				-- ++({(\y2-\y1)*0.35},{-(\x2-\x1)*0.35});
				\draw[edgedof]
				let \p1=(\P), \p2=(\Q) in
				($(\P)!0.67!(\Q)$)
				-- ++({(\y2-\y1)*0.35},{-(\x2-\x1)*0.35});
			}
		\end{scope}
		
	\begin{scope}[xshift=6cm,yshift=2cm]
		\Pentagon
		
		\node[interiordof] at (1.05,1.3) {};
		
		\foreach \P/\Q in {A/B,B/C,C/D,D/E,E/A}{
			
			\node[edgedofpoint] at ($(\P)!0.25!(\Q)$) {};
			\node[edgedofpoint] at ($(\P)!0.5!(\Q)$) {};
			\node[edgedofpoint] at ($(\P)!0.75!(\Q)$) {};
			
			\foreach \t in {0.25,0.5,0.75}{
				\draw[edgedof]
				let \p1=(\P), \p2=(\Q) in
				($(\P)!\t!(\Q)$)
				-- ++({(\y2-\y1)*0.35},{-(\x2-\x1)*0.35});
			}
		}
	\end{scope}	
		
		
		\begin{scope}[xshift=-2cm,yshift=-2cm]
			\Pentagon
			\foreach \P/\Q in {A/B,B/C,C/D,D/E,E/A}{
				\draw[edgedof]
				let \p1=(\P), \p2=(\Q) in
				($(\P)!0.5!(\Q)$)
				-- ++({(\y2-\y1)*0.3},{-(\x2-\x1)*0.3});
			}
		\end{scope}
		
\begin{scope}[xshift=2cm,yshift=-2cm]
	\Pentagon
	\foreach \P/\Q in {A/B,B/C,C/D,D/E,E/A}{
		
		\node[mid] at ($(\P)!0.5!(\Q)$) {};
		
		\draw[edgedof]
		let \p1=(\P), \p2=(\Q) in
		($(\P)!0.33!(\Q)$)
		-- ++({(\y2-\y1)*0.3},{-(\x2-\x1)*0.3});
		
		\draw[edgedof]
		let \p1=(\P), \p2=(\Q) in
		($(\P)!0.67!(\Q)$)
		-- ++({(\y2-\y1)*0.3},{-(\x2-\x1)*0.3});
	}
\end{scope}

\begin{scope}[xshift=6cm,yshift=-2cm]
	\Pentagon
	
\node[interiordof] at (1.05,1.2) {};
	
	\foreach \P/\Q in {A/B,B/C,C/D,D/E,E/A}{
		
		\node[mid] at ($(\P)!0.33!(\Q)$) {};
		\node[mid] at ($(\P)!0.67!(\Q)$) {};
		
		\foreach \t in {0.25,0.5,0.75}{
			\draw[edgedof]
			let \p1=(\P), \p2=(\Q) in
			($(\P)!\t!(\Q)$)
			-- ++({(\y2-\y1)*0.3},{-(\x2-\x1)*0.3});
		}
	}
\end{scope}
\end{tikzpicture}
	\caption{DoFs of the $C^0$-nonconforming (top) and Morley-type (bottom) VEMs on a pentagonal element for $k=2,3,4$. Dots, squares, arrows, and diamonds correspond to $\mathbf{(D1)}$, $\mathbf{(D2)}$ ($i=0,1$), and $\mathbf{(D3)}$, respectively.}\label{fig:Dofs:pentagon}
\end{figure}
\end{example}

\subsection{Polynomial projections and the discrete multilinear forms}\label{polynomialproy}

Now, we will recall the definition of some polynomial projections,  
which  will be useful to build  explicit representations  for the bilinear forms 
$M_{h}(\cdot,\cdot)$, $A_{h}(\cdot,\cdot)$, and $B_h(\cdot,\cdot)$ along with the 
discrete functional $F_h(\cdot; \cdot)$ satisfying the Assumptions ${\bf (A3)}-{\bf (A5)}$.

We start with the definition of  the usual local $L^2$-projections onto the polynomial scalar space $\P_{n}(\E)$ and its vectorial version, with $n \in \mathbb{N} \cup \{0\}$:
\begin{itemize}
	\item {\bf $L^2$-projection} $\Pi_{\E}^{n}: L^2(\E) \to \mathbb{P}_n(\E)$, given by 
	\begin{equation}\label{ProyL2Pk}
		( q_{n}, v-\Pi_{\E}^{n} v)_{\E}  =0 
		\qquad 	\forall  v \in L^2(\E) \quad \text{and} 	\qquad\forall q_{n}\in\P_{n}(\E).
	\end{equation}
	\item  {\bf $\LL^2$-projection} $\boldsymbol{\Pi}_{\E}^{n}: \LL^2(\E) \to \PP_n(\E)$, given by 
	\begin{equation}\label{ProyL2Pk-vectorial}
		( \boldsymbol{q}_{n}, \bv-\boldsymbol{\Pi}_{\E}^{n} \bv)_{\E}  =\0 
		\qquad 	\forall  \bv \in \LL^2(\E) \quad \text{and} 	\qquad\forall  \boldsymbol{q}_{n}\in\PP_{n}(\E).
	\end{equation}
\end{itemize}
Now, we continue with the construction of  the classical $H^1$- and $H^2$-VEM projections, constructed for each $k \geq 2$:
\begin{itemize}
	\item the {\bf $H^1$-seminorm projection} $\PinablaK: H^1(\E) \to \mathbb{P}_k(\E)$, given by 	
	\begin{equation}\label{ProyH1}
		\left\{ 	\begin{array}{ll}
			B^{\E}(v-\PinablaK v,p_{k})=0\qquad \forall  v \in H^1(\E) \quad \text{and} \quad \forall p_{k} \in \mathbb{P}_{k}(\E),\\[2ex]
			\displaystyle	 P_0 (v- \PinablaK v) =0.
		\end{array} \right.
	\end{equation}
	
	\item {\bf $H^2$-seminorm projection} $\PiK: H^2(\E) \to \mathbb{P}_k(\E)$, given by 	
	\begin{equation}\label{H2:VEM:operator}
		\left\{ 	\begin{array}{ll}
			A^{\E}( v-\PiK v, p_k ) \qquad \forall  v \in H^2(\E) \quad \text{and} \quad \forall p_k \in \mathbb{P}_k(\E),\\[2ex]
			\displaystyle P_0 (v- \PiK v) =0,  \qquad 
			\displaystyle 	 P_0 \big(\nabla(v- \PiK v)\big) =\0,
		\end{array} \right.
	\end{equation}
	where the projection  $P_0(\cdot)$ is defined by 
	\begin{equation*}
		P_0(v_h) = (v_h,1)_{\partial \E}.
	\end{equation*}
\end{itemize}

We also consider the global counterpart of the $L^2$-projection~\eqref{ProyL2Pk}, $\Pi_{h}^{k}: L^2(\O) \to \mathbb{P}_k(\O_h)$, defined for all $\E \in \O_h$ by 
\[
\big(\Pi_h^{k} v \big)|_{\E}  = \Pi^{k}_{\E} v.
\]
\begin{assumption}[Computability of the polynomial projections]\label{assump:projections}
	For each $v_h \in V_k^h(\E)$ we have that the polynomial functions $\PioKk v_h, \PimunoK\nabla v_h$, $\PinablaK v_h$ and $\PiK v_h$ defined in \eqref{ProyL2Pk}, \eqref{ProyL2Pk-vectorial}, \eqref{ProyH1} and \eqref{H2:VEM:operator}, respectively,  
	are computable using only the information of the DoFs $({\bf D1})-({\bf D3})$.
\end{assumption}
We recall that the space $H_k^{2,\nc}(\O_h)$  is the $H^2$-nonconforming space defined \eqref{H2:nonconf}  (and  used in Assumption ${\bf A2}$).
Next, we  present the following extension of Definition~\ref{def:local:dofs} to get  the global DoFs.
\begin{definition}[Associated global DoFs]\label{def:global:dofs}
	For a function $v_h \in H_k^{2,\nc}(\O_h)$, the global DoFs $\Lambda_{\boldsymbol{\xi}}$ associated to the DoFs-tuple $\boldsymbol{\xi}$ defined in \eqref{dof:tuple} are given by the following set of linear operators:
	\begin{itemize}
		\item[${\bf (D1_G)}$] The values of $ v_h$ at each vertex $\vb \in \VVi$.
		\item[${\bf (D2_G)}$] The moments of $\partial_{\bn_e}^iv_h$ up to order $d^e_i$ on each $e \in \EEi $, for $i =0,1,$
		\begin{align*}
			h_{e}^{-1+i}  (\partial_{\bn_e}^{i} v_h,p)_{e} \quad \forall p \in \M_{d^e_i}(e).
		\end{align*} 
		\item[${\bf (D3_G)}$] The  moments of $v_h$ up to order $\innerDofs$ inside every $\E \in \O_h$,
		\begin{align*}
			h^{-2}_{\E} ( v_h, p)_{\E} \quad \forall p \in \M_{\innerDofs}(\E).
		\end{align*}
		\end{itemize}		
\end{definition}
\begin{remark}\label{remark:Dofs:BCs}
We observe that the global DoFs defined above are unisolvent in the global space $\Vh$. This fact follows from the unisolvency of the local DoFs (cf. Assumption~\ref{assump:unisolvency}) and the definition of the local spaces.
Moreover, depending on the imposed boundary conditions (cf. \eqref{Campled:BCs}, \eqref{supported:BCs}, and \eqref{Cahn-Hilliard:BCs}), the global boundary-associated DoFs, as expected, must be removed to enforce the corresponding constraints. For simplicity, we retain the notation for the global VE space $\Vh$ (cf. \eqref{global:VEM}), although, strictly speaking, the discrete space used in the computations may differ slightly due to the application of the boundary conditions.
\end{remark}

On the other hand, for the sake of clarity, we present in this section explicit representations of the local multilinear forms involved in the discretization, under Assumption~\ref{assump:projections}.

\begin{definition}\label{disc:multiforms}
	For every $ v_h, w_h\in \VK$,  define the local discrete multilinear forms $M^{\E}_{h}(\cdot,\cdot), A^{\E}_{h}(\cdot,\cdot)$, $B^{\E}_{h}(\cdot,\cdot)$ and $F^{\E}_{h}(\cdot;\cdot)$ as 
	\begin{align}
		M^{\E}_{h}(v_h,w_h) &:=  (\PioKk v_h,  \PioKk w_h)_{0,\E}   + S_{0}^{\E}\big(({\rm I} -\PioKk)v_h,({\rm I} -\PioKk)w_h\big),\label{local:formM}\\
		A^{\E}_{h}(v_h,w_h) &:=  A^{\E}(\PiK v_h,\PiK w_h) 
		+ S_{\D^2}^{\E}\big(({\rm I} -\PiK)v_h,({\rm I} -\PiK)w_h\big),\label{local:formA}\\
		B^{\E}_{h}(v_h,w_h) &:=  (\PimunoK\nabla v_h,\PimunoK\nabla w_h)_{0,\E}   
		 + S_{\nabla}^{\E}\big(({\rm I} -\PinablaK)v_h,({\rm I} -\PinablaK)w_h\big),\label{local:formB}\\
		F^{\E}_{h}(v_h;w_h) &:=  (\PioKk f(\PioKk v_h),  w_h)_{0,\E}. \label{local:formF}
	\end{align} 	
	
	We will take the \emph{stabilization terms} $S_{0}^{\E}(\cdot,\cdot)$, $S_{\nabla}^{\E}(\cdot,\cdot)$ and  $S_{\D^2}^{\E}(\cdot,\cdot)$  to be a symmetric, positive definite bilinear forms satisfying
\begin{equation}\label{eqn:formsStabi:boundedM}
	\begin{alignedat}{2}
		M^{\E}(v_h,v_h) &\approx S_{0}^{\E}(v_h,v_h)
		&\quad& \forall v_h \in \ker(\PioKk), \\
		A^{\E}(v_h,v_h) &\approx S_{\D^2}^{\E}(v_h,v_h)
		&& \forall v_h \in \ker(\PiK),\\
		B^{\E}(v_h,v_h) &\approx S_{\nabla}^{\E}(v_h,v_h)
		&& \forall v_h \in \ker(\PinablaK),
	\end{alignedat}
\end{equation}
	where all the involved constant are positive and independent of $h_{\E}$.
\end{definition}

In this work, we adopt the standard choices for the stabilization bilinear forms $S_{0}^{\E}(\cdot,\cdot)$, $S_{\nabla}^{\E}(\cdot,\cdot)$, and $S_{\D^2}^{\E}(\cdot,\cdot)$, which satisfy the stability condition~\eqref{eqn:formsStabi:boundedM} and are based on suitably scaled scalar products of the local DoFs. More precisely, for the sake of completeness, we consider the following definitions.
\[
S^{\E}(v_h,w_h):= \sum_{i=1}^{d_k} \dof_i(v_h) \dof_i(w_h) 
\quad \qquad \forall v_h,w_h \in \VK.
\]

Thus, for all $v_h,w_h \in \VK$, we choose the following computable representations (see for instance~\cite{Beirao2013_M3AS,Antonietti2018_M3AS}):
\begin{align*}
	S_{0}^{\E}(v_h,w_h)&:= h_{K}^2 	S^{\E}(v_h,w_h), \quad
	S_{\nabla}^{\E}(v_h,w_h):=S^{\E}(v_h,w_h)
	\quad \text{and} \quad 	S_{\D^2}^{\E}(v_h,w_h):= h^{-2}_{K} S^{\E}(v_h,w_h).
\end{align*} 

\begin{remark}
We note that the projection operators $\PioKk$,  $\PimunoK\nabla$, $\PinablaK$ and $\PiK$ may alternatively be defined using the constrained least-squares approach introduced in~\cite{Dedner2022_IMA}. For the sake of clarity and ease of exposition, however, we restrict our attention to the standard construction presented above.
\end{remark} 

For the continuous bilinear form $M(\cdot,\cdot)$,  we adopt the following notation for all $v_h, w_h \in V+\Vh$:
\begin{equation}\label{notation:broken}
	M(v_h, w_h) := \sum_{\E \in \O_h} M^{\E}(v_h, w_h).
\end{equation}
Moreover, we adopt the same notation also for all other continuous forms.

We finish this subsection summarizing important properties of the discrete bilinear forms.	The proof follow directly from the definition of the forms and standard VEM arguments. Interested reader can look into \cite{ALKM2016,Antonietti2018_M3AS,CMS2016} for better understanding.
\begin{proposition}\label{prop:miltilinear}
	The global bilinear forms $M_{h}(\cdot,\cdot)$, $A_{h}(\cdot,\cdot)$, $B_h(\cdot,\cdot)$ constructed by employing \eqref{general:discrete:Mh}-\eqref{general:discrete:Bh} in Assumption ${\bf (A3)}$ and Definition~\ref{disc:multiforms} satisfy the following properties:
	\begin{align*}
		|M_h(v_h,w_h)|  & \leq C_{M}^* \|v_h\|_{0,\O} \|w_h\|_{0,\O}\qquad \text{and} \qquad	
		M^h(v_h,v_h) \geq C_{M,*} \|v_h\|_{0,\O}^2,\\
		|A_h(v_h,w_h)|  & \leq C_{A}^* |v_h|_{2,h} |w_h|_{2,h}\qquad \	\quad \text{and} \qquad	A_h(v_h,v_h)\geq C_{A,*}  |v_h|_{2,h}^2,\\
		|B_h(v_h,w_h)|  & \leq C_{B}^* |v_h|_{1,h} |w_h|_{1,h}\qquad \quad \  \text{and} \qquad	B_h(v_h,v_h) \geq C_{B,*} |v_h|_{1,h}^2.
	\end{align*}
	All the involved constants are positive and independent of the mesh discretization parameter $h$.
\end{proposition}
%
\subsection{The discrete formulations and their well-posedness}\label{SUBSECTION:DISCRETE:SCHEMES}
In this section we present the discrete virtual element formulations (the semi- and fully-discrete problems) and we provide their well-posedness in an abstract framework.

\subsubsection{The semi-discrete formulation}
The semi-discrete nonconforming-VE formulation for the time dependent semi-linear biharmonic problem~\eqref{conti:weak:form},  reads as: 	seek 
$$u_h \in L^2(0,T; \Vh),$$
such that for a.e. $t \in  (0,T)$
\begin{equation}\label{semi_discrete:scheme}
\left\{
\begin{aligned}	
M_h(\partial_t u_h(t),v_h)+ \alpha_1 A_h(u_h(t),v_h )+ \alpha_2 B_h(u_h(t), v_h)&= F_h(u_h;v_h), \quad \forall v_h \in \Vh,\\
u_h(0)&= u_{h,0},
\end{aligned}
\right.
\end{equation}	
where  $u_{h,0}$ is an adequate approximation of $u_0$ in the space $\Vh$. In fact, we set $u_{h,0}=u_I(0)$, where $u_I(0)$ is a suitable interpolation of $u_0$ in $\Vh$ (see below Assumption~\ref{virtual:interp:result}).  
 
Let us assume that \(\mathbf{A}, \mathbf{B}, \mathbf{M}\) be the matrix representation corresponding to the discrete forms \( A_h(\cdot, \cdot), B_h(\cdot, \cdot) \), and $M_h(\cdot,\cdot)$, respectively. Therefore, problem \eqref{semi_discrete:scheme} reduces to a system of nonlinear differential equations as follows
\begin{equation}
	\begin{split}
		&\mathbf{M} \frac{d\boldsymbol{U}_h}{dt} + (\alpha_1 \mathbf{A}+ \alpha_2 \mathbf{B}) \boldsymbol{U}_h  = \mathbf{F}(\boldsymbol{U}_h) \\
		&\boldsymbol{U}_h(0) = \boldsymbol{U}_0, 
	\end{split}
	\label{mat:semidiscrete}
\end{equation}
where \(\boldsymbol{U}_h\) denotes the vector whose entries are the components in the basis of \(u_h \in V_k^h\). Moreover, \(\mathbf{F}(\boldsymbol{U}_h)\) is the matrix corresponding to the nonlinear term and \(\mathbf{F}\) be the right hand side load vector corresponding to the basis \(u_h\). 

Now, we emphasize that the nonlinear term, i.e., \( \mathbf{F}(\boldsymbol{U}_h) \) satisfies Lipschitz's continuity condition and matrices \(\mathbf{A}, \mathbf{B}, \mathbf{M}\) are positive definite and hence inverse exist.  Therefore,
we can use the well-known Picard Theorem to have that \eqref{mat:semidiscrete} has a unique solution.

\subsubsection{The fully-discrete formulation}
In order to present the fully-discrete scheme of problem~\eqref{semi_discrete:scheme}, we introduce the following preliminaries notations. The interval $[0,T]$ is decomposed into subintervals $I_n :=[t_{n-1},t_n]$, where the sequence of time steps  $t_n = n \Delta t$, $n=0,1,2,\ldots,N$, and $\Delta t=T/N$ is the length of time step. The solution of semi-discrete and fully-discrete schemes at time $t = t_n$ will be denoted by $u_h(t)$ and $u_h^n$, respectively. 

For any generic function $g$, the approximation of the time derivative at $t_n$ 
is defined by
\begin{equation*}
	\delta_t g^n := \frac{g(t_n)-g(t_{n-1})}{\Delta t} \equiv  \frac{\: g^n-g^{n-1} \:}{\Delta t}.	
\end{equation*}

Moreover,  given a Hilbert space $V$ endowed with the norm $\|\cdot\|_{V}$, we consider the following discrete-in-time norms
\begin{equation}\label{def:norm:time}
\|g\|_{\ell^{2}(0,T;V)}
:=  \Big(\Delta t \sum_{n=1}^{N} \|g^{n}\|_{V}^{2} \Big)^{1/2},
\qquad
\|g\|_{\ell^{\infty}(0,T;V)}
:= \max_{0 \le n \le N} \|g^{n}\|_{V}.
\end{equation}

Now, with the above preliminaries we consider the backward Euler method coupled with the semi-discrete VE 
discretization presented \eqref{semi_discrete:scheme}, which read as: given  $u_h^0$, seek
$$\{u_h^n\}_{n=1}^{N} \subset \Vh,$$
such that
\begin{equation}\label{fully:dis:schm}
M_h(\delta_t u^n_h, v_h)+ \alpha_1 A_h(u_h^n,v_h)+\alpha_2 B_h(u_h^n, v_h)=F_h(u^n_h;v_h) \quad \forall v_h \in \Vh.
\end{equation}

For existence of discrete solutions to the fully-discrete scheme~\eqref{fully:dis:schm}, we  recall the following result, which is a consequence of the Brouwer fixed point theorem (see for instance~\cite{GR}).
\begin{lemma}\label{lemma:Brouwer:mod}
	Let $H$ be a finite dimensional Hilbert space endow with  inner product $\langle\cdot,\cdot\rangle_H$ and norm $\Vert \cdot \Vert_{H}$, and let $\mathcal{L}: H \to H$ be a continuous mapping. Furthermore, assume that there exists $\rho>0$,
	such that $\langle\mathcal{L}(\zeta),\zeta\rangle_H>0$, for all  $\zeta \in H $ with $\Vert \zeta\Vert_{H} = \rho >0$. Then, there exists $\zeta^{\ast} \in H$ such that $\mathcal{L}(\zeta^{\ast}) =0,$ with  $\Vert \zeta^{\ast}\Vert_{H} \leq \rho$.
\end{lemma}

The next result provide the well-posedness of the fully-discrete formulation~\eqref{fully:dis:schm}. 
\begin{theorem}\label{theorem:well-posed:fully}
	Let be $\widehat{M} :=\max\{C_{M,*},(C^{*}_M)^2 C_{M,*}^{-1}\}$, where $C_{M,*}$ and $C_M^{*}$ are the constants in Lemma~\ref{prop:miltilinear}. Assume that
	\begin{equation}\label{asump:Deltat}
		\Delta t  < (2 L_{f})^{-1} \min\{ \widehat{M},C_{M,*}\},
	\end{equation}	
where $L_f$ is the Lipschitz constant in \eqref{ineq:Lipt:f}, and that $u^j_h$ for $j=0,1, \ldots, n-1$ are given. Then, under condition \eqref{asump:Deltat}, there exists a unique solution $u^n_h \in \Vh$  to the discrete problem \eqref{fully:dis:schm}. 
\end{theorem}
\begin{proof}
We start by demonstrating existence of the solution. For this propose,  for any $u_h \in \Vh$,  we consider the operator  $\mathcal{L}: \Vh \to (\Vh)^{\ast}$, defined by 
\begin{equation*}
	\begin{split}
		\langle \mathcal{L}(u_h), v_h  \rangle 
		&:= 	M_h(u_h,v_h)+ \Delta t \big( \alpha_1 A_h(u_h,v_h)+  \alpha_2 B_h(u_h,v_h) 
		+F_h(u_h;v_h)\big)- M_h(u^{n-1}_h,v_h) \quad \forall v_h \in \Vh. 	
	\end{split}
\end{equation*}
For $h$ and $\Delta t$ fixed, we have that the operator $\mathcal{L}$ is continuous.

On the other hand, by employing again Lemma~\ref{prop:miltilinear} and Young's inequality, for all $u_h \in \Vh$, we obtain
\begin{equation}\label{eq1:operator}
	\begin{split}
		\langle \mathcal{L}(u_h), u_h  \rangle 
		&\geq  C_{M,*}\|u_h\|^2_{0,\O}- C_M^*\|u^{n-1}_h\|_{0,\O} \|u_h\|_{0,\O}+  \Delta t \big( \alpha_1 C_{A,*}|u_h|^2_{2,h}+  \alpha_2 C_{B,*}|u_h|^2_{1,h} + F_h(u_h;v_h)\big)\\
		& \geq C_{M,*}\|u_h\|^2_{0,\O}- \frac{C^{*2}_M}{2 C_{M,*}}\|u^{n-1}_h\|^2_{0,\O}- \frac{ C_{M,*} }{2} \|u_h\|^2_{0,\O}
		+ \Delta t  F_h(u_h;v_h)\\
		& \geq \frac{C_{M,*}}{2}\|u_h\|^2_{0,\O}- \frac{C^{*2}_M}{2 C_{M,*}}\|u^{n-1}_h\|^2_{0,\O}  + \Delta t  F_h(u_h;v_h).
	\end{split}
\end{equation}
Now, by using the Lipschitz continuity condition~\eqref{ineq:Lipt:f}, and the continuity of the projector $\PioKk$, we have 
\begin{equation}\label{eq2:operator}
	\begin{split}
		|F_h(u_h;v_h)|	&=|(f(\Pi^{k} u_h)-f(0),\Pi^{k} u_h)_{0,\O}+ (f(0),\Pi^{k} u_h)_{0,\O}| \\
		&\leq \sum_{\E \in \O_h} L_{f} \|\PioKk u_h\|_{0,\E} \|\PioKk u_h\|_{0,\E} + \|f(0)\|_{0,\E}\|\PioKk u_h\|_{0,\E} \\
		&\leq L_{f} \|u_h\|^2_{0,\O}  + \|f(0)\|_{0,\O}\| u_h\|_{0,\O} \\
		&\leq  L_{f} \|u_h\|^2_{0,\O}  + \frac{ C_{N}}{2 \alpha_1 C_{A,*}}\|f(0)\|^2_{0,\O}+\frac{\alpha_1 C_{A,*}}{2}|| u_h||^2_{0,\O},
	\end{split}
\end{equation}
where we have applied again Young's inequality. Thus, by combining estimates \eqref{eq1:operator} and \eqref{eq2:operator}, we derive 
\begin{equation*}
	\begin{split}
		\langle \mathcal{L}(u_h), u_h  \rangle 
		& \geq  \Big(\frac{C_{M,*}}{2}- L_{f} \Delta t  \Big)\|u_h\|^2_{0,\O}- \frac{C^{*2}_M}{2 C_{M,*}}\|u^{n-1}_h\|^2_{0,\O}  -\frac{C_{N} \Delta t }{2 \alpha_1 C_{A,*}}\|f(0)\|^2_{0,\O}\\
		& \geq  \frac{1}{2}\left(C_{M,*}- 2 L_{f} \Delta t  \right)\|u_h\|^2_{0,\O}- \frac{C^{*2}_M}{2 C_{M,*}}\|u^{n-1}_h\|^2_{0,\O} -\frac{C_{N} \Delta t }{2 \alpha_1 C_{A,*}}\|f(0)\|^2_{0,\O},
	\end{split}
\end{equation*}
where we  have also used that  $\frac{\alpha_1 C_{A,*} \Delta t}{2}>0$.
Then, from~\eqref{asump:Deltat} we  consider the radius $\rho$, given by
\begin{equation*}
	\rho := \Big(C_{M,*}- 2 L_{f} \Delta t  \Big)^{-1/2}\Big(\frac{C^{*,2}_M}{ C_{M,*}}\|u^{n-1}_h\|^2_{0,\O} +\frac{ C_{N}C_{M,*}}{2L_f C_{A,*}}\|f(0)\|^2_{0,\O}  \Big)^{1/2}.
\end{equation*}

Therefore, we have the hypothesis of Lemma~\ref{lemma:Brouwer:mod} and then there exists $u^n_h \in  \Vh$ satisfying $\mathcal{L}(u^n_h) =0$, i.e., the fully-discrete formulation~\eqref{fully:dis:schm} admits at least one solution at every time step $t_n$.

For the uniqueness, let $u^n_{1h}$ and $u^n_{2h}$ be two distinct solutions of system~\eqref{fully:dis:schm}.
Setting $w^n_h  := u^n_{1h}-u^n_{2h}$, we have
\begin{equation*}
	M_h(\delta_t w^n_h, v^n_h)+ \alpha_1 A_h(w_h^n,v_h)+\alpha_2 B_h(w_h^n, v_h)=F_h(u^n_{1h};v_h)- F_h(u^n_{2h};v_h) \quad
	\forall v_h \in \Vh.	
\end{equation*}

Then, taking $w^n_h$ in the above identity as the test function we obtain  
\begin{equation}\label{to:comb1}
	M_h(\delta_t w^n_h, w^n_h)+ \alpha_1 A_h(w_h^n,w^n_h)+\alpha_2 B_h(w_h^n, w^n_h)=(f(\Pi^{k} u^n_{1h})-f(\Pi^{k} u^n_{2h}),\Pi^{k} w^n_h)_{0,\O}.	
\end{equation}

Now, by employing the property \eqref{ineq:Lipt:f} and the continuity of the $L^2$-projector $\PioKk$, we have 
\begin{equation}\label{to:comb2}
	\begin{split}
		|(f(\Pi^{k} u^n_{1h}-f(\Pi^{k} u^n_{2h}),\Pi^{k} w^n_h)_{0,\O}& \leq \sum_{E \in \O_h} L_f\|\PioKk u^n_{1h}-\PioKk u^n_{2h}\|_{0,\E} \|\PioKk w^n_h\|_{0,\E}   
		\leq L_f\|w^n_h\|^2_{0,\O}.   	
	\end{split}
\end{equation}

Moreover, from Lemma~\ref{prop:miltilinear} and the Young inequality, we have
\begin{equation}\label{to:comb3}
	\begin{split}
		M_h(\delta_t w^n_h, w^n_h)	& =  \frac{1}{\Delta t }\big( M_h( w^n_h, w^n_h) -M_h(w^{n-1}_h, w^n_h) \big) \geq \frac{1}{\Delta t }\big(C_{M,*} \| w^n_h\|^2_{0,\O}  - C_M^* \| w^n_h\|_{0,\O} \| w^{n-1}_h\|_{0,\O} \big)\\
		& \geq \frac{1}{2\Delta t }\max\{C_{M,*}, C^{*,2}_M C^{-1}_{M,*}\} \big( \| w^n_h\|^2_{0,\O} -\| w^{n-1}_h\|^2_{0,\O}\big)= \frac{\widehat{M}}{2 }	 \delta_t \| w^n_h\|^2_{0,\O}.
	\end{split}
\end{equation}
Thus, by combining estimates~\eqref{to:comb1}, \eqref{to:comb2}, \eqref{to:comb3} and Lemma~\eqref{prop:miltilinear} 
we obtain
\[
\frac{\widehat{M}}{2 \Delta t} \big( \| w^n_h\|^2_{0,\O} -\| w^{n-1}_h\|^2_{0,\O}\big) - L_f \| w^n_h\|^2_{0,\O} \leq 0,
\]
which implies 
$$\Big(1-  \frac{2L_f}{\widehat{M}} \Delta t\Big) \| w^n_h\|^2_{0,\O}  \leq \| w^{n-1}_h\|^2_{0,\O}.
$$

Assuming $w^{n-1}_h = 0$, the above inequality and assumption~\eqref{asump:Deltat}, implies $w^{n}_h = 0$. This complete the rest of the proof.

\end{proof}

We finish this section with the following remark.
\begin{remark}
	We recall that, in the VEM framework, the construction of a computable approximation $F_h(u_h;v_h)$ for the nonlinear term $F(u;v)$ involves the operator $\PioKk$  (cf. \eqref{local:formF}), distinguishing it from the traditional FEM approach. Nevertheless, in Theorem~\ref{theorem:well-posed:fully}, we successfully established the existence and uniqueness of the problem under the classical assumption of small values for the time step  $\Delta t $, which is also a standard consideration in the FEM context.	
\end{remark}

\setcounter{equation}{0}
\setcounter{equation}{0}	
\section{Main results: unified error analysis, companion and Ritz-type operators}\label{SECTION:ERROR:UNIFIED}

In this section we will develop a new unified error analysis for our NCVEMs in an abstract framework. More precisely, we shall present the construction of a novel Ritz-type projection, which is built from a Companion operator and allow us derive optimal error bounds for the broken $H^i$-norms (with $i=1,2$) under minimal regularity condition of the solution in space (cf. Remark~\ref{remark:Ritz:min:regu}).

For this new error analysis, we will first outline the following standard assumptions concerning the polygonal mesh \cite{Beirao2013_M3AS}.
\begin{assumption}[Mesh regularity]\label{mesh:regularity} There exists a  real number $\rho>0$ such that,
	for every $h$ and every $\E\in \O_h$:
	\item[${\bf (A6)}$] $\E\in\O_h$ is star-shaped with respect to every point of a  ball of radius $\rho h_\E$;
	\item[${\bf (A7)}$] the ratio between the shortest edge and the diameter $h_\E$ of $\E$ is larger than $\rho$.
\end{assumption}

\subsection{Some useful results}
In this subsection we shall present some results which will be useful for the forthcoming sections~\cite{Huang2021_JCAM}. 

Throughout this work, $C$ denotes a generic positive constant, independent of the mesh size $h$ and
the time-step size $\Delta t$ (cf.\ Subsection~4.2), and possibly depending on the domain $\Omega$, the
final time $T$, and the polynomial order $k,$ of the method. The notation $\lesssim$ is used to
denote inequalities up to such a constant. The same convention applies to quantities with subscripts,
superscripts, tildes, bars, or hats.

\begin{lemma}[Scaled inequalities]\label{trace-scaling}
	For all $\varepsilon >0$, there exist positive constants $C_1, C_2$ and $C_{\varepsilon}$, independent of $h_{\E}$, such that
	\begin{equation*}
		\begin{split}
			\|v\|_{0,\partial \E} &\leq C_1(\varepsilon h^{1/2}_{\E}|v|_{1,\E} +C_{\varepsilon} h^{-1/2}_{\E}\|v\|_{0,\E}) \qquad \: \: \forall v \in H^1(\E),\\
			|v|_{1,\E} &\leq C_2(\varepsilon h_{\E}|v|_{2,\E} +C_{\varepsilon} h^{-1}_{\E}\|v\|_{0,\E}) \qquad \quad \: \: \: \forall v \in H^2(\E).
		\end{split}
	\end{equation*}
\end{lemma}
\begin{lemma}[Integral average]\label{lemma:average:edge}
	The projection $\widetilde{\Pi}^0_{e}: H^1(\E) \to \P_{0}(e)$ defined by the following average
	$\widetilde{\Pi}^0_{e} v := h^{-1}_e  (v,1)_{e}$, satisfies 
	\begin{equation*}
		\|v - \widetilde{\Pi}^0_{e} v\|^2_{0,e} \lesssim h_{\E}|v|^2_{1,\E} \qquad  \forall v \in H^1(\E).
	\end{equation*}
\end{lemma}

We consider the following polynomial approximation result.

\begin{lemma}[Polynomial approximation]\label{polynomial:interp:result}
	Under assumption ${\bf (A6)}$ for every $v \in H^{2+s}(\E)$, with $0 \leq s \leq k-1$, 
	there exists $v_{\pi} \in \mathbb{P}_k(\E)$, independent of $h$, such that
	\begin{equation*}
		|v-v_{\pi}|_{\ell,\E} \lesssim h^{2+s-\ell} |v|_{2+s,\E}, \quad \ell=0,1,2. 
	\end{equation*}
\end{lemma}

Next, we assume the following error interpolation result.
\begin{assumption}[Existence of an virtual interpolation operator]\label{virtual:interp:result}
	Under assumptions ${\bf (A6)-(A7)}$ for every $v \in H^{2+s}(\E)$, with $0 \leq s \leq k-1$, there exists $v_I \in \VK$, independent of $h$, such that 
	\begin{equation*}
		\begin{split}
			&dof_i (v-v_I)=0 \qquad  \text{for all $ 1 \leq i \leq d_k$, and }\\
			&| v -v_I|_{\ell,\E} \lesssim h^{2+s-\ell}|v|_{2+s,\E}, \quad \ell=0,1,2.	
		\end{split}	
	\end{equation*}
\end{assumption}

We have the following result, which will be useful to obtain an optimal estimate in the broken $H^1$-norm
for the  Ritz-type projection.

\begin{lemma}\label{lemma:extra:power}
	Let $ \chi \in H^{2+s}(\O) \cap V$, with $0 < s \leq k-1$. For each $v \in H^{2+s_0}(\O)$ with $s_0 \in [0,1]$, it holds
	\begin{equation*}
		A(v,\chi-\chi_I) \lesssim  h^{s+s_0}\|v\|_{2+s_0,\O} \|\chi\|_{2+s,\O}.
	\end{equation*}
	where $\chi_I \in \Vh$ is the interpolant of $\chi$ in the virtual space $\Vh$ (cf. Assumption~\ref{virtual:interp:result}).	
\end{lemma}
\begin{proof}
This result can be seen as a extension of \cite[Lemma 4.10]{Adak2024_CMAME} to high-order cases and different boundary conditions. For sake of completeness, we will provide the sketch of its proof. Indeed, let  $\chi \in H^{2+s}(\O) \cap V$, with $0 < s \leq k-1$.  For $v \in H^2(\O)$, we can apply the continuity of $A(\cdot,\cdot)$ to obtain
\begin{equation}\label{interpolation:uno}
	\begin{split}
		A(v,\chi-\chi_I) &= \sum_{\E \in\O_h} A^{\E}(v,\chi-\chi_I)
		\lesssim  h^s\|v\|_{2,\O}\|\chi\|_{2+s,\O}.
	\end{split}
\end{equation}
Now, for $v \in H^{3}(\O)$, we apply integration by parts to obtain 
\begin{equation}\label{itp1:dos}
	\begin{split}
		A(v,\chi-\chi_I) & 
		=- \sum_{\E \in \O_h}\int_{\E} \nabla(\Delta v) \cdot \nabla(\chi-\chi_I) 
		+\sum_{\E \in \O_h} \int_{\partial \E} \left(  \Delta v - \frac{\partial^2 v}{\partial \bt^2_{\E} } \right) \frac{\partial(\chi-\chi_I)}{\partial \bn_{\E}}\\		
		&\qquad+ \sum_{\E \in \O_h}\int_{\partial \E} \frac{\partial^2 v}{\partial \bn_{\E} \partial \bt_{\E}} \frac{\partial(\chi-\chi_I)}{\partial \bt_{\E} }=: T_1 + T_2+ T_3.
	\end{split}
\end{equation}
For the term $T_1$, we use the Cauchy-Schwarz inequality (for sequences) 
\begin{equation}\label{T1:dos}
	\begin{split}
		T_1 
		&					\lesssim \|v\|_{3,\O}|\chi-\chi_I|_{1,h} \lesssim h^{s+1}\|v\|_{3,\O} \|\chi\|_{2+s,\O}.
	\end{split}
\end{equation} 
Now, we will bound the terms $T_2$ and $T_3$. For convenience, we set  $\zeta_2=  \left(  \Delta v - \frac{\partial^2 v}{\partial \bt^2_{e} } \right) $ and $\zeta_3= \frac{\partial^2 v}{\partial \bn_{e} \partial \bt_{e}}$. 

By using the fact that  $\chi  \in C^1(\bar{\O}) \cap V$ and since $\chi_I \in \Vh \subset H^{2,\nc}(\O_h)$, we have:
\begin{equation}\label{int-normal}
	\int_e p_0\jump{\nabla(\chi-\chi_I)\cdot \bn_{e}} =0 \qquad \forall p_0 \in \P_{0}(e) \subset \P_{k-2}(e).
\end{equation} 
Then, for $T_2$,  from \eqref{int-normal}, with $p_0 =\widetilde{\Pi}_e^0 \zeta_2 \in \P_{0}(e)$ 
(cf. Lemma~\ref{lemma:average:edge}), we obtain
\begin{equation}\label{T2}
	\begin{split}
		T_2 
		& = \sum_{e \in \EE_h} \int_e (\zeta_2-\widetilde{\Pi}_e^0 \zeta_2) \jump{\nabla(\chi-\chi_I)\cdot \bn_{e}}	 
		\lesssim h^{s+1}\|v\|_{3,\O} \|\chi\|_{2+s,\O},		 
	\end{split}
\end{equation}
where we have used Lemmas \ref{trace-scaling} and~\ref{lemma:average:edge}, along with Assumption~\ref{virtual:interp:result}. 
On the other hand, for all $ p_0 \in \P_{0}(e)$,  we have the following identity	
\begin{equation*}
	\begin{split}
		\int_e p_0 \Jump{\frac{\partial(\chi-\chi_I)}{\partial \bt_{e} }} &= - \int_e  \frac{\partial p_0}{\partial \bt_e} \jump{\chi-\chi_I} +
		(\jump{\chi-\chi_I}p_0)(\vb_2)-(\jump{\chi-\chi_I}p_0)(\vb_1)=0,
	\end{split}
\end{equation*}
where we have used the fact that the jump $\jump{\chi-\chi_I}$ is zero when evaluated 
at the endpoints $\vb_1$ and $\vb_2$ of edge $e \in \EE_h$ (by construction of the interpolant $\chi_I$; see Assumption \ref{virtual:interp:result} and definition of space $\Vh$).

Thus, taking  $p_0 =\widetilde{\Pi}_e^0 \zeta_3 \in \P_{0}(e)$ in the above identity, we have that
\begin{equation*}
	\begin{split}
		T_3 
		&	= \sum_{e \in \EE_h} \int_e  \zeta_3 \Jump{ \frac{\partial(\chi- \chi_I)}{\partial \bt_{e} }}
		\leq \Big( \:\sum_{e \in \EE_h} |e|^{-1}\|\zeta_3-\widetilde{\Pi}_e^0 \zeta_3\|^2_{0,e} \Big)^{1/2}
		\Big( \:\sum_{e \in \EE_h} |e| \left\|\jump{ \nabla(\chi- \chi_I)\cdot \bt_{\E} } \right\|^2_{0,e} \Big)^{1/2}.
	\end{split}
\end{equation*}

Then, employing the same arguments used to obtain the estimation \eqref{T2}, we  get 
\begin{equation}\label{T3}
	T_3   \lesssim h^{s+1}\|v\|_{3,\O} \|\chi\|_{2+s,\O}.
\end{equation} 
By inserting \eqref{T1:dos}, \eqref{T2} and \eqref{T3} in \eqref{itp1:dos},  since 	$s_0 \in (0,1]$, we find
\begin{equation}\label{interpolation:dos}
	A(v,\chi-\chi_I) \lesssim h^{s+s_0}\|v\|_{3,\O} \|\chi\|_{2+s,\O}.
\end{equation}

Now, we observe  that by the real method of interpolation $H^{2+s_0}(\O)$ can be obtained as the intermediate space of $H^2(\O)$ and $H^3(\O)$ (see for instance~\cite{Adams2003_Sobolev,BSZZ2013}). Thus, by applying the above fact, estimates~\eqref{interpolation:dos} and \eqref{interpolation:uno}, we arrive at the desired thesis.

\end{proof}
\subsection{Companion and Ritz operators}\label{Comp:Ritz:opertors}

In this subsection we assume the existence of a Companion operator $\Jh$ mapping from the nonconforming VE space $\Vh$ to the continuous space $V$. In addition, by using such operator, we propose the construction of a novel Ritz projection, which allows us to develop a new unified error analysis for our NCVEMs with optimal error estimates under \emph{minimal regularity} condition of the solution (in space).

We start assuming the existence of a Companion operator $\Jh$ satisfying some useful orthogonality properties. 

\begin{assumption}[Existence of a Companion operator]\label{Assump:companion}
	There exists a linear operator $\Jh: \Vh \to V$ satisfying the following properties:	
	\begin{itemize}
		\item[(a)]  $A(v_h-\Jh v_h, \chi_{k})=0$ for all $\chi_k \in \P_k(\O_h)$;
		\item[(b)]   $\nabla(v_h-\Jh v_h) \perp \PP_{k-3}(\O_h)$   in $\LL^2(\O)$, for $k  \geq 3$;
		\item[(c)]  $(v_h-\Jh v_h) \perp \P_{k}(\O_h)$  in $L^2(\O)$, for $k  \geq 2$;
		\item[(d)] $\displaystyle\sum_{j=0}^{2}h^{j-2}|v_h-\Jh v_h|_{j,h} \lesssim \inf_{q_k\in \P_k(\O_h)}|q_k-v_h|_{2,h} + \inf_{v \in V}|v-v_h|_{2,h}.$ 
	\end{itemize}	
\end{assumption}

The construction of companion operators for the specific virtual spaces
$ \Vh $ and boundary conditions considered in this work
is developed in Subsection~\ref{const:Compation-operator}.
While inspired by the recent approach of~\cite{Khot2025_MathComp},
our construction is adapted and extended to the present framework.
\paragraph{A modified Ritz-type operator.} Next, we shall define a novel Ritz-type operator, which will allow 
us to obtain optimal error estimates under minimal regularity assumption (in space) on the continuous solution. 
Indeed, we start by defining useful tools.

Recalling the notation \eqref{notation:broken}, let us consider the forms in the sum space $V + \Vh$,
	\begin{equation}\label{sum:forms}
	\begin{split}
	\widehat{A}(\cdot,\cdot) &:= \alpha_1 A(\cdot,\cdot) + \alpha_2 B(\cdot,\cdot) + \alpha_0 (\cdot,\cdot)_{\O};\\
    \widehat{A}_h(\cdot,\cdot) &:= \alpha_1 A_h(\cdot,\cdot) + \alpha_2 B_h(\cdot,\cdot) + \alpha_0 (\cdot,\cdot)_{\O},		
	\end{split}	
	\end{equation}
	 where	 $A(\cdot,\cdot)$, $B(\cdot,\cdot)$, $A_h(\cdot,\cdot)$ and $B_h(\cdot,\cdot)$ are the bilinear forms defined in~\eqref{bilinear-contA},~\eqref{bilinear-contB},~\eqref{general:discrete:Ah},~\eqref{general:discrete:Bh}, respectively.
	 In order to develop an unified analysis the constant $\alpha_0 \geq 0$ is chosen so that the bilinear forms $\widehat{A}(\cdot,\cdot)$ and $\widehat{A}_h(\cdot,\cdot)$ be coercive in  $\Vh$ (for the three BCs). More precisely, we set 
		 \begin{equation*}
		\alpha_0	 =
		\begin{cases}
			0  & \quad \text{for \CP \: and \SSP};\vspace*{2mm}\\	
			1  & \quad \text{for \CH}.
		\end{cases}
	\end{equation*}	  
	
Further, let $\Jh: \Vh \to V$  be any operator satisfying the properties in Assumption~\ref{Assump:companion}, then for each $\varphi \in V$, we consider the functional (see Figure~\ref{fig:functional}):
\begin{equation}\label{def:functional}
	\begin{split}
		\widehat{A}_{\varphi}&:\Vh \longrightarrow \R\\
		& \quad v_h \longmapsto \widehat{A}_{\varphi}(v_h):= \big(\widehat{A}(\varphi, \cdot) \circ \Jh\big)(v_h)= \widehat{A}(\varphi, \Jh v_h).
	\end{split}
\end{equation}
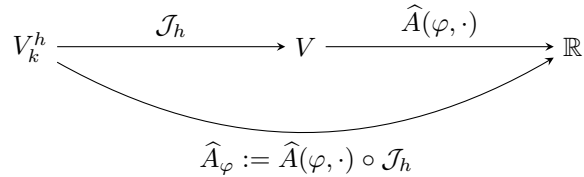
\begin{figure}[ht!]
	\begin{center}
		\begin{tikzpicture}
			\node at (0,0) (a) {$\Vh$};
			\node[right =3cm of a]  (b){$V$};
			\node[right =3cm of b]  (c){$\R$};
			\draw[->,>=stealth] (a) --node[above]{$\Jh $} (b);
			\draw[->,>=stealth] (b) --node[above]{$\widehat{A}(\varphi, \cdot)$} (c);
			\draw[->,>=stealth] (a) edge[bend right=30]node[below]{$\widehat{A}_{\varphi}:= \widehat{A}(\varphi, \cdot) \circ \Jh $} (c);
		\end{tikzpicture}
		\caption{Schematic representation of the construction of the functional
			$\widehat{A}_{\varphi}$ defined in~\eqref{def:functional}.}
		\label{fig:functional} 
	\end{center}
\end{figure}

Thus, using the functional introduced above, we define the Ritz-type operator $\mathcal{R}_h: V \rightarrow  \Vh$  as the solution of problem:
\begin{equation}\label{Ritz:operator}
	\widehat{A}_h(\mathcal{R}_h \varphi, v_h)= \widehat{A}_{\varphi}(v_h) \qquad  \forall v_h \in \Vh.
\end{equation}
Here, $\widehat{A}_h(\cdot,\cdot)$ denotes the discrete bilinear form defined in~\eqref{sum:forms}.

In order to derive error estimates for $\mathcal{R}_h$ we first establish several preliminary results. With this aim, we consider the following norms
	 \begin{equation}\label{not:norms}
	 	\begin{cases}
	|||v|||_{2,\O}	: =	|v|_{2,\O} + \|v\|_{0, \O}  & \quad \text{for $v \in V$};\vspace*{2mm}\\	
	|||v_h|||_{2,h}	: =	|v_h|_{2,h} + \|v_h\|_{0, \O}  & \quad \text{for $v_h \in V + \Vh$}.
			\end{cases}
		\end{equation}

We have the following key result involving the norms defined above.
\begin{lemma}\label{bounds:h1-norms}
For all $v \in V$ and for each $v_h \in \Vh$, it hold true
\begin{align*}
|v|_{1,\O} \lesssim |||v|||_{2,\O}\qquad \text{and} \qquad |v_h|_{1,h} \lesssim |||v|||_{2,h}.
\end{align*}
\end{lemma}
\begin{proof}
Let us proof the first bound. Indeed, let $v \in V$. Then, for the three types of boundary conditions, we have
\begin{equation*}
|v|_{1,\O}^2 = (\nabla v, \nabla v )_{\O} = ( v, \dn v )_{\partial \O}-  (\Delta v , v )_{\O} =  (\Delta v , v )_{\O}		
\end{equation*}
and consequently, by using the Young's inequality, we have 
\begin{equation}\label{ineq:triple:norm}
|v|_{1,\Omega}^2
\lesssim
\|\Delta v\|_{0,\O}^2 + \|v\|_{0,\O}^2  \lesssim |||v|||^2_{2,\O}.
\end{equation}

Now,  we will proof the  second inequality. Indeed, for the cases $\CP$ and $\SSP$, the proof follows from the fact that (cf. \cite{Zhao2016_M3AS,Zhao2018_JSC,Antonietti2018_M3AS})
\begin{equation*}
\|v_h\|_{0,\O} + |v_h|_{1,h} \lesssim |v_h|_{2,h} \quad \forall v_h \in \Vh.	
\end{equation*}
For the $\CH$ boundary condition case, we consider $v_h \in \Vh$. Then, by using the properties of the operator $\Jh$, and~\eqref{ineq:triple:norm}, we obtain
\begin{equation*}
\begin{split}
|v_h|_{1,h} &\leq |v_h-\Jh  v_h|_{1,h} + |\Jh  v_h|_{1,\O}	\\
&\lesssim |v_h-\Jh  v_h|_{1,h} + |\Jh  v_h|_{2,\O} + \|\Jh v_h\|_{0,\O}	\\
&\lesssim h|v_h|_{2,h} + |v_h|_{2,h} + \|\Jh v_h-v_h\|_{0,\O} +\|v_h\|_{0,\O}\\
&\lesssim 3|v_h|_{2,h} +\|v_h\|_{0,\O}\\
&\lesssim |||v_h|||_{2,h},
\end{split}	
\end{equation*} 
where we used the properties of $\Jh $ (cf. Assumption~\ref{Assump:companion}), and the fact $0<h \leq 1$.

\end{proof} 

For the next result we recall the notations~\eqref{notation:broken} and~\eqref{not:norms}. 
\begin{proposition}\label{prop:A:Ah:forms}
Let $\widehat{A}(\cdot,\cdot)$ and $\widehat{A}_h(\cdot,\cdot)$  be the bilinear forms defined in \eqref{sum:forms}. Then, for all $v,w \in V$  and  $v_h,w_h \in \Vh$ , the following bounds hold
\begin{equation*}
	\begin{split}
|\widehat{A}(v,w)| &  \lesssim |||v|||_{2,\O} |||w|||_{2,\O}
	\qquad \	\quad \text{and} \qquad 	|\widehat{A}(v_h,w_h)|   \lesssim |||v_h|||_{2,h} |||w_h|||_{2,h}\\	
|\widehat{A}_h(v_h,w_h)|  & \lesssim |||v_h|||_{2,h} ||||w_h|||_{2,h}\qquad \text{and} \qquad	
\widehat{A}_h(v_h,v_h) \gtrsim  |||v_h|||_{2,h}^2.			
	\end{split}
\end{equation*}
\end{proposition}
\begin{proof}
We will proof the boundedness of the bilinear form $\widehat{A}_h(\cdot,\cdot)$, since the proof for the continuous form   $\widehat{A}(\cdot,\cdot)$ is analogous. In fact, let  $v_h,w_h \in \Vh$. Then, using Lemma~\ref{bounds:h1-norms}, we have
\begin{equation*}
\begin{split}
|\widehat{A}_h(v_h,w_h)| 
& \leq  \max\{\alpha_0,\alpha_1, \alpha_2\} 	\big(|v_h|_{2,h}|w_h|_{2,h} + |v_h|_{1,h}|w_h|_{1,h}  + ||v_h||_{0,\O} ||w_h||_{0,\O} \big)\\	
& \lesssim \big(|v_h|_{2,h}|w_h|_{2,h} + |||v_h|||_{2,h} |||w_h|||_{2,h} + ||v_h||_{0,\O} ||w_h||_{0,\O} \big) \\	
& \lesssim  |||v_h|||_{2,h} |||w_h|||_{2,h}.
\end{split}
\end{equation*}
Moreover, the  $\Vh$-ellipticity of $\widehat{A}_h(\cdot,\cdot)$ follows directly from its definition (cf. \eqref{sum:forms}).

\end{proof}

The following theorem  is one of the main results of this section. It provides the well definition of the operator $\mathcal{R}_h$ (cf. \eqref{Ritz:operator}) and optimal error estimates in $|||\cdot|||_{2,h}$ and  $|\cdot|_{1,h}$. We point out that this Ritz-type projection is  different from those reported in the existing VEM literature~\cite{Zhao2019_parabolic,Li2021_IMA,Pei2023_CMA}. In particular, we emphasize that the construction of this operator requires only  $H^2$-regularity, and that the error estimates are established under a \emph{minimal spatial regularity} assumption.

\begin{theorem}[Ritz-type projection]\label{Error:Ellip}
Let $\varphi \in V$ and $k \geq 2$. Then, there exists a unique function $\mathcal{R}_h \varphi \in \Vh$ satisfying~\eqref{Ritz:operator}.  
	
Moreover, if $\varphi \in V \cap H^{2+s}(\Omega)$ for some $0 < s \leq k-1$, the following error estimates hold:
	\begin{enumerate}
		\item There exists a positive constant $C$, independent of $h$, such that
		\begin{equation*}
			||| \varphi - \mathcal{R}_h \varphi |||_{2,h} \leq C \, h^{s} \,  \|\varphi\|_{2+s,\Omega}.
		\end{equation*}
		
		\item There exist a constant $C > 0$ and an index $\tilde{s} \in (0,1]$, independent of $h$, such that
		\begin{equation*}
			\vert \varphi - \mathcal{R}_h \varphi \vert_{1,h} \leq C \, h^{s+\tilde{s}} \, \|\varphi\|_{2+s,\Omega}.
		\end{equation*}
		The index $\tilde{s}$ depends on the largest re-entrant angle of the domain $\Omega$ and on the boundary conditions.
	\end{enumerate}
\end{theorem}
\begin{proof}
We divide the proof into three steps. First, we establish the well-posedness of the Ritz operator. In the second step, we derive the error estimates in the broken $H^2$-norm. Finally, in the third step, we obtain the corresponding $H^1$-error estimates.
We recall that, with the unified notation and the previously established results (cf. Lemma~\ref{bounds:h1-norms} and Proposition~\ref{prop:A:Ah:forms}), the present estimates can be derived uniformly for all boundary conditions.
\paragraph{Step 1: Well-posedness of $\mathcal{R}_h$.} 
We first recall that, by Proposition~\ref{prop:A:Ah:forms}, the bilinear form 
$\widehat{A}_h(\cdot, \cdot)$ is coercive and bounded. 
Moreover, combining the same proposition with property (d) in Assumption~\ref{Assump:companion}, we obtain
	\begin{equation*}
	\begin{split}
	|\widehat{A}_{\varphi}(v_h)| &=|\widehat{A}(\varphi,\Jh  v_h)| \leq \max\{\alpha_0,\alpha_1,\alpha_2\}|||\varphi|||_{2,\O}|||\Jh  v_h|||_{2,\O} \\
	& = \alpha_{M} C_{\varphi}(|\Jh  v_h|_{2,\O}+\|\Jh  v_h\|_{0,\O})\\
	&  \lesssim  |v_h|_{2,h} + 2\|v_h\|_{0,\O} \\
	& \lesssim  |||v_h|||_{2,h},	
	\end{split}	
	\end{equation*}
	where the hidden constant is independent of $h$, $\alpha_{M}=\max\{\alpha_0,\alpha_1,\alpha_2\} $ $C_{\varphi}=|||\varphi|||_{2,\O}$. Therefore,  the  functional $\widehat{A}_{\varphi}(\cdot)$ is bounded in $\Vh$.  Therefore, by applying the   Lax-Milgram theorem we conclude that $\mathcal{R}_h $ is well-defined.
	
\paragraph{Step 2: Error estimates in the $|||\cdot|||_{2,h}$ norm.} 
We now assume that $\varphi \in V \cap H^{2+s}(\Omega)$, with $0 < s \leq k-1$. 
Using the interpolation operator in the virtual space (cf. Assumption~\ref{virtual:interp:result}), we split
\begin{equation}\label{split}
\varphi - \mathcal{R}_h \varphi = (\varphi - \varphi_I) + ( \varphi_I-\mathcal{R}_h \varphi) =: \zeta_h + \theta_h.
\end{equation}

The estimate for $\zeta_h$ follows directly from the virtual element approximation in Assumption~\ref{virtual:interp:result}. 
The term $\theta_h$, however, requires a more detailed analysis. 
In particular, by exploiting the coercivity of $\widehat{A}_h(\cdot,\cdot)$ and the definition of $\mathcal{R}_h$ 
(cf. Proposition~\ref{prop:miltilinear} and~\eqref{Ritz:operator}), 
and then adding and subtracting $\varphi_{\pi} \in \P_{k}(\E)$, we obtain:
	\begin{equation}\label{main:inequ:H2norm}
	\begin{split}
		|||\theta_h|||_{2,h}^2 &\lesssim \widehat{A}_h(\theta_h,\theta_h)=\widehat{A}_h(\varphi_I, \theta_h)-\widehat{A}_h(\mathcal{R}_h \varphi, \theta_h)
		= \widehat{A}_h(\varphi_I, \theta_h)-\widehat{A}(\varphi, \Jh  \theta_h)\\
		&= \widehat{A}(\varphi, \theta_h-\Jh  \theta_h)+(\widehat{A}_h(\varphi_I, \theta_h)-\widehat{A}(\varphi, \theta_h))\\
		&= \widehat{A}(\varphi-\varphi_{\pi},\theta_h- \Jh  \theta_h)+\widehat{A}(\varphi_{\pi}, \theta_h-\Jh  \theta_h) + (\widehat{A}_h(\varphi_I, \theta_h)-\widehat{A}(\varphi, \theta_h))\\
		&=: T_1+T_2 + T_3. 
	\end{split}
\end{equation}

The term $T_1$ is bounded by using the Lemma~\ref{polynomial:interp:result},  the fact that $|\Jh  \theta_h|_{2,\O} \lesssim |\theta_h|_{2,h}$, and Lemma~\ref{bounds:h1-norms} as follows:
	\begin{equation*}
		\begin{split}
			T_1 &= 	\widehat{A}(\varphi-\varphi_{\pi}, \theta_h- \Jh  \theta_h)  \\
	&\lesssim (|||\varphi-\varphi_{\pi}|||_{2,h}+|\varphi-\varphi_{\pi}|_{1,h})(|||\Jh  \theta_h|||_{2,\O}+|||\theta_h|||_{2,h}+|\Jh  \theta_h|_{1,\O}+ |\theta_h|_{1,h})\\
			&\lesssim  (h^{s}+h^{1+s})  |\varphi|_{2+s,\O}|||\theta_h|||_{2,h}.
		\end{split}
	\end{equation*}

For the term $T_2$, we use orthogonality properties (a) and (c) in Assumption~\ref{Assump:companion}, to get
	\begin{equation*}
		\begin{split}
			T_2 &=	\widehat{A}(\varphi_{\pi}, \theta_h- \Jh  \theta_h) = \alpha_2 B(\varphi_{\pi}, \theta_h- \Jh  \theta_h).
		\end{split}
	\end{equation*}
Next, we consider two cases. If $k =2$, from approximation property (d) in Assumption~\ref{Assump:companion}, for $0<s \leq 1$ we obtain
	\begin{equation*}
		\begin{split}
T_2 & = \alpha_2 B(\varphi_{\pi}, \theta_h- \Jh  \theta_h) \lesssim  |\varphi_{\pi}|_{1,h} |\theta_h- \Jh  \theta_h|_{1,h}
\lesssim  h^s \|\varphi\|_{2,\O}  |||\theta_h|||_{2,h}.
		\end{split}
	\end{equation*} 
	On the another hand, if $k \geq 3$, for every $\boldsymbol{q}_{k-3} \in \PP_{k-3}(\O_h)$, we use orthogonality property (b) in Assumption~\ref{Assump:companion}, to obtain
	\begin{equation*}
		\begin{split}
			T_2 &  =\alpha_2 \sum_{\E\in \O_h}  \big(\nabla\varphi_{\pi}, \nabla (\theta_h- \Jh  \theta_h)\big)_{\E} = \alpha_2 \sum_{\E\in \O_h}  \big(\nabla\varphi_{\pi}- \boldsymbol{q}_{k-3}, \nabla (\theta_h- \Jh  \theta_h)\big)_{\E} \\
			&= 	\alpha_2 \sum_{\E\in \O_h} \big( (\nabla\varphi_{\pi} - \nabla\varphi) + (\nabla\varphi- \boldsymbol{q}_{k-3}), \nabla(\theta_h- \Jh  \theta_h)\big)_{\E} \\	
			&\lesssim \big( \|\nabla(\varphi_{\pi} -\varphi)\|_{0,h} + \|\nabla\varphi -\boldsymbol{q}_{k-3}\|_{0,h} \big) |\theta_h- \Jh  \theta_h|_{1,h}.
		\end{split}
	\end{equation*} 
	Hence, taking $\boldsymbol{q}_{k-3} = \boldsymbol{\Pi}_{\E}^{k-3} \nabla \varphi  \in \PP_{k-3}(\O_h)$, and then applying approximation properties of the operators $\boldsymbol{\Pi}_{\E}^{k-3}$ and $ \Jh $ together with Lemma~\ref{polynomial:interp:result}, we derive 
	\begin{equation*}
		T_2  \lesssim  h^{\min\{s,k-2\}} |\varphi|_{2+s,\O} h |||\theta_h|||_{2,h}  \lesssim  h^s |\varphi|_{2+s,\O}|||\theta_h|||_{2,h},
\end{equation*}
where we have used approximation property (d) in Assumption~\ref{Assump:companion}.

Now, we will analyze the term $T_3$. By applying the consistency property of the bilinear form $\widehat{A}^{\:\E}_h(\cdot,\cdot)$ (which is a consequence of $({\bf A5})$ in Assumption \ref{assump:VEM}) we have 
\begin{equation*}
		\begin{split}
			T_3 &= \sum_{\E \in \O_h} \big(\widehat{A}^{\E}_h(\varphi_I, \theta_h)-\widehat{A}^{\E}(\varphi, \theta_h)\big) = \sum_{\E \in \O_h} \big(\widehat{A}^{\:\E}_h(\varphi_I-\varphi_{\pi}, \theta_h) + \widehat{A}^{\:\E}_h(\varphi_{\pi}, \theta_h) -\widehat{A}^{\E}(\varphi, \theta_h)\big) \\
			&=\sum_{\E \in \O_h}  \big(\widehat{A}^{\E}_h(\varphi_I-\varphi_{\pi}, \theta_h) + \widehat{A}^{\E}(\varphi_{\pi}-\varphi, \theta_h)\big) \\
			& \lesssim (|||\varphi_I-\varphi_{\pi}|||_{2,h}  +|\varphi_I-\varphi_{\pi}|_{1,h}+|||\varphi-\varphi_{\pi}|||_{2,h} + |\varphi-\varphi_{\pi}|_{1,h} )|||\theta_h|||_{2,h}\\ 
			& \lesssim  (h^{s}+ h^{1+s})|\varphi|_{2+s,\O} |||\theta_h|||_{2,h}.
		\end{split}
	\end{equation*}
 	
Thus, by combining the above estimates and~\eqref{main:inequ:H2norm}, using identity~\eqref{split} and triangle inequality, we deduce the desired result. 
\paragraph{Step 3: Error estimates in the $|\cdot|_{1,h}$-seminorm.} 
Although the derivation of this bound is based on a duality argument, it is not straightforward and requires additional analysis.
In particular, by employing the Companion operator $\Jh$ together with the interpolation $\varphi_I$ of $\varphi$ in $\Vh$, we obtain
	\begin{equation}\label{ineq:bound:H1:main}
		\begin{split}
			|\varphi-\mathcal{R}_h \varphi|_{1,h}
			&\leq  |\varphi-\varphi_I|_{1,h}+|\theta_h-\Jh  \theta_h|_{1,h}+ |\Jh  \theta_h|_{1,h}\\
			&\leq  |\varphi-\varphi_I|_{1,h}+h|\theta_h|_{2,h}+ |\Jh  \theta_h|_{1,h}\\
			& \lesssim  h^{1+s}|\varphi|_{2+s,\O}+ h(|\varphi_I-\varphi|_{2,h}+|\varphi-\mathcal{R}_h \varphi|_{2,h})+\|\nabla \Jh  \theta_h\|_{0,\O}\\
			& \lesssim  h^{1+s}|\varphi|_{2+s,\O}+\|\nabla \Jh  \theta_h\|_{0,\O},
		\end{split}
	\end{equation}
	where we have used $|\theta_h-\Jh  \theta_h|_{1,h} \lesssim h|\theta_h|_{2,h}$, Assumption~\ref{virtual:interp:result} and the error $H^2$-estimate derived in Step 1.

	Next, to obtain error bounds for the term $\|\nabla \Jh  \theta_h\|_{0,\O}$, we consider the following auxiliary problem: find $\xi^* \in V$ such that
	\begin{equation}\label{aux:H1}
	\widehat{A}(\xi^*,v)= (\nabla \Jh  \theta_h,  \nabla v)_{\Omega} \quad \forall v \in V.
	\end{equation}
	Moreover, from \cite{BR80,Brenner2012_SINUM} there exists $\ts \in(0,1]$ (depending only on $\O$ and the boundary conditions), such that  $\xi^* \in H^{2+\tilde{s}}(\Omega)$ and
	\begin{equation}\label{add:reg}
		\|\xi^*\|_{2+{\ts},\Omega} \lesssim \| \nabla \Jh \theta_h \|_{0,\Omega}.
	\end{equation}
	Next, by choosing $v= \Jh  \theta_h \in V$ in \eqref{aux:H1}, we obtain 
	\begin{equation}\label{eq:dual:arg}
		\begin{split}
			\|\nabla \Jh  \theta_h \|_{0,\O}^2=\widehat{A}(\xi^*,\Jh  \theta_h)
			&= \widehat{A}( \xi^*,\Jh  \theta_h-\theta_h)
			+ \widehat{A}( \theta_h, \xi^*)\\
			&=: I_1 + I_2. 
		\end{split}
	\end{equation}	
For  $\xi_{\pi}^{*}\in \P_{2}(\O_h)$, we use properties (a) and (c) of Assumption~\ref{Assump:companion}, to get
\begin{equation}\label{ident:ortho}
A(\xi^*_{\pi},\Jh  \theta_h-\theta_h)+ \alpha_0(\xi^*_{\pi},\Jh  \theta_h-\theta_h)_{\O}=0.
\end{equation}

Hence, by adding and subtracting $\xi_{\pi}^{*}$ and using~\eqref{ident:ortho},
the definition of $\theta_h$, Lemma~\ref{polynomial:interp:result},
and Assumption~\ref{virtual:interp:result}, we obtain
	\begin{equation}\label{term:I1}
	\begin{split}
		I_1&= \widehat{A}(\xi^*,\Jh  \theta_h-\theta_h) =\widehat{A}(\xi^*-\xi^*_{\pi},\Jh  \theta_h-\theta_h) +B(\xi^*_{\pi},\Jh  \theta_h-\theta_h)\\
		&  \lesssim |||\xi^*-\xi^*_{\pi}|||_{2,h}|||\Jh  \theta_h-\theta_h|||_{2,h} +|\xi^*_{\pi}|_{1,h}\: |\Jh \theta_h-\theta_h|_{1,h}\\
		&	\lesssim  h^{\ts}\|\xi^*\|_{2+{\ts},\O}\: |||\theta_h|||_{2,h}+ h\|\xi^*\|_{2,\O}\:  |\theta_h|_{2,h}\\
		&	\lesssim  h^{\ts}\|\xi^*\|_{2+{\ts},\O}\: |||\theta_h|||_{2,h}\\
		&\lesssim h^{\ts}\|\xi^*\|_{2+{\ts},\O}(|||\varphi_I-\varphi|||_{2,h}+|||\varphi-\mathcal{R}_h \varphi|||_{2,h})\\
		&\lesssim h^{s+\ts}|\varphi|_{2+s,\O}\| \nabla \Jh  \theta_h \|_{0,\O},
	\end{split}
\end{equation}
where we have also used the error estimate in $|||\cdot|||_{2,h}$  norm obtained in Step 1, and bound~\eqref{add:reg}.

The term $I_2$ needs more analysis. We start  by adding and subtracting $\xi^*_I \in V_2^h$ (the $V_2^h$-interpolant of  $\xi \in V \cap H^{2+\ts}(\O)$), then we apply Lemma~\ref{lemma:extra:power}, Assumption~\ref{virtual:interp:result}, and  again the error estimate in the $|||\cdot|||_{2,h}$  norm, to obtain
	\begin{equation*}
		\begin{split}
			I_2 &= \widehat{A}(\xi^*,\theta_h) = \widehat{A}(\xi^*,\varphi_{I}-\varphi) +  \widehat{A}(\xi^*,\varphi-\mathcal{R}_h\varphi)\\
			& =  \big(\alpha_1 A(\xi^*,\varphi_{I}-\varphi) + \alpha_2 B(\xi^*,\varphi_{I}-\varphi) + \alpha_0 (\xi^*,\varphi_{I}-\varphi)_{\O}\big) +  \widehat{A}(\xi^*-\xi^*_I,\varphi-\mathcal{R}_h\varphi)+  \widehat{A}(\xi^*_I,\varphi-\mathcal{R}_h\varphi)\\
			&\lesssim
			h^{s+{\ts}}\|\xi^*\|_{2+{\ts},\O}|\varphi|_{2+s,\O}+ (|||\xi^*-\xi^*_I|||_{2,h}+|\xi^*-\xi^*_I|_{1,h})(|||\varphi-\mathcal{R}_h \varphi|||_{2,h}+|\varphi-\mathcal{R}_h \varphi|_{1,h})+\widehat{A}(\xi^*_I,\varphi-\mathcal{R}_h\varphi)\\
		&\lesssim
		h^{s+{\ts}}\|\xi^*\|_{2+{\ts},\O}|\varphi|_{2+s,\O}+h^{s+{\ts}}\|\xi^*\|_{2+{\ts},\O}|\varphi|_{2+s,\O} +\widehat{A}(\xi^*_I,\varphi-\mathcal{R}_h\varphi)\\
			&\lesssim h^{s+{\ts}}|\varphi|_{2+s,\O}\| \nabla \Jh  \theta_h \|_{0,\O}+ \widehat{A}(\xi^*_I,\varphi-\mathcal{R}_h\varphi).
		\end{split}
	\end{equation*} 

In turn, exploiting the symmetry of $\widehat{A}(\cdot,\cdot)$, the definition of
the operator $\mathcal{R}_h(\cdot)$, and the orthogonality property of $\mathcal{J}_h$,
and arguing as in \eqref{ident:ortho}, we obtain the following identity for the
last term above, valid for every $\varphi_{\pi} \in \mathbb{P}_{k}(\O_h)$:
\begin{equation*}
		\begin{split}
			\widehat{A}(\xi^*_I,\varphi-\mathcal{R}_h\varphi) 
			&= \widehat{A}(\varphi,\xi^*_I)-\widehat{A}(\mathcal{R}_h\varphi,\xi^*_I) = \widehat{A}(\varphi,\xi^*_I-\Jh \xi^*_I)+\widehat{A}(\varphi,\Jh \xi^*_I)-\widehat{A}(\mathcal{R}_h\varphi,\xi^*_I)\\
			&= \widehat{A}(\varphi,\xi^*_I-\Jh \xi^*_I)+\widehat{A}_h(\mathcal{R}_h\varphi,\xi^*_I)-\widehat{A}(\mathcal{R}_h\varphi,\xi^*_I) \\
			&= \widehat{A}(\varphi-\varphi_{\pi},\xi^*_I-\Jh \xi^*_I)+ \alpha_{2} B(\varphi_{\pi},\xi^*_I-\Jh \xi^*_I) +(\widehat{A}_h(\mathcal{R}_h\varphi,\xi^*_I)-\widehat{A}(\mathcal{R}_h\varphi,\xi^*_I))\\
			&=: I_{21} + I_{22} +I_{23}.
		\end{split}
	\end{equation*}
Using arguments analogous to those employed in the estimation of the term $T_2$ in Step~2, we can derive that
\begin{equation*}
I_{2i} \lesssim h^{s+{\ts}}|\varphi|_{2+s,\O} \|\nabla \Jh  \theta_h\|_{0,\O}, \quad \text{for} \quad i=1,2,3.	
\end{equation*}
Therefore, we have
	\begin{equation}\label{term:I2}
		I_2 \lesssim  h^{s+{\ts}}|\varphi|_{2+s,\O} \|\nabla \Jh  \theta_h\|_{0,\O}.
	\end{equation}
	
Finally, by combining the estimates \eqref{eq:dual:arg}, \eqref{term:I1} and \eqref{term:I2}, we obtain
	\begin{equation*}
		\|\nabla \Jh  \theta_h \|_{0,\O} \lesssim  h^{s+{\ts}}\|\varphi\|_{2+s,\O},
	\end{equation*} 
	and from \eqref{ineq:bound:H1:main} we conclude the desired result.
	
\end{proof}
	
We now state a remark concerning the Ritz operator $\mathcal{R}_h(\cdot)$ and the error
estimates established in Theorem~\ref{Error:Ellip}.

\begin{remark}\label{remark:Ritz:min:regu}
We remark that the well-definition of the operator $\mathcal{R}_h(\cdot)$
only requires minimal spatial regularity of the continuous function
$\varphi$, namely $\varphi \in H^{2}(\Omega$) (cf.~\eqref{Ritz:operator}).
Moreover, the error analysis relies on the additional regularity
$\varphi \in H^{2+s}(\Omega)$, with $s \in (0,k-1]$, which reduces to
$s \in (0,1]$ in the lowest-order case $k=2$.
This regularity framework is milder than those commonly adopted in related
VEM analyses of Ritz-type operators, where higher regularity assumptions,
such as $\varphi \in H^{4}(\Omega)$, are typically imposed even for the
well-definition of the operator and in the lowest-order case $k=2$
(see, e.g.,~\cite{Li2021_IMA,Pei2023_CMA,Adak2023_M3AS}).
	\end{remark}	

We conclude this subsection by recalling the continuous and discrete versions of Gr\"{o}nwall’s inequality, which will be used to derive error estimates for the semi- and fully discrete virtual element schemes (cf.~\eqref{semi_discrete:scheme} and~\eqref{fully:dis:schm}, respectively) developed in the next two subsections.
\begin{lemma}[Gr\"{o}nwall’s Lemma]\label{cont:gronwall}
	Let $g$, $h$, and $r$ be nonnegative integrable functions on $[0, T]$ and suppose that $g$ satisfies
	\[
	g(t) \leq h(t) + \int_0^t r(s) g(s) \, ds \quad \text{for all } t \in (0, T).
	\]
	Then
	\[
	g(t) \leq h(t) + \int_0^t h(s) r(s) \exp\left( \int_s^t r(\tau) \, d\tau \right) ds \quad \text{for all } t \in (0, T).
	\]
	
\end{lemma}

\begin{lemma}[Discrete Gr\"{o}nwall Lemma~\cite{HR1990}]\label{discrete:gronwall}
	Let $D\geq 0$, $a_j$, $b_j$, $c_j$ and $\lambda_j$ be non negative numbers for any integer $j \geq 0$, such that 
	\begin{equation*}
		a_n+ \Delta t \sum_{j=0}^n b_j \leq \Delta t \sum_{j=0}^n \lambda_j a_j+ \Delta t \sum_{j=0}^n c_j+D,  \quad n \geq 0.
	\end{equation*}
	Suppose that $\Delta t \lambda_j <1$ for all $j$, and set $\sigma_j:=(1-\Delta t \lambda_j)^{-1}$. 
	Then, the following bound holds
	\begin{equation*}
		a_n+ \Delta t \sum_{j=0}^n b_j \leq \exp \Big( \Delta t \sum_{j=0}^n \sigma_j \lambda_j \Big) \Big( \Delta t \sum_{j=0}^n c_j+D\Big).
	\end{equation*}
\end{lemma}
	
\subsection{Error analysis for the semi-discrete formulation} \label{semi_discrete}
In this subsection, we derive an abstract error estimate for the semi-discrete
problem~\eqref{semi_discrete:scheme}. The analysis is carried out under suitable
regularity assumptions on the external force $f$, the initial datum $u_0$, and
the exact solution $u$, which are stated below.

\begin{assumption}[Regularity assumptions on the data and solution]
	\label{assump:reg:add}
	The external force, the initial datum, and the solution of
	problem~\eqref{conti:weak:form} satisfy
	\begin{itemize}
		\item $f(u) \in L^2(0,T; H^{s}(\Omega_h))$ and
		$u_0 \in H^{2+s}(\Omega)$;
		\item $u \in L^2(0,T; H^{2+s}(\Omega)) \cap
		H^1(0,T; H^{1+s}(\Omega))$,
	\end{itemize}
	for some $0 < s \leq k-1$.
\end{assumption}

Assumptions of this kind are commonly adopted in the VEM literature; cf.
\cite{Vacca_parabolico,Adak2019_NMPDE,Zhao2019_parabolic}.

Now, we introduce the Ritz projection operator $\mathcal{R}_h(\cdot)$ and decompose the total error into approximation and discrete components, namely,

\begin{equation}\label{split:semi}
	u_h - u
	= (u_h - \mathcal{R}_h u)
	+ (\mathcal{R}_h u - u) =: \eta_h + \rho_h.
\end{equation}

The second term $\rho_h$ represents the projection error and can be directly controlled by the elliptic
error estimate provided in Theorem~\ref{Error:Ellip}. In particular, we obtain

\begin{equation}\label{esti:rho_h}
	\begin{split}
		\|\rho_h\|_{0,\O} &\lesssim  h^{1+s} |u(t)|_{1+ s,\O}
		\lesssim   h^{1+s} \left| u(0) + \int_{0}^{t} \partial_t u(\tau)\, d\tau \right|_{1+s,\O}\\
		&\lesssim  h^{1+s} \left( |u_0|_{1+s, \O} + \int_{0}^{t} |\partial_t u(\tau)|_{1+s,\O}\, d\tau \right)\\
		&\lesssim  h^{1+s} \left( |u_0|_{1+s}
		+ \sqrt{T} \left( \int_{0}^{t} |\partial_t u(\tau)|_{1+s,\O}^2 \, d\tau \right)^{1/2} \right)\\
		&\lesssim  h^{1+s} \left( \|u_0\|_{2+s,\O}
		+ \|\partial_t u\|_{L^2(0,t; H^{1+s}(\Omega))} \right).
	\end{split}
\end{equation}

The control of $\eta_h$, however, requires a different approach. More precisely, 
we derive a new error equation that is tailored to the
nonconforming virtual setting considered here. This equation is obtained
by combining the continuous and semi-discrete formulations
(cf.~\eqref{conti:weak:form} and~\eqref{semi_discrete:scheme}) with the
definition of the Ritz operator $\mathcal{R}_h$ in~\eqref{Ritz:operator}.
In particular, for all $w_h \in \Vh$, the following identity holds:

\begin{equation}\label{main:semi:dis:err}
	\begin{split}
		M_h(\partial_t \eta_h,w_h)&+ \alpha_1 A_h(\eta_h,w_h) + \alpha_2 B_h(\eta_h,w_h)
		= \big(M_h(\partial_t u_h,w_h)+ \alpha_1 A_h(u_h,w_h) + \alpha_2 B_h(u_h,w_h)\big) \\
		& \quad - M_h(\partial_t \mathcal{R}_hu,w_h)- \widehat{A}_h(\mathcal{R}_h u,w_h)+ \alpha_0 (\mathcal{R}_hu,w_h)_{\O} \\
		&= F_h(u_h;w_h) - M_h(\partial_t \mathcal{R}_hu,w_h)- \widehat{A}_h(\mathcal{R}_h u,w_h)+ \alpha_0 (\mathcal{R}_hu,w_h)_{\O} \\
		&= F_h(u_h;w_h) - M_h(\partial_t \mathcal{R}_hu,w_h)- \widehat{A}( u,\Jh  w_h)+ \alpha_0 (\mathcal{R}_hu,w_h)_{\O} \\
		&= \big (F_h(u_h;w_h)-F(u; \Jh  w_h) \big ) + \big ( M(\partial_t u,\Jh w_h)- M_h(\partial_t \mathcal{R}_hu,w_h) \big )\\
		& \quad \quad + \alpha_0 \big ( (\mathcal{R}_hu,w_h)_{\O} -(u, \Jh w_h)_{\O} \big ) \\
		& = \big(F_h(u_h;w_h)-F(u;w_h)\big) + \big( F(u;w_h)-F(u;\Jh  w_h) \big) \\
		& \quad + M(\partial_t u,\Jh w_h-w_h) + \big(M(\partial_t u,w_h) - M_h(\mathcal{R}_h\partial_t u,w_h)\big)\\
		&\quad + \alpha_0(\mathcal{R}_h u-u,w_h)_{\O} + ( u,w_h-\Jh  w_h)_{\O}\\
		&=: \sum_{i=1}^6 T_i. 
	\end{split}
\end{equation}
	
In what follows, we derive error estimates for each term in \eqref{main:semi:dis:err}. In particular, the control of the terms $T_2$ and $T_3$ relies on a global inverse inequality and on the stability property of the operator $\mathcal{J}_h$ (cf.~\eqref{L2:bound}). This, in turn, motivates the introduction of the following quasi-uniformity assumption on the mesh.
\begin{itemize}
	\item[${\bf (A8)}$]  For each $h > 0$ and for each $\E \in \O_h$, there exists a constant $c > 0$, independent of $h$, such that	$h_{\E} \ge \widehat{c}\, h$.
\end{itemize}
Under this assumption, we have the following global inverse inequality.  
\begin{lemma}[An global inverse estimate]\label{inverse-ineq}
	For each $v_h \in \Vh$	there exists a positive constant $C$ independent of $h$, such that 
	\begin{equation*}
		\begin{split}
			|v_h|_{2,h} &\leq Ch^{-2}\|v_h\|_{0,\O}.
		\end{split}
	\end{equation*}
\end{lemma}
\begin{proof}
	From~\cite[Lemma 6.3]{Khot2025_MathComp}, we can obtain the following local inverse inequality
	\begin{equation*}
		|v_h|_{2,\E} \leq C h_{\E}^{-2}\|v_h\|_{0,\E}.
	\end{equation*} 
	Then, under assumption ${\bf (A8)}$ and the above bound, we have 
	\begin{equation*}
		\begin{split}
			|v_h|^2_{2,h} = \sum_{\E \in \O} |v_h|^2_{2,\E} &\leq C \sum_{\E \in \O}h_{\E}^{-4}\|v_h\|^2_{0,\E}\leq C  h^{-4}\sum_{\E \in \O}\|v_h\|^2_{0,\E}  \leq C h^{-4} \|v_h\|^2_{0,\O}.
		\end{split}
	\end{equation*} 
	We conclude the proof by taking the square root on both sides of the inequality.
\end{proof}

As a direct consequence of property (d) and the global inverse inequality stated in Lemma~\ref{inverse-ineq}, we obtain the following stability bound:
\begin{equation}\label{L2:bound}
	\|v_h-\Jh v_h\|_{0,\O} \lesssim \|v_h\|_{0,\O} \qquad \forall v_h \in \Vh.
	\end{equation}

We have the following  bound for  $\|\eta_h\|_{0,\Omega}$.
\begin{lemma}\label{L2:converg:semi:discre:schm}
Let $u(t) \in V$ be the solution of the continuous problem~\eqref{conti:weak:form},
and let $u_h(t) \in \Vh$ denote the solution of the semi-discrete problem~\eqref{semi_discrete:scheme}.
Let Assumption~\ref{assump:reg:add} hold,	then the following error estimate holds:
\begin{equation*}
\begin{split}
	\|\eta_h\|_{0,\Omega}
	\lesssim
	\|u_0 - u_{h,0}\|_{0,\Omega} + h^{s}(& h^{\ts}\|u_0\|_{2+s,\O} + (1+h)\|u\|_{L^2(0,t;H^{2+s}(\O))} \\
	& +\|f(u)\|_{L^2(0,t;H^{s}(\O_h))} + h^{\ts}\|\partial_t u\|_{L^2(0,t;H^{1+s}(\O))} ),
\end{split}
\end{equation*}
where the hidden constant depends on the mesh regularity Assumption~\ref{mesh:regularity}, but it is independent of the mesh size $h$.
\end{lemma}
\begin{proof}
		In what follows, we will establish bounds for the terms $T_i$ ($i=1,\ldots, 6$) in the error equation \eqref{main:semi:dis:err}.  We start with the term $T_1$.  By adding and subtracting adequate terms we have
	\begin{equation*}
		\begin{split}
			T_1 
			& =\sum_{\E \in \O_h} (\PioKk  f(\PioKk u_h)-\PioKk f(\PioKk u),w_h)_{\E} + (\PioKk f(\PioKk u) - \PioKk f(u),w_h)_{\E} 
			+ (\PioKk f( u) - f(u),w_h)_{\E}\\
			& \leq\sum_{\E \in \O_h} \Big(\Vert \PioKk ( f(\PioKk u_h)- f(\PioKk u))\Vert_{0,\E}  +\Vert \PioKk ( f(\PioKk u)- f( u))\Vert_{0,\E} 
			+ \Vert \PioKk f( u) - f(u)\Vert_{0,\E}\Big)\Vert w_h\Vert_{0,\E}\\
			& \leq\sum_{\E \in \O_h} \Big(\Vert  \PioKk u_h-\PioKk u\Vert_{0,\E}  +\Vert \PioKk u- u\Vert_{0,\E} 
			+ \Vert \PioKk f( u) - f(u)\Vert_{0,\E} \Big)\Vert w_h\Vert_{0,\E}\\
			& \lesssim \Big(\Vert  u_h- u\Vert_{0,\O} +h^{2+s}\vert u\vert_{2+s,\O}+ h^{s}\vert f(u)\vert_{s,h}\Big)\Vert w_h\Vert_{0,\O}\\
			& \lesssim \Big(\Vert  \eta_h\Vert_{0,\O} +h^{2+s}\vert u\vert_{2+s,\O} + h^{s}\vert f(u)\vert_{s,h}\Big)\Vert w_h\Vert_{0,\O},
		\end{split}
	\end{equation*}
	where we have used approximation properties of	$\PioKk$ and  Theorem~\ref{Error:Ellip}.
	For the term $T_2$, we  employ the orthogonality property (c) of the Companion operator $\Jh $ in Assumption~\ref{Assump:companion}, approximation properties of $\PioKk$ and bound~\eqref{L2:bound}, as follows: 
	\begin{equation*}
		\begin{split}
			T_2 
			= (f(u)-\PioKk f(u), w_h- \Jh  w_h)_{\O} 
			\lesssim   h^{s}|f(u)|_{s,h} \Vert w_h\Vert_{0,\O}.		
		\end{split}
	\end{equation*}

Next, for the term $T_3$, we add and subtract the term $\Pi_\E^{k} \partial_t u $, then we use orthogonality property in Assumption~\ref{Assump:companion}-(c), and we use again bound~\eqref{L2:bound}, to obtain
	\begin{equation*}
		\begin{split}
			T_3 
			& = M(\partial_t u-\Pi_\E^{k} \partial_t u,\Jh  w_h-w_h)  + M(\Pi_\E^{k} \partial_t u,\Jh  w_h-w_h) \\
			& = M(\partial_t u-\Pi_\E^{k} \partial_t u,\Jh  w_h-w_h) \\
			&\lesssim   h^{1+s} | \partial_t u |_{1+s,\O} \Vert w_h\Vert_{0,\O}.		
		\end{split}
	\end{equation*}
The last term $T_4 $ is bounded by using the polynomial consistency property of form $M_h(\cdot,\cdot)$ (cf. $({\bf A5})$ in Assumption~\ref{assump:VEM}) as follows
	\begin{equation*}
		\begin{split}
			T_4&=M(\partial_t u,w_h)-M_h(\mathcal{R}_h \partial_t  u, w_h)\\
			& =\sum_{\E \in \O_h}\Big(M^{\E}(\partial_t u-\PioKk \partial_t u, w_h )
			+M_h^{\E}(\PioKk \partial_t u-\mathcal{R}_h \partial_t u,w_h) \Big)\\
			&\lesssim  \sum_{\E \in \O_h}\Big(\Vert\partial_t u-\PioKk \partial_t u\Vert_{\E}\Vert w_h\Vert_{0,\E}
			+\Vert\PioKk \partial_t u-\mathcal{R}_h \partial_t u\Vert_{0,\E}\Vert w_h\Vert_{0,\E}\Big)\\
			&\lesssim (h^{1+s}+  h^{s+\ts})|\partial_t u|_{1+s,\O}\Vert w_h\Vert_{0,\O},
		\end{split}
	\end{equation*}
	where we have used  the Cauchy-Schwarz inequality, the approximation property of $\PioKk$ and $\mathcal{R}_h$.
	
By using similar arguments, along with Assumption \ref{Assump:companion} and Theorem \ref{Error:Ellip}, we can obtain
		\begin{equation*}
			T_5+ T_6 \lesssim h^{s}|u|_{2+s,\O}\Vert w_h\Vert_{0,\O}.
	\end{equation*}
Then, by combining the above estimates, taking $w_h =\eta_h \in \Vh$ in \eqref{main:semi:dis:err} and applying the Young inequality,  we obtain
\begin{equation*}
\begin{split}
\beta_{*} \frac{1}{2} \frac{d}{dt} \Vert \eta_h \Vert^2_{0,\O}	 + \alpha_{1,*}\vert \eta_h \vert^2_{2,h}+   \alpha_{2,*} \vert \eta_h \vert^2_{1,h}  &\lesssim  \Vert \eta_h  \Vert^2_{0,\O}+ h^{2(1+s)}|u|^2_{2+s,\O}  
+h^{2s}|f(u)|^2_{s,h}  \\
&\quad+(h^{2(1+s)}+ h^{2(s+ \ts)})|\partial_t u|^2_{1+s,\O} .
\end{split}	
\end{equation*}

Thus, dropping the positive  term $ \alpha_{1,*}\vert \eta_h \vert^2_{2,h}+   \alpha_{2,*} \vert \eta_h \vert^2_{1,h}$  and integrating over the interval $(0,t)$, we obtain 
\begin{equation*}
\begin{split}
 \Vert \eta_h \Vert^2_{0,\O} \lesssim \Vert \eta_h(0) \Vert^2_{0,\O} + \int_{0}^t \Vert \eta_h  \Vert^2_{0,\O} + h^{2s}&\left( h^{2\ts}|u_0|^2_{2+s,\O} + (1+h^2)\|u\|^2_{L^2(0,t;H^{2+s}(\O))}  \right.\\ 
 &\quad \left. + \: \|f(u)\|^2_{L^2(0,t;H^{s}(\O_h))} + h^{2\ts}\|\partial_t u\|^2_{L^2(0,t;H^{1+s}(\O))} \right).
\end{split}	
\end{equation*} 
Moreover, by using Theorem~\ref{Error:Ellip} and Assumption~\ref{virtual:interp:result}, we have that 
\[
\|\eta_h(0)\|_{0,\O} \lesssim  \|u_{h,0}-u_{0}\|_{0,\O} + h^{s+\ts}\,|u_{0}|_{2+s,\O}.
\]

Then, from the Gr\"{o}nwall inequality (cf. Lemma~\ref{cont:gronwall}), we have 
\begin{equation}\label{ineq:eta:L2}
\begin{split}
\Vert \eta_h \Vert_{0,\O} \lesssim \Vert u_0- u_{h,0}  \Vert_{0,\O} + h^{s}&\left( h^{\ts}|u_0|_{2+s,\O} +(1+h)\|u\|_{L^2(0,t;H^{2+s}(\O))} \right. \\ 
& \left. \quad + \: \|f(u)\|_{L^2(0,t;H^{s}(\O_h))} + h^{\ts}\|\partial_t u\|_{L^2(0,t;H^{1+s}(\O))} \right).
\end{split}	
\end{equation}

\end{proof}
	
Using Lemma~\ref{L2:converg:semi:discre:schm}, we obtain the following convergence result in the $H^{2}$-seminorm.
		\begin{theorem}\label{converg:semi:discre:schm}
Under same assumptions of Lemma~\ref{L2:converg:semi:discre:schm}, the following error estimate holds 	
\begin{equation*}
\begin{split}
|u-u_h|_{2,h}  \lesssim   |u_0- u_{h,0}|_{1,h} + |||u_0- u_{h,0}|||_{2,h} + h^{s}&\left( h^{\ts}\|u_0\|_{2+s,\O} +(1+h)\|u\|_{L^2(0,t;H^{2+s}(\O))} \right. \\
&  \quad \left. + \: \|f(u)\|_{L^2(0,t;H^{s}(\O_h))} + (1+h^{\ts})\|\partial_t u\|_{L^2(0,t;H^{1+s}(\O))} \right),
\end{split}
\end{equation*}
where  the hidden constant depends on mesh regularity Assumption~\ref{mesh:regularity},
but is independent of mesh size $h$.
\end{theorem}	
\begin{proof}
By combining the estimates for the term $T_i$, with $i=1, \ldots, 4$, with $w_h:=\partial_t\eta_h \in \Vh$ in \eqref{main:semi:dis:err} and applying the Young inequality,  we obtain
\begin{equation}\label{ineq:pre:H2bound}
\begin{split}
\beta_{*} \Vert \partial_t \eta_h \Vert^2_{0,\O}+ \frac{1}{2} \frac{d}{dt}  (\alpha_{1,*}\vert \eta_h \vert^2_{2,h}+   \alpha_{2,*} \vert \eta_h \vert^2_{1,h}) &\lesssim  \frac{\beta_*}{2} \Vert \partial_t\eta_h  \Vert^2_{0,\O} + \Vert \eta_h  \Vert^2_{0,\O}\\
& \quad  +   h^{2s}  (|u|^2_{2+s,\O}  +|\partial_t u|^2_{1+s,\O} +|f(u)|^2_{s,h}).
\end{split}	
\end{equation}

Now, we absorb the term $\frac{\beta_*}{2}\Vert\partial_t\eta_h \Vert^2_{0,\Omega}$ into the left-hand side of~\eqref{ineq:pre:H2bound} and drop the resulting positive term. Then, using the estimate for $\Vert \eta_h  \Vert_{0,\O} $ (cf. \eqref{ineq:eta:L2}) and integrating over the interval $[0,t]$ and  we derive
		\begin{equation*}
			\begin{split}
\vert \eta_h \vert^2_{2,h}+  \vert \eta_h \vert^2_{1,h} \lesssim  \vert \eta_h(0) \vert^2_{2,h}+  \vert \eta_h(0) \vert^2_{1,h} &+\Vert u_0- u_{h,0}\Vert_{0,\O} + h^{2s} \left( h^{\ts}|u_0|_{2+s,\O} +(1+h)\|u\|_{L^2(0,t;H^{2+s}(\O))} \right.\\
& \left. \qquad + \: \:
\|f(u)\|_{L^2(0,t;H^{s}(\O_h))} + (1+h^{\ts})\|\partial_t u\|_{L^2(0,t;H^{1+s}(\O))} \right).				
			\end{split}
		\end{equation*}
Finally, since
\[
|\eta_h(0)|_{2,h} + |\eta_h(0)|_{1,h}
\lesssim |u_0 - u_{h,0}|_{2,h} + |u_0 - u_{h,0}|_{1,h}
+ h^{s+\ts}\,|u_0|_{2+s,\Omega},
\]
the desired estimate follows by discarding the nonnegative term
$|\eta_h|_{1,h}^2$ in the previous bounds, and by invoking the decomposition
\eqref{split:semi}, the triangle inequality, and estimate~\eqref{esti:rho_h}.

\end{proof}
\subsection{Error analysis for the fully-discrete scheme}
	In this section we will provide error estimate for the fully-discrete scheme. By using the Ritz projection operator $\mathcal{R}_h$, we split the error as
	\begin{equation}\label{split:fully}
		u^n_h-u(t_n)=\big(u^n_h-\mathcal{R}_h u(t_n)\big)+ \big(\mathcal{R}_h u(t_n)- u(t_n) \big)=: \eta_h^n + \rho^n_h.
	\end{equation}

Since the error bound for $\rho^n_h$ is obtained from Theorem~\ref{Error:Ellip}, similarly as \eqref{esti:rho_h}, we will focus on the term $\eta^n_h$ in \eqref{split:fully}. Indeed, as usual, first we will establish an error equation. By using the continuous and fully-discrete formulation (cf. \eqref{conti:weak:form} and \eqref{fully:dis:schm}), together with the Ritz projection $\mathcal{R}_h$ defined in \eqref{Ritz:operator}, we obtain the following error equation (similar as in \eqref{main:semi:dis:err}).
\begin{equation}\label{main:fully:dis:err}
		\begin{split}
			M_h(\delta_t \eta^n_h,w_h)&+ \alpha_1 A_h(\eta^n_h,w_h) + \alpha_2 B_h(\eta^n_h,w_h)\\
			& = \big( F_h(u(t_n)_h;w_h)-F(u(t_n));w_h) \big) +  \big(F(u(t_n);w_h)-F(u(t_n);\Jh  w_h) \big) \\
			& \quad +M(\partial_t u(t_n),\Jh w_h-w_h)	+ \big( M(\partial_t u(t_n),w_h) - M_h(\mathcal{R}_h\delta_t u(t_n) ,w_h)\big)\\
			&\quad + \alpha_0(\mathcal{R}_h u(t_n)-u(t_n),w_h)_{\O} + \alpha_0( u(t_n),w_h-\Jh  w_h)_{\O}\\
			&=: \sum_{i=1}^6  T^n_i.
		\end{split}
	\end{equation}

The following result provides an error estimate in the broken $\ell^2(H^2)$-norm for the proposed numerical scheme. In addition to Assumption~\ref{assump:reg:add}, we further assume that
\[
u \in  H^2(0, t; L^2(\O)) \quad  \text{for a.e.\: $0<t \leq T$.}
\]
This type of regularity assumption is standard in the VEM literature; see, for instance, \cite{Vacca_parabolico,Adak2019_NMPDE,Zhao2019_parabolic}.

\begin{theorem}\label{converg:fully:discre:schm}
	Let $u(t_n) \in V$ be the solution of the continuous problem at time $t=t_n$ and $u_h^n \in \Vh $ be the VE solution generated by \eqref{fully:dis:schm}. Under mesh Assumption~\ref{mesh:regularity} and the assumptions of Theorem~\ref{theorem:well-posed:fully}, then the following estimate holds
	\begin{equation*}
\Vert u-u_h \Vert_{\ell^2(0,t_n;H^2(\O_h))} \lesssim h^{s} + \Delta t,
	\end{equation*}
	where the hidden constant is positive, and depends on mesh regularity parameter, Sobolev regularity of the exact solution $u$, $\partial_t u$, and $\partial_{tt} u$, and nonlinear force function $f(u)$, but independent of mesh size $h$, and $\Delta t$.
\end{theorem}
\begin{proof} 		
In the forthcoming parts, we will estimates $T^n_i$, where $i=1,\ldots, 6$ in the discrete error equation~\eqref{main:fully:dis:err}. In fact, following the estimates derived in Lemma~\ref{L2:converg:semi:discre:schm}, we have
\begin{equation*}
	|T_1^n| \lesssim \big(\Vert  \eta_h^n \Vert_{0,\O} +h^{s+\ts}\vert u(t_n)\vert_{2+s,\O} + h^{s}\vert f(u(t_n))\vert_{s,h}\big)\Vert w_h\Vert_{0,\O},
\end{equation*}
\begin{equation*}
	|T_2^n| \lesssim h^{s}|f(u(t_n))|_{s,h} \Vert w_h \Vert_{0,\O}, \quad \text{and} \quad 	|T_3^n|  \lesssim h^{s}\vert \partial_t u(t_n) \vert_{1+s,\O} ~ \Vert w_h\Vert_{0,\O}.
\end{equation*}
In addition we have 
	\begin{equation*}
		|T^n_5|+ |T^n_6| \lesssim h^{s}|u(t_n)|_{2+s,\O}\Vert  w_h \Vert_{0,\O}.
\end{equation*}

Using similar arguments as in \cite{Vacca_parabolico}, we derive the following bound:
\begin{equation*}
\begin{split}	
|T_4^n| &\leq   |M(\partial_t u(t_n),w_h) - M_h(\mathcal{R}_h\delta_t u(t_n) ,w_h)| \\	
& \lesssim \big( \left\|
\Delta t\, \partial_t u(t_n)-\bigl(u(t_n) - u(t_{n-1})\bigr)\right\|_{0,\O}
+h^{s}\left\| u(t_n) - u(t_{n-1})\right\|_{2+s,\O}\big) \|w_h\|_{0,\O}\\
& =: \frac{C}{\Delta t}	\big( \xi^n +\chi^n \big) \|w_h\|_{0,\O}.
\end{split}
\end{equation*}
For convenience, we introduce the notation
\[
|||\eta_h^n|||_{0,\Omega}^2 := M_h(\eta_h^n,\eta_h^n).
\]
Taking $w_h=\eta_h^n$ in \eqref{main:fully:dis:err}, combining the above bounds, and using the identity
\[
M_{h}(\delta_t\eta_h^n,\eta_h^n)
=\frac{1}{2\Delta t}
\big(|||\eta_h^n|||_{0,\Omega}^2-|||\eta_h^{n-1}|||_{0,\Omega}^2
+|||\eta_h^n-\eta_h^{n-1}|||_{0,\Omega}^2\big),
\]
we obtain, after applying Young’s inequality,
\begin{equation*}
	\begin{split}
		\frac{1}{2 \Delta t}\Big(&|||\eta_h^n|||^2_{0,\Omega}- |||\eta_h^{n-1}|||^2_{0,\Omega}\Big)
		+ C\big(\alpha_1 |\eta_h^n|^2_{2,h} + \alpha_2 |\eta_h^n|^2_{1,h}\big) \\
		&\lesssim C|||\eta_h^n|||^2_{0,\Omega}
		+ h^{2s}\big(\|f(u(t_n))\|^2_{s,h}
		+ \|u(t_n)\|^2_{2+s,\Omega}
		+ \|\partial_t u(t_n)\|^2_{1+s,\Omega}\big) \\
		&\quad + \frac{C}{\Delta t^2}\big( (\xi^n)^2 +(\chi^n)^2 \big).
	\end{split}
\end{equation*}

Dropping the non-negative term $\alpha_2 |\eta_h^n|^2_{1,h}$, multiplying by $2\Delta t$, and summing over $j = 1, \ldots , n$, we obtain
\begin{equation*}
	\begin{split}
		|||\eta_h^n|||^2_{0,\Omega}
		+ C\Delta t \sum_{j=1}^{n}|\eta_h^j|^2_{2,h}
		&\lesssim ||| \eta_h^0 |||^2_{0,\Omega}
		+ C\Delta t \sum_{j=1}^{n}|||\eta_h^j|||^2_{0,\Omega} \\
		&\quad + \frac{C}{\Delta t} \sum_{j=1}^{n} \big( (\xi^j)^2 +(\chi^j)^2 \big)\\
		&\quad + h^{2s} \Delta t \sum_{j=1}^{n}
		\big(\|f(u(t_j))\|^2_{s,h}
		+ \|u(t_j)\|^2_{2+s,\Omega}
		+ \|\partial_t u(t_j)\|^2_{1+s,\Omega}\big).
	\end{split}
\end{equation*}
 
Using the approximation properties stated in Assumption~\ref{virtual:interp:result}, Theorem~\ref{Error:Ellip} and similar arguments of \cite{Vacca_parabolico,Zhao2019_parabolic}, we derive the following bounds:
\begin{equation*}
\begin{split}
||| \eta^0_h|||^2_{0,\O}  &\lesssim  ||| u_0-u_{h,0}|||^2_{0,\O} + |||u_0-\mathcal{R}_h u_0|||^2_{2,h} \lesssim h^{2s}|u_0|^2_{2+s,\O}\\	
\sum_{j=1}^{n} (\xi^j)^2  & \lesssim \Delta t^3 \sum_{j=1}^{n} \int_{t_{j-1}}^{t_j} \|\partial_{tt} u(s)\|^2_{0,\O} ds \lesssim 
\Delta t^3 \|\partial_{tt} u(s)\|^2_{L^2(0,t_n;L^2(\O))}\\
\sum_{j=1}^{n} (\chi^j)^2  & \lesssim \Delta t h^{2s} \sum_{j=1}^{n} \int_{t_{j-1}}^{t_j} \|\partial_{t} u(s)\|^2_{2+s,\O} ds \lesssim 
\Delta t h^{2s} \|\partial_{t} u(s)\|^2_{L^2(0,t_n;H^{1+s}(\O))}.
\end{split}	
\end{equation*}

Thus, the desired result follows from the above bounds, the equivalence of norms, and an application of the discrete Gr\"{o}nwall inequality (cf.~Lemma~\ref{discrete:gronwall}), provided that the time step is sufficiently small.

\end{proof}

We finish this section with the following remarks regarding the error estimates obtained above.
\begin{remark}\label{remark:minimal:reg}
We emphasize that in the present approach, we have derived  error equations  \eqref{main:semi:dis:err} and \eqref{main:fully:dis:err} different to the classical approaches in \cite{Vacca_parabolico,Adak2019_NMPDE,Pei2023_CMA}.
We observe that the error analysis developed in this section yields optimal error estimates under low spatial regularity of the weak solution. More precisely, for $k = 2$ the error bounds are established under minimal regularity of the continuous solution $u(\cdot, t)$, namely $u(\cdot,t) \in H^{2+s}(\Omega)$ with $s \in (0,1]$. This is a significant result, as it covers more general and practical scenarios where the continuous solution exhibits low regularity, beyond what is typically reported in the VEM literature, for instance, in domains with re-entrant angles, and in cases with mixed or non-homogeneous boundary data. Moreover, for Cahn–Hilliard BCs, the index $s$ can be smaller than $1/2$, and may even approach zero in convex domains \cite{Brenner2012_SINUM}. 
	\end{remark}
\begin{remark}
	The quasi-uniformity assumption (${\bf A8}$) is not required for the formulation of the method itself, nor for the construction of the companion and Ritz operators and the derivation of their properties. However, it is employed to derive a global inverse inequality that is essential for the energy-norm analysis. This requirement naturally stems from the global nature of the companion operator introduced in this work.
\end{remark}

\setcounter{equation}{0}
\setcounter{equation}{0}
\section{Particular choices of nonconforming  VEMs}\label{SECTION:VEM:SPACES}
In this section we will provide specific class of VEMs satisfying the Assumptions made in Section~\ref{SECTION:VEM:ABSTRACT}. In particular, we will consider the $C^0$-nonconforming \cite{Zhao2016_M3AS}, and Morley-type VE schemes of high-order $k \geq 2$ \cite{Antonietti2018_M3AS,Zhao2018_JSC} (see Example~\ref{example:dofs}). 

We start by introducing the following preliminary local VE space.
\begin{definition}[The enlarged VE space]
	For all $\E \in \O_h$ and each $k \geq 2$, the enlarged VE space  $\widetilde{V}_k^{h}(\E)$ is defined by	
	\begin{align*}
		\widetilde{V}_k^{h}(\E)
		:= \Big\{v_h\in \HdoK : \Delta^2v_h\in \P_{k}(\E),  \: \:v_h|_e\in\P_{k}(e),
		\: \: \Delta v_h|_e\in\P_{k-2}(e) \quad \forall e\in\partial\E \Big\},
	\end{align*}
	with dimension $\widetilde{d}_k :=\dim(\widetilde{V}_k^{h}(\E))= N^{\E}(2k-1)+ \dfrac{(k+1)(k+2)}{2}$.	
\end{definition} 

By employing the above enlarged VE space and the enhancement technique~\cite{Ahmad2013_CMA}, we will introduce  the $C^{0}$-nonconforming and  the Morley-type VEMs.

\subsection{The $C^{0}$-nonconforming VEM}
\begin{definition}[The local $C^{0}$-nonconforming  VE space]
	For every polygon $\E$ and any $k \geq 2$, we define the local VE space
	\begin{align*}
		V_k^{h}(\E) \equiv V^{nc,h}_{k}(\E)
		:= \Big\{v_h\in \widetilde{V}^{h}_{k}(\E): (q, v_h- \Pi^{nc,\D}_{\E} v_h)_{0,\E}=0 \quad\forall  q \in \P_{k}(\E) \setminus \P_{k-4}(\E)\Big\},
	\end{align*}
	with dimension $d^{nc}_k :=\dim(V_k^{nc,h}(\E))= N^{\E}(2k-1)+ \dfrac{(k-2)(k-3)}{2}$.
\end{definition} 

Now, we provide the  local DoFs associate with the $C^{0}$-nonconforming VE space.
\begin{definition}[The local DoFs of the $C^{0}$-nonconforming VE space]\label{def:local:dofs:C0}
	For a function $v_h \in H^2(\E)$,  the DoFs-tuple $\boldsymbol{\xi}^{nc}$ for the $C^{0}$-nonconforming  local VE space is given by 
	\begin{equation}\label{dof:tuple:nc}
		\boldsymbol{\xi}_{nc} := (0,k-2,k-2,k-4).
	\end{equation}
	Moreover,  the local DoFs $\Lambda^{nc,\E}_{\boldsymbol{\xi}}$ associated to the tuple $\boldsymbol{\xi}^{nc}$ defined in \eqref{dof:tuple:nc} are given by the following set of linear operators
	\begin{itemize}
		\item[${\bf (D^{nc}_1)}$] The values of $v_h$ at each vertex $\vb \in\VV^{\E}_h$. 
		\item[${\bf (D^{nc}_2)}$] The edge moments of $v_h$ up to order $k-2$ on each $e \in\EE^{\E}_h$,
		\begin{align*}
			h^{-1}_{e} (v_h, p)_{e} \quad \forall p \in \M_{k-2}(e).
		\end{align*} 
		\item[${\bf (D^{nc}_3)}$] The edge moments of $\partial_{\bn_e}v_h$ up to order $k-2$ on each $e \in\EE^{\E}_h$,
		\begin{align*}
			(\partial_{\bn_e} v_h, p)_{e}  \quad \forall p \in \M_{k-2}(e).
		\end{align*} 
		\item[${\bf (D^{nc}_4)}$] The internal moments of $v_h$ up to order $k-4$,
		\begin{align*}
			h^{-2}_{\E} (v_h, p)_{\E}  \quad \forall p \in \M_{k-4}(\E).
		\end{align*}
	\end{itemize}		
\end{definition}
We have the following theorem.
\begin{theorem}
	The triplet $\Big(\E,V^{nc,h}_{k}(\E),\Lambda^{nc,\E}_{\boldsymbol{\xi}}\Big)$ is a finite element in the sense of Ciarlet.
\end{theorem}
\begin{proof}
	The proof can be follows by using standard arguments of the VEM literature (see for instance~\cite{Zhao2016_M3AS,Ahmad2013_CMA,CMS2016,Chinosi2016_CMA}).
	
\end{proof}

\begin{definition}[The global $C^{0}$-nonconforming virtual space]
	For every decomposition $\O_h$ and $k \geq 2$, we define the global virtual element space
	\begin{equation*}
		V^{nc,h}_k  = \{ v_h\in H_k^{2,\nc}(\O_h) \cap H^1(\O): v_h|_{\E}\in V^{nc,h}_{k}(\E) \quad \forall \E \in \O_h \}.
	\end{equation*}
	The global DoFs for the space  $V^{nc,h}_k$ are the direct extension of the local DoFs in Definition~\ref{def:local:dofs:C0}.
\end{definition}

\subsection{The Morley-type virtual element method}

\begin{definition}[The local Morley-type virtual space]
	For every polygon $\E$ and any $k \geq 2$, we define the local VE space
	\begin{align*}
		V_k^{h}(\E) \equiv V^{fnc,h}_{k}(\E)
		:= \Big\{v_h\in \widetilde{V}^{h}_{k}(\E) &: (q, v_h- \Pi^{fnc,\D}_{\E} v_h)_{0,\E}=0 \quad\forall  q \in \P_{k}(\E) \setminus \P_{k-4}(\E)\\
		& \qquad (q, v_h- \Pi^{fnc,\D}_{\E} v_h)_{0,e}=0 \quad\forall  q \in \P_{k-2}(e) \setminus \P_{k-3}(e)\Big\},
	\end{align*}
with dimension $d^{fnc}_k :=\dim(V_k^{fnc,h}(\E))= 2\:N^{\E}(k-1)+ \dfrac{(k-2)(k-3)}{2}$.
\end{definition} 

Next, we provide the local  DoFs associate with the Morley-type VE space.
\begin{definition}[Local DoFs of the Morley-type VE space]\label{def:local:dofs:Morley}
	For a function $v_h \in H^2(\E)$,  the DoFs-tuple $\boldsymbol{\xi}^{fnc}$  for the Morley-type VE space is given by 
	\begin{equation}\label{dof:tuple:fnc}
		\boldsymbol{\xi}^{fnc} := (0,k-3,k-2,k-4).
	\end{equation}
	Moreover,  the DoFs $\Lambda^{fnc,\E}_{\boldsymbol{\xi}}$ associated with the tuple $\boldsymbol{\xi}^{fnc}$ defined in \eqref{dof:tuple:fnc} are given by the following set of linear operators
	\begin{itemize}
		\item[${\bf (D^{fnc}_1)}$] The values of $v_h$ at each vertex $\vb \in\VV^{\E}_h$.
		\item[${\bf (D^{fnc}_2)}$] The edge moments of $v_h$ up to order $k-3$ on each $e \in\EE^{\E}_h$,
		\begin{align*}
			h^{-1}_{e} (v_h, p)_{e} \quad \forall p \in \M_{k-3}(e).
		\end{align*} 
		\item[${\bf (D^{fnc}_3)}$] The edge moments of $\partial_{\bn_e}v_h$ up to order $k-2$ on each $e \in\EE^{\E}_h$,
		\begin{align*}
			(\partial_{\bn_e} v_h, p)_{e} \quad \forall p \in \M_{k-2}(e).
		\end{align*} 
		\item[${\bf (D^{fnc}_4)}$] The internal moments of $v_h$ up to order $k-4$,
		\begin{align*}
			h^{-2}_{\E}  (v_h, p)_{\E}  \quad \forall p \in \M_{k-4}(\E).
		\end{align*} 
	\end{itemize}		
\end{definition}

Based on the DoFs, we have the following result.
\begin{theorem}
	The triplet $\Big(\E,V^{fnc,h}_{k}(\E),\Lambda^{fnc,\E}_{\boldsymbol{\xi}}\Big)$ is a finite element in the sense of Ciarlet.
\end{theorem}

\begin{proof}
	The proof can be obtained from standard arguments of the VEM literature (see for instance~\cite{Antonietti2018_M3AS, Zhao2018_JSC,Ahmad2013_CMA,CMS2016,Chinosi2016_CMA}).
\end{proof}

\begin{definition}[The global Morley-type virtual space]
	For every decomposition $\O_h$ and $k \geq 2$, we define the global VE space
	\begin{equation*}
		V^{nc,h}_k  = \{ v_h\in H_k^{2,\nc}(\O_h): v_h|_{\E}\in V^{fnc,h}_{k}(\E) \quad \forall \E \in \O_h \}.
	\end{equation*}
	The global DoFs for the space  $V^{nc,h}_k$ are the direct extension of the local ones in Definition~\ref{def:local:dofs:C0}.
\end{definition}

\begin{remark}
The virtual element spaces constructed above satisfy all the abstract assumptions introduced in Sections~\ref{SECTION:VEM:ABSTRACT} and~\ref{SECTION:ERROR:UNIFIED}. Consequently, all theoretical results derived at the abstract level, including stability properties and error estimates, apply to these spaces. For related constructions and further details, we refer to \cite{Antonietti2018_M3AS,Zhao2016_M3AS,Zhao2018_JSC,CKP-SINUM2023}.
\end{remark}

\subsection{Construction of the Companion operators}\label{const:Compation-operator}
In this subsection we will provide the construction of the Companion $\Jh$ operator considered in Subsection~\ref{Comp:Ritz:opertors} and satisfying the properties of Assumption~\ref{Assump:companion}. 
The construction is based on the ideas developed in~\cite{CKP-SINUM2023,Khot2025_MathComp}, which are extended to the present setting. Indeed, we have that  in the  design are two main ingredient: first, the construction of a preliminary operator, which  map from the nonconforming VE spaces to a $C^1$-conforming counterpart VE space of one degree higher, and then by using this operator together with  standard bubble-function techniques is possible to design a new companion operator to achieve additional $L^2$-orthogonality  (cf. property (c) in Assumption~\ref{Assump:companion}).

With the aim of clarifying notation and avoid any ambiguity, we will adopt the following notation:
\begin{itemize}
	\item for the $C^0$-nonconforming space $V^{nc,h}_{k}$:
	\begin{equation*}
		\begin{split}
			&\text{The Companion operator: 	$\Jh^{nc}: V^{nc,h}_{k} \to V$;} \\
			&\text{The linear functionals associated with global DoFs:	$\dof^{k,nc}_i$ for $i \in  \{1, \ldots , d^{nc}_k\}$.} 
		\end{split}			
	\end{equation*}
	
	\item for the fully-nonconforming space $V^{fnc,h}_{k}$:
	\begin{equation*}
		\begin{split}
			&\text{The Companion operator: 	$\Jh^{fnc}: V^{fnc,h}_{k} \to  V$;} \\
			&\text{The linear functionals associated with global DoFs:	$\dof^{k,fnc}_i$ for $i \in  \{1, \ldots , d^{fnc}_k\}$.} 
		\end{split}			
	\end{equation*}	
\end{itemize}

Additionally, for sake of completeness we will recall the definition of the $C^1$-conforming counterpart  of the nonconforming VE spaces considered here. In fact, for all $\E \in \O_h$, for our purposes, we consider the VE space of order $\ell+1$ (with $\ell \geq 2$):
\begin{equation*}
	V^{c,h}_{\ell+1}  := \{ v_h\in H^{2}(\O): v_h|_{\E}\in V^{c,h}_{\ell+1}(\E) \quad \forall \E \in \O_h \} \cap V \subset V,		
\end{equation*}
where  the local space is given by
\begin{align*}
	V_{\ell+1}^{c,h}(\E)	:= \Big\{v_h\in \HdoK \cap C^1(\E) : & \: \Delta^2v_h\in \P_{\ell+1}(\E),  \: \:v_h|_e\in\P_{\ell+1}(e),
	\: \: \partial_{\bn} v_h|_e\in\P_{\widehat{\ell}}\:(e) \quad \forall e\in\partial\E,\\
	& \quad  (q, v_h- \Pi^{c,\D}_{\E} v_h)_{0,\E}=0 \quad\forall  q \in \P_{\ell+1}(\E) \setminus \P_{\ell-3}(\E) \Big\},
\end{align*}
with $\widehat{\ell}= \max \{\ell+1,3\}$ and the respective definition for the $H^2$-VEM projector $\Pi^{c,\D^2}_{\E}$ and the global DoFs, which are given:
\begin{itemize}
	\item[${\bf (D^{c}_1)}$] The values $v_h$ and the $h_{\vb}\nabla v_h$ for each vertex $\vb$ of $\E$. Here $h_{\vb}$ is some local length scale associated to the vertex $\vb$.
	\item[${\bf (D^{c}_2)}$] The moments of $v_h$ up to order $\ell-3$ on each $e \subset \partial \E$,
	\begin{align*}
		h^{-1}_{e}  (v_h, p)_{e} \quad \forall p \in \P_{\ell-3}(e).
	\end{align*} 
	\item[${\bf (D^{c}_3)}$] The moments of $\partial_{\bn_e}v_h$ up to order $\ell-2$ on each $e \subset \partial \E$,
	\begin{align*}
		(\partial_{\bn_e} v_h, p)_{e} \, \mathrm{d} s \quad \forall p \in \P_{\ell-2}(e).
	\end{align*} 
	\item[${\bf (D^{c}_4)}$] The internal moments of $v_h$ up to order $\ell-3$,
	\begin{align*}
		h^{-2}_{\E} ( v_h, p)_{\E} \quad \forall p \in \P_{\ell-3}(\E).
	\end{align*} 
\end{itemize}	
For further details the $C^1$-conforming VE scheme we refer to~\cite[Section 3.1]{Brezzi2013_CMAME} and \cite[Section 3.1]{Chinosi2016_CMA}.

Now, we continue with the construction of the Companion operators. Since this construction will be make in a unified way, we will use the superscript $\dagger \in \{fnc,nc\}$ to refer a generic definition or properties that is valid for both spaces $V^{fnc,h}_{k}$ or $V^{nc,h}_{k}$. Thus, with this notation, we start with the definition of a preliminary Companion operator $\widetilde{\mathcal{J}}^{\dagger}_h: V^{\dagger,h}_{k} \to V^{c,h}_{k+1} \subset V$, for $\dagger \in \{fnc,nc\}$. First, we note that the DoFs of $V^{nc,h}_{k}$ contains the DoFs of $V^{fnc,h}_{k}$, more precisely  $\big\{ \dof_i^{fnc,k}\big\}_{i=1}^{ d^{fnc}_k} \subset \big\{ \dof_i^{nc,k}\big\}_{i=1}^{ d^{nc}_k}$. Additionally,  we have $\big\{ \dof_i^{\dagger,k}\big\}_{i=1}^{ d^{fnc}_k} \subset \big\{ \dof_j^{c,k+1}\big\}_{j=1}^{ d^{c}_{k+1}}$, for $\dagger \in \{fnc,nc\}$. Then, based on \cite{Khot2025_MathComp}, we consider the following definition: for all $v_h \in V^{\dagger,h}_{k}$, we set
\begin{equation}\label{dof:Etilde:morley}
	\begin{split}
		\dof_i^{\dagger,k} (\widetilde{\mathcal{J}}_h^{\dagger} v_h) &:= 	\dof_i^{\dagger,k} ( v_h) \quad \forall i \in \{1, \ldots , d^{fnc}_k\};\\
		\nabla \widetilde{\mathcal{J}}_h^{\dagger} v_h(z) &:= \frac{1}{{\rm card}(\omega(z))} \sum_{\E \in \O_h}	\nabla \Pi_{\E}^{\dagger,\D} v_h(z) \quad \forall z \in \VVi;\\
		(\widetilde{\mathcal{J}}_h^{\dagger} v_h, \chi)_{\E} &:= (v_h, \chi)_{\E} \qquad \forall \chi \in \M^{*}_{k-3}(\E).
	\end{split}
\end{equation} 
where $\omega(z) := \{\E \in \O_h: z \in \E \}$  is the set that contains the neighboring polygons $\E$ sharing the vertex $z$, and ${\rm card}(\omega(z))$ denote its  cardinality.

We recall that the identity of the DoFs in the first equation of \eqref{dof:Etilde:morley} are taken for both spaces in the smallest set, i.e., for $i \in \{1, \ldots , d^{fnc}_k\}$. Thus, we deduce that for both VE nonconforming spaces, this Companion operator $\widetilde{\mathcal{J}}^{\dagger}_h: V^{\dagger,h}_{k} \to V^{c,h}_{k+1} \subset V$  is well-defined. Moreover, by following the same arguments developed in \cite[Theorem 5.1]{Khot2025_MathComp}, it is possible to prove that it satisfies the properties (a), (b) and (d) of Assumption~\ref{Assump:companion}.

The next step is to defined the Companion operator $\mathcal{J}^{\dagger}_h: V^{\dagger,h}_{k} \to  V$ satisfying also the $L^2$-orthogonality respect to $\P_k(\E)$ in property (c) of Assumption~\ref{Assump:companion}. To do that, we extend the ideas presented in \cite{Khot2025_MathComp}, i.e., from the definition of $\widetilde{\mathcal{J}}^{\dagger}_h$ in \eqref{dof:Etilde:morley}, we can consider its definition exactly as in \cite[Theorem 5.2]{Khot2025_MathComp}.

For sake of completeness we will write explicitly its construction. For each $\E \in \O_h$, let  $b_{\E} \in  H_0^2(\E)$  be a bubble-function supported in $\E$. Now, let us consider the Hilbert $\P_k(\E)$  endowed with the weighted scalar product $(b_{\E} \cdot, \cdot)_{0,\E}$. Moreover, let  $v_{\E} \in \P_k(\E)$ be the Riesz representative of the linear functional mapping from $\P_k(\E)$ to $\R$, defined by $q_k \mapsto (v_h-\Jh^{\dagger} v_h, q_k)_{0,\E}$, for all $q_k \in \P_k(\E)$. Thus, given $v_h \in V^{\dagger,h}_{k}$, the function $\widetilde{v}_h \in \P_k(\O_h)$, with $\widetilde{v}_h|_{\E}:= v_{\E}$ and the bubble-function  defined by $b_h|_{\E}:= b_{\E} \in H_0^2(\O)$, satisfy
\begin{equation*}
	(b_h \widetilde{v}_h,q_k)_{0,\O} = (v_h-\Jh^{\dagger} v_h, q_k)_{0,\O} \quad \forall q_k \in \P_k(\O_h). 
\end{equation*} 
Then, we define $\Jh^{\dagger}: V^{\dagger,h}_{k} \to  V$ by 
\begin{equation}\label{def:companion:final}
	\Jh^{\dagger} v_h :=   \widetilde{\mathcal{J}}^{\dagger}_h v_h + b_h \widetilde{v}_h \in  V.
\end{equation}

We summarize the previous discussion in the following result.

\begin{theorem}
	The Companion operator $\Jh^{\dagger}: V^{\dagger,h}_{k} \to  V$ (with  $\dagger \in \{fnc,nc\}$),  defined in  \eqref{def:companion:final} satisfies the properties (a)-(d) of Assumption~\ref{Assump:companion}.	
\end{theorem}
\begin{proof}
	By construction the proof is an extension of the arguments presented in~\cite[Theorems 5.1 and 5.2]{Khot2025_MathComp}.
	
\end{proof}

We finish this subsection with three remarks regarding the construction of the Companion operators.
\begin{remark}
	We recall that, in the construction of the operator $\widetilde{\mathcal{J}}^{nc}_h$ for the $C^0$-VE space $V^{nc,h}_{k}$, it is not required that all the moments (i.e., up to $k-2$) associated to the $ v_h$ and $ \widetilde{\mathcal{J}}_hv_h$  in the DoFs ${\bf (D^{nc}_2)}$ have to be identical. Instead, it suffices for these moments to agree for all polynomials of degree up to $k - 3$ (for this reason we consider $i \in \{1, \ldots , d^{fnc}_k\}$ instead of $i \in \{1, \ldots , d^{nc}_k\}$ in the first condition of \eqref{dof:Etilde:morley}). This fact allows us to follow a similar construction as in~\cite{Khot2025_MathComp} and ensures the validity of properties $(a)-(d)$ of  Assumption~\ref{Assump:companion}. 
\end{remark}		

\begin{remark}
	We note that the Companion operator introduced in~\cite{Khot2025_MathComp} was originally developed for the Morley-type VEM. However, in this work, we adapt these ideas to extend its construction to the $C^0$-nonconforming VE space and different boundary conditions. In this way, we provide here a unified approach for both families of nonconforming VEMs. Moreover, we observe that it is valid for any nonconforming VE scheme that employed the same DoFs. 
\end{remark}
\setcounter{equation}{0}
\setcounter{equation}{0}
\section{Applications}\label{SECTION:APPLICATIONS}

In this section we illustrate the scope of the proposed framework through several classes of fourth-order time-evolution problems arising in different application areas. Rather than focusing on individual models, we emphasize how a unified variational structure encompasses reaction--diffusion systems, phase-field dynamics, and formulations of incompressible flows based on stream functions. This highlights the flexibility of the analysis and its relevance across distinct scientific contexts.

\subsection{The extended Fisher--Kolmogorov equation}

A prototypical example within this class is the extended Fisher--Kolmogorov equation, which appears in models of pattern formation, population dynamics, and biological transport phenomena \cite{Danumjaya-Pani_2006,Danumjaya2021_CMA,Pei2023_CMA,al2024finite,DNR2025_JSC}. The equation combines fourth-order diffusion with lower-order reaction mechanisms, leading to rich dynamical behavior.

Let $\alpha_1,\alpha_2>0$ and $f(u)=u-u^3$. The problem reads: find $u(\bx,t)$ such that
\begin{equation}\label{EFK-eq}
\left\{
\begin{alignedat}{2}
	\partial_t u + \alpha_1 \Delta^2 u - \alpha_2 \Delta u &= f(u)
	&\qquad &\text{in } \Omega\times(0,T],\\
	u(\bx,0) &= u_0(\bx)
	&\qquad &\text{in } \Omega.
\end{alignedat}
\right.
\end{equation}
Typical boundary conditions include the clamped plate  conditions $\CP$
\[
u=\partial_{\bn}u=0 \quad \text{on } \Gamma\times(0,T],
\]
or the Navier conditions $\SSP$
\[
u=\Delta u=0 \quad \text{on } \Gamma\times(0,T].
\]

\subsection{The semilinear system with Cahn--Hilliard-type BCs}

When the same fourth-order operator is complemented with Cahn--Hilliard-type BCs $\CH$ (cf. \eqref{Cahn-Hilliard:BCs}),
\[
\partial_{\bn}u=\partial_{\bn}\Delta u=0 \quad \text{on } \Gamma\times(0,T],
\]
the resulting problem is closely related to models arising in phase separation processes in binary mixtures and polymeric materials \cite{Gudi-Gupta_camwa2013,Brenner2012_SINUM,Antonietti2016_SINUM}. In this context, the fourth-order diffusion reflects mass-conserving mechanisms, while the nonlinear reaction terms represent thermodynamic driving forces.

Both reaction--diffusion and phase-field systems fit naturally within the same abstract formulation. More precisely, recalling that $V$ denotes the energy space associated with the prescribed boundary conditions (cf.~\eqref{cont:space:V}), the weak formulation reads: seek
\[
u\in L^2(0,T;V),
\]
such that for a.e. $t\in(0,T]$
\begin{equation*}
\left\{
\begin{aligned}
M(\partial_t u,v)+\alpha_1 A(u,v)+\alpha_2 B(u,v) &= F(u;v)  \qquad \forall v\in V,\\
u(0)&=u_0.
\end{aligned}
\right.
\end{equation*}

\subsection{Fluid flow problems in stream function formulation}
In $2$D incompressible fluid flow, the velocity field \(\boldsymbol{u} = (u_1, u_2)\) can expressed using the stream function \(\psi(\bx, t)\), where:
\begin{equation}\label{stream}
	u_1 = \frac{\partial \psi}{\partial y}, \quad u_2 = -\frac{\partial \psi}{\partial x}.
\end{equation}

We observe that one of most important features of this formulation is that  the incompressibility condition \(\div \boldsymbol{u} = 0\) is automatically satisfied by definition~\cite{GR}.

By exploiting the identity \eqref{stream}, the Stokes system in the pure stream-function formulation reduces to the following fourth-order evolution equation for $\psi$
\[
-\partial_t \Delta \psi + \nu \Delta^2 \psi =  \widetilde{f} \quad \text{in } \Omega \times (0, T],
\]
subject to the BCs of the first kind (cf.~\eqref{Campled:BCs}) and the initial conditions
\[
\begin{aligned}
	\psi = \dn \psi &= 0 \quad \quad\quad \text{on } \Gamma \times (0, T], \\
	\psi(\bx, 0) &= \psi_0(\bx) \quad \text{in } \Omega,
\end{aligned}
\]
Here, $\nu>0$ denotes the kinematic viscosity, and $\widetilde{f}:=\rot \boldsymbol{f}$, where $\boldsymbol{f}$ is a sufficiently regular external force. We note that, in this formulation, the right-hand side $\widetilde{f}$ is independent of the unknown. For further details, we refer to~\cite{GR}.

We have that a variational formulation of this problem read as: find $\psi \in L^2(0,T;V)$, such that 
\begin{equation*}
\left\{
\begin{aligned}
		B(\partial_t \psi,v)+ \nu A(\psi,v)&=(\widetilde{f},v)_{\O} \qquad \forall v \in V,\\
		\psi(0) &= \psi_0.	
\end{aligned}
\right.
\end{equation*}

The analysis developed in this work can be applied to this system, with appropriate treatment for the time derivative now in the  bilinear form $B$ and the lineal functional $\widetilde{f}$.  

\setcounter{equation}{0}
\setcounter{equation}{0}
\section{Numerical Experiments}\label{SECTION:NUMERICAL:EXPERIMENTS}
In this section, we present three numerical experiments to validate the theoretical results and assess the practical performance of the proposed VEMs.  The results demonstrate the effectiveness of the method in approximating solutions of the extended Fisher--Kolmogorov equation under clamped and Navier boundary conditions, as well as of the semilinear problem with Cahn--Hilliard-type boundary conditions.

Moreover, the third example is specifically designed to assess the accuracy of the method by considering an analytical solution with lower regularity on a non-convex $\Gamma$-shaped domain. This choice allows us to further justify our new findings in Theorems~\ref{converg:semi:discre:schm} and \ref{converg:fully:discre:schm}.
\subsection{Some aspects of the numerical implementation}

\paragraph{General setting and nonlinear solver.}
All numerical experiments are performed in \textsc{Matlab} using the lowest-order fully nonconforming Morley-type virtual element method ($k=2$).  Moreover, we consider four families of polygonal meshes:
\begin{itemize}
	\item $\O^1_h$: triangular meshes;
	\item $\O^2_h$: distorted quadrilateral meshes;
	\item $\O^3_h$: concave quadrilateral meshes;
	\item $\O^4_h$: centroidal Voronoi tessellations.
\end{itemize}

\begin{figure}[H]
	\begin{center} \includegraphics[height=4.2cm, width=4.2cm]{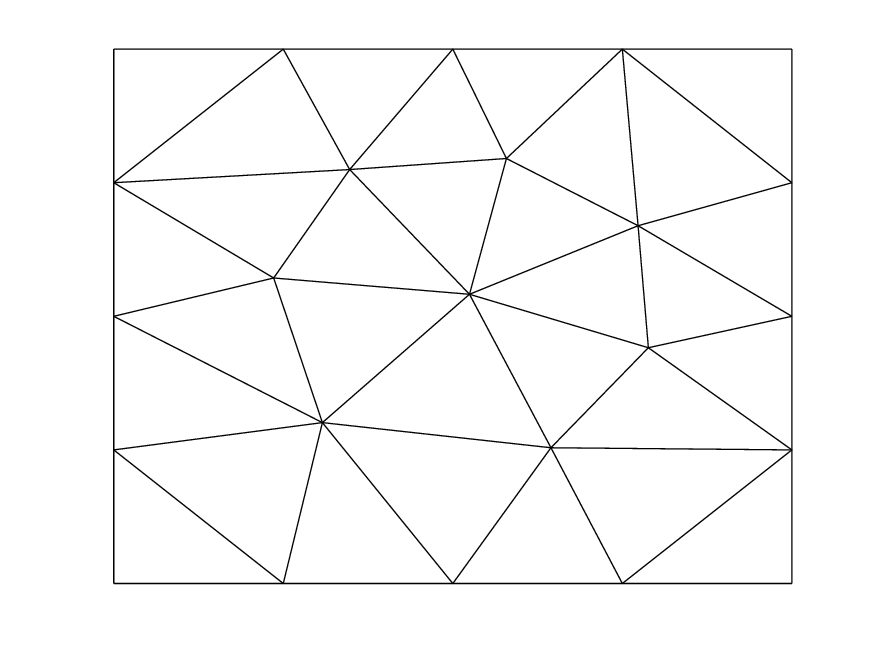}\hspace*{-0.40cm} \includegraphics[height=4.2cm, width=4.2cm]{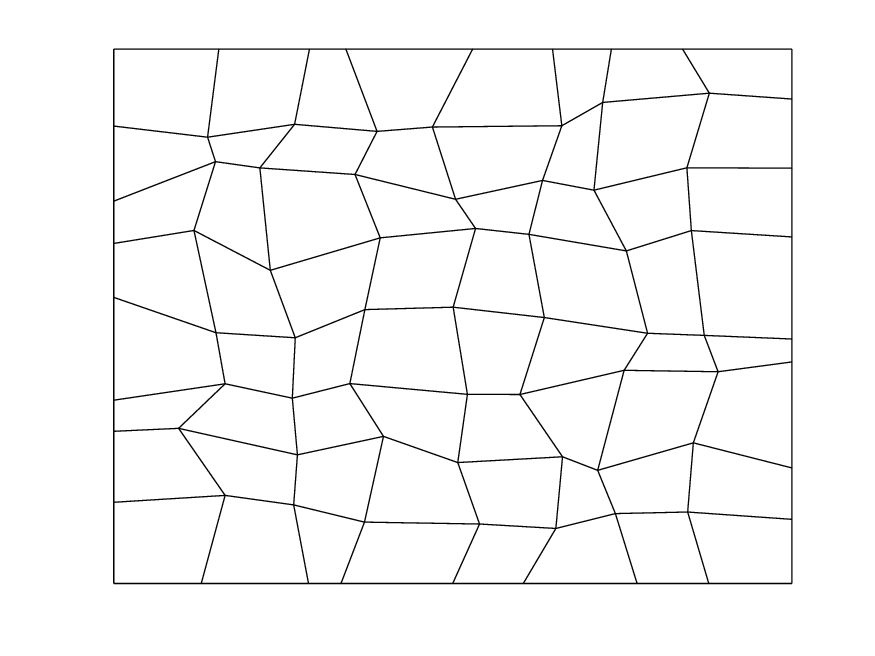}\hspace*{-0.40cm} \includegraphics[height=4.2cm, width=4.2cm]{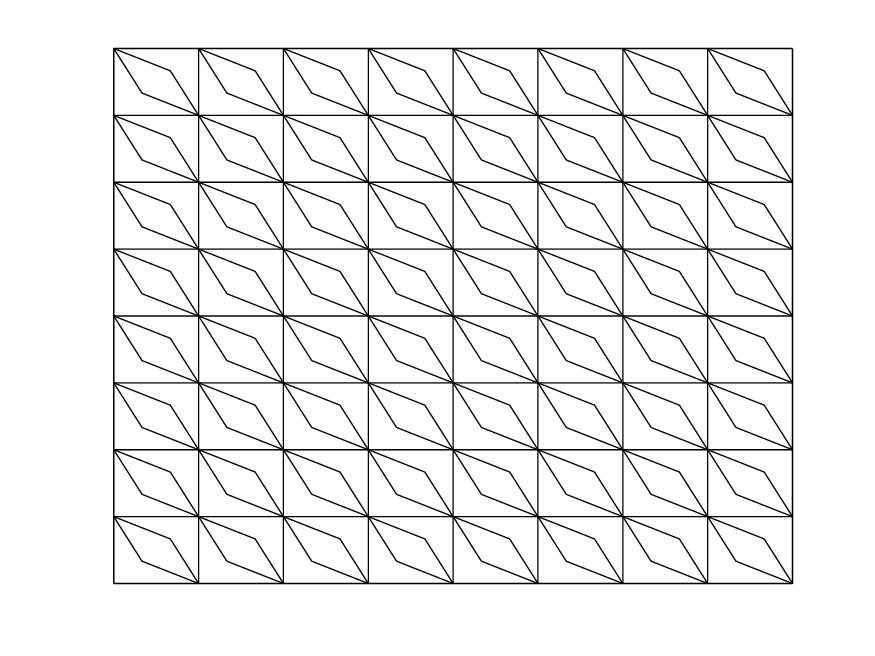}\hspace*{-0.40cm} \includegraphics[height=4.2cm, width=4.2cm]{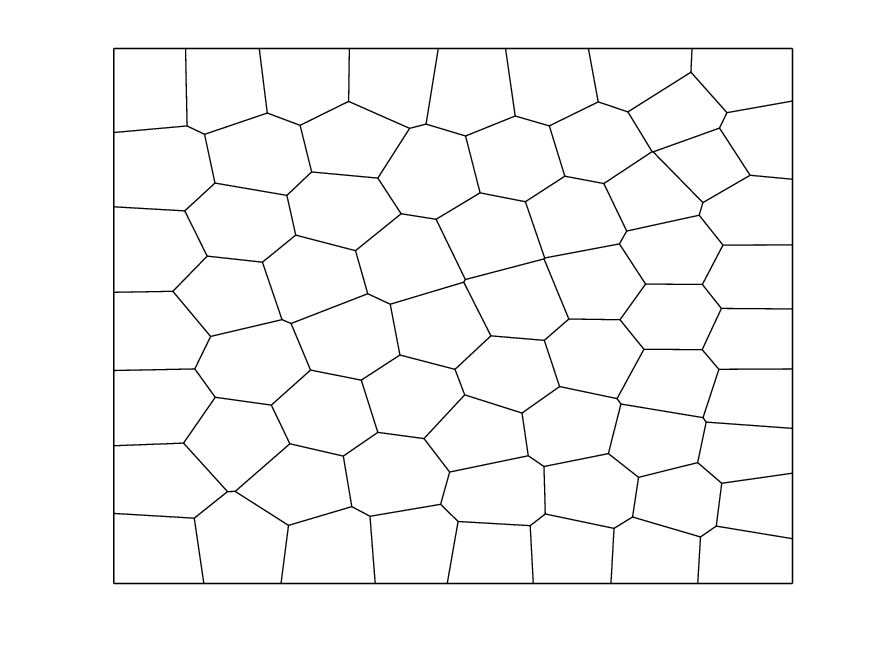}\hspace*{-0.40cm} 
	\end{center} \caption{Example of the fourth families of meshes used: $\O^1_h$, $\O^2_h$, $\O^3_h$ and $\O^4_h$.} 
\end{figure}

Since the virtual element solution is not explicitly available inside each element, we employ the reconstructed function
\[
\widetilde{u}_h^n := \Pi_{h}^{k,\nabla^2} u_h^n .
\]
The error is measured in the broken $\ell^2(H^2)$-norm (cf.~\eqref{def:norm:time}) as
\begin{equation}\label{error:quant}
	e_{2}(u):=\| u - \widetilde{u}_h \|_{\ell^{2}(0,T;H^2(\O_h))}.
\end{equation}

At each time step, the nonlinear fully discrete scheme~\eqref{fully:dis:schm} is solved by Newton’s method. 
For the initial step we set $u_h^{\mathrm{in}}=0$, while for $n\ge1$ we use $u_h^{\mathrm{in}}=u_h^{n-1}$. 
The iterations are stopped when the $\ell^\infty$-norm of the global increment falls below ${\rm Tol}=10^{-8}$.

Following~\eqref{mat:semidiscrete}, the nonlinear residual reads: find ${\bf U}^n\in\mathbb{R}^{ndof}$ such that
\begin{equation*}
	(\mathbf{M}+\mathbf{A}+\mathbf{B}){\bf U}^n
	-\Delta t\,\mathbf{F}({\bf U}^n)
	-\Delta t\,\mathbf{M}{\bf U}^{n-1}=0,
\end{equation*}
where the matrices and nonlinear vector are defined in~\eqref{mat:semidiscrete}.  

The Jacobian matrix $JC=(jc_{ij})$ is given by
\begin{equation*}
	jc_{ij}
	=
	M_h(\phi_i,\phi_j)+A_h(\phi_i,\phi_j)+B_h(\phi_i,\phi_j)
	-\Delta t\,\frac{\partial f_h({\bf U}^n)}{\partial {\bf U}^n}
	(\Pi^k \phi_i,\Pi^k \phi_j)_{\O},
\end{equation*}
with the bilinear forms introduced in Section~\ref{polynomialproy}.

\paragraph{Treatment of the different boundary conditions.}
A key feature of the present approach is its unified and versatile nature. The local virtual element construction, together with the associated element matrices, is completely independent of the boundary conditions under consideration. Differences between the ${\bf CP}$, ${\bf NC}$, and ${\bf CH}$ cases arise solely at the global level, through the enforcement of boundary degrees of freedom and the corresponding constraints (see Remark~\ref{remark:Dofs:BCs}). Consequently, switching between boundary conditions does not require any modification of the local formulation or element routines, resulting in a flexible and unified implementation.

\subsection{Test 1. The Extended Fisher--Kolmogorov equation with different BCs}		
In this numerical test, we consider the Extended Fisher--Kolmogorov equation~\eqref{EFK-eq}, with the $\CP$ and $\SSP$ boundary conditions~\eqref{Campled:BCs} and~\eqref{supported:BCs} on the square domain $\Omega := (0,1)^{2}$. We take the load term, boundary and initial conditions in such a way that the analytical solution is given by
\begin{equation*}
	u(\bx,t)=\sin(t)\: x^6 y^6(1-x)^6(1-y)^6. 
\end{equation*}
Moreover, we consider the time interval $I=(0,1]$, and  $\alpha_1 =\alpha_2 =1$.

Tables~\ref{table1-1} and~\ref{table1-2} report the errors $e_2(u)$ obtained for the model problem~\eqref{EFK-eq} under the $\CP$ and $\SSP$ boundary conditions, respectively. For the $\CP$ case, computations are carried out on the mesh families $\O_h^1$ and $\O_h^2$, whereas for the $\SSP$ boundary conditions the meshes $\O_h^1$ and $\O_h^3$ are employed. These results illustrate the performance of the numerical scheme under different mesh configurations and boundary treatments.

To verify the convergence rate predicted by Theorem~\ref{converg:fully:discre:schm}, a simultaneous refinement in space and time is performed. Starting from $h_0=\Delta t_0=1/4$, both the mesh size and time step are successively halved at each refinement level.

As observed along the main diagonal of Tables~\ref{table1-1} and~\ref{table1-2}, the error $e_2(u)$ exhibits a linear decay consistent with the combined convergence order $\mathcal O(h+\Delta t)$, thereby confirming the theoretical estimate for $k=2$.
Moreover, for sufficiently small time steps, the error is dominated by the spatial discretization. In particular, the values reported in Tables~\ref{table1-1} and~\ref{table1-2} show that $e_2(u)$ remains essentially insensitive to further reductions in $\Delta t$ when $h$ is fixed, even for relatively coarse meshes.

\begin{table}[H] 
	\begin{center}
		{\small \begin{tabular}{rrrccccccccc}
				\hline
				\hline\noalign{\smallskip}
				&	&dof &$h$ &$\Delta t_0$ & $\Delta t_0/2$ & $\Delta t_0/4$& $\Delta t_0/8$ &$\Delta t_0/16 $ \\
				\hline  
				\hline
				
				& 	     &45    &$h_0$   &$\fbox{7.716336e-7}$ &7.194463e-7&6.929311e-7 &6.795427e-7 &$\fbox{6.728136e-7}$ \\
				& 	     &233   &$h_0/2$ & 4.932702e-7         &$\fbox{4.599104e-7}$    &4.344284e-7 &4.429609e-7 &$\fbox{4.301467e-7}$\\
				&$\O^1_h$&1005  &$h_0/4$ & 2.716347e-7         &2.532643e-7&$\fbox{2.439306e-7}$ &2.392870e-7 &$\fbox{2.368963e-7}$\\
				&		 &4197  &$h_0/8$ & 1.354426e-7         &1.262827e-6&1.216287e-6 &$\fbox{1.192804e-7}$ &$\fbox{1.181197e-7}$\\
				&		 &16953 &$h_0/8$ & 6.944831e-8         &6.475160e-8&6.236529e-8 &6.116119e-7 &$\fbox{6.055620e-8}$\\
				\\
				&	      &33   &$h_0$   &$\fbox{6.350002e-6}$ &5.920526e-7&5.702613e-7 &5.592401e-7 &$\fbox{5.537154e-7}$ \\
				&		  &161  &$h_0/2$ &3.414177e-7 &$\fbox{3.183275e-7}$&3.065958e-7 &3.007228e-7 &$\fbox{2.977844e-7}$\\
				&$\O^2_h$ &705  &$h_0/4$ &1.722849e-7 &1.606333e-7&$\fbox{1.547134e-7}$ &1.518427e-7 &$\fbox{1.502989e-7}$\\
				&		  &2945 &$h_0/8$ &7.873970e-8 &7.341457e-8&7.070898e-8 &$\fbox{6.934377e-8}$ &$\fbox{6.869182e-8}$\\
				&		  &12033&$h_0/16$&3.836253e-8 &3.576814e-8&3.444997e-8 &3.378485e-8 & $\fbox{3.342653e-8}$\\   		
				\hline 				
				\hline
		\end{tabular}}
	\end{center}
	\caption{Test~1. Errors in the discrete $e_2(u)$ norm obtained with $k=2$, $\alpha_1=\alpha_2=1$ and $\CP$ BCs.}
	\label{table1-1}
\end{table}

\begin{table}[H] 
	\begin{center}
		{\small \begin{tabular}{crrrccccccccc}
				\hline
				\hline\noalign{\smallskip}
				&	&dof &$h$ &$\Delta t_0$ & $\Delta t_0/2$ & $\Delta t_0/4$& $\Delta t_0/8$ &$\Delta t_0/16 $ \\
				\hline  
				\hline
				& 	     &61    &$h_0$   &$\fbox{7.756268e-7}$ &7.231695e-7&6.965170e-7 &6.830591e-7 &$\fbox{6.762949e-7}$ \\
				& 	     &261   &$h_0/2$ &4.933406e-7 &$\fbox{4.599760e-7}$&4.430242e-7 &4.344904e-7 &$\fbox{4.302081e-7}$\\
				&$\O^1_h$&1057  &$h_0/4$ &2.716450e-7 &2.532738e-7&$\fbox{2.439398e-7}$ &2.392961e-7 &$\fbox{2.369053e-7}$\\
				&		 &4297  &$h_0/8$ &1.354432e-7 &1.262833e-7&1.216293e-7          &$\fbox{1.192810e-7}$ &$\fbox{1.181203e-7}$\\
				&	   	 &17153 &$h_0/16$&6.944873e-8 &6.475204e-8&6.236574e-8          &6.116164e-8 & $\fbox{6.055666e-8}$\\   		
				\\			
				&	      &37   &$h_0$   &$\fbox{6.788983e-7}$ &6.329825e-7&6.096485e-7 &5.978788e-7 &$\fbox{5.919611e-7}$ \\
				&      	  &145  &$h_0/2$ &5.608125e-7 &$\fbox{5.228843e-7}$&5.036138e-7 &4.939017e-7 &$\fbox{4.890251e-7}$\\
				& $\O^3_h$&577  &$h_0/4$ &3.059007e-7 &2.852125e-7&$\fbox{2.747013e-7}$ &2.694531e-7 &$\fbox{2.667678e-7}$\\
				&		  &2305 &$h_0/8$ &1.568858e-7 &1.462757e-7&1.408848e-7 &$\fbox{1.381647e-7}$ &$\fbox{1.368144e-7}$\\
				&		  &9217&$h_0/16$ &7.888526e-8 &7.355031e-8&7.083973e-8 &6.947201e-8 &        $\fbox{6.878481e-8}$\\   		
				\hline  
				\hline
		\end{tabular}}
	\end{center}
	\caption{Test~1. Errors in the discrete $e_2(u)$ norm obtained with $k=2$, $\alpha_1=\alpha_2=1$ and $\SSP$ BCs.}
	\label{table1-2}
\end{table}		

Figure~\ref{fig:snapshops:Test1} displays snapshots of the exact and numerical solutions at the time instants $t_n = 0.25$, $0.5$, and $1$ for the Extended Fisher–Kolmogorov equation with $\CP$ boundary conditions. The results are obtained using the fully-discrete VE scheme~\eqref{fully:dis:schm} with $\alpha_1=\alpha_2=1$, on the mesh $\O_h^{2}$ with mesh size $h=1/32$ and time step $\Delta t = 10^{-2}$.
\begin{figure}[h!]
	\centering
\subfloat[$u$ at $t_n=0.25$]{\includegraphics[width=0.3\linewidth, height=0.25\textwidth]{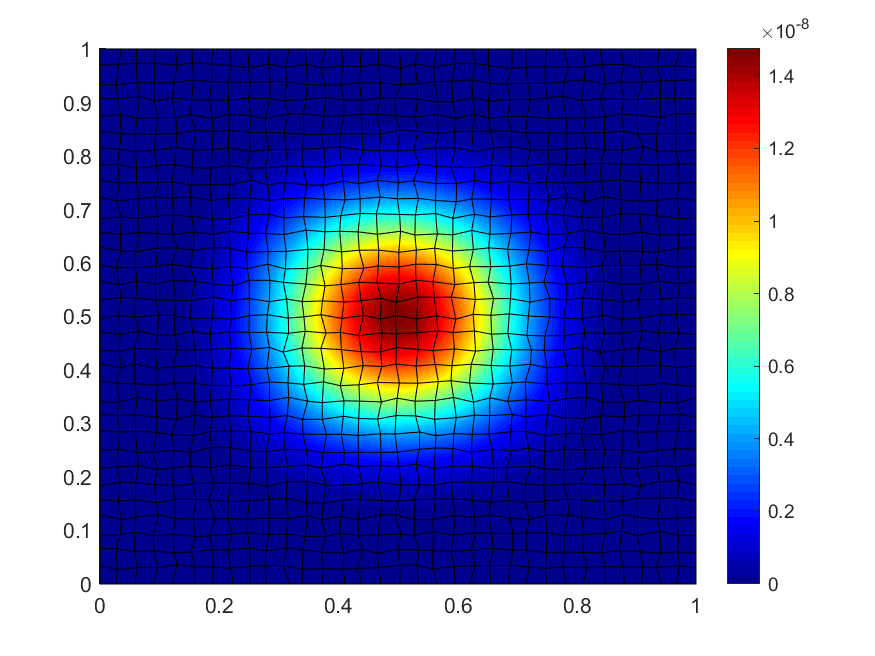}}\hfill
\subfloat[$u$ at $t_n=0.5$]{\includegraphics[width=0.3\linewidth, height=0.25\textwidth]{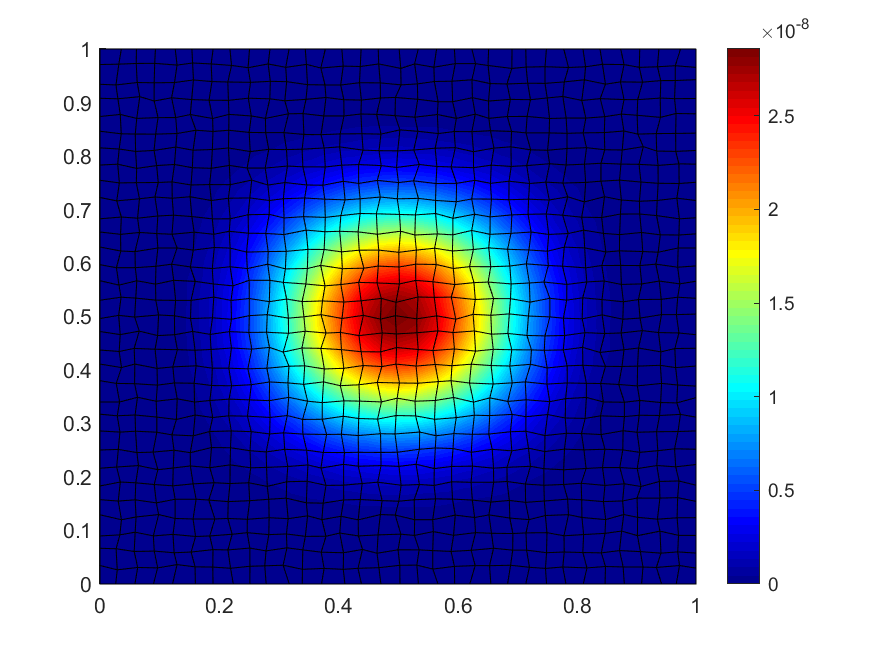}}\hfill
\subfloat[$u$ at $t_n=1$]{\includegraphics[width=0.3\linewidth, height=0.25\textwidth]{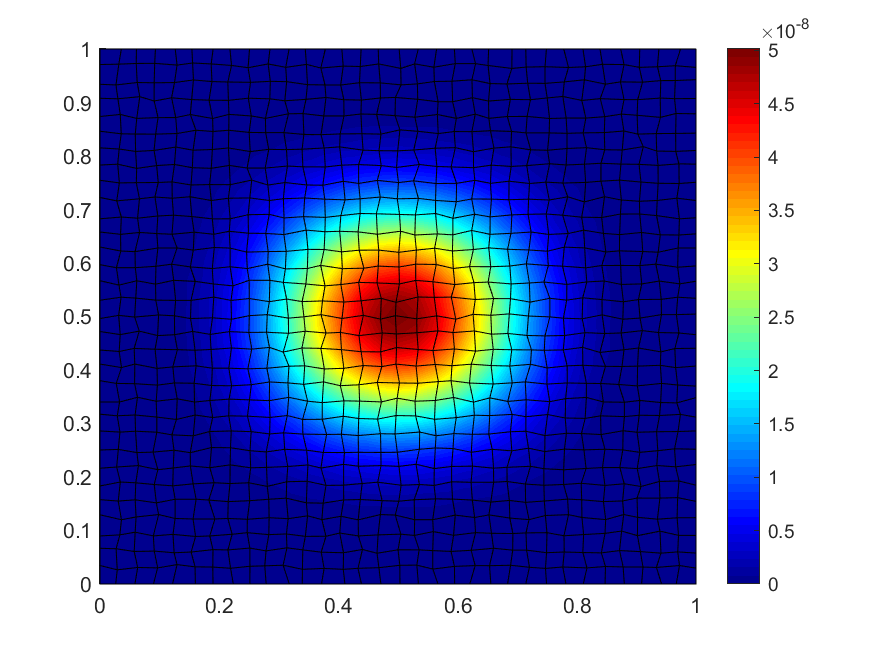}}\\	
\subfloat[$u^n_h$ at $t_n=0.25$]{\includegraphics[width=0.3\linewidth,height=0.25\textwidth]{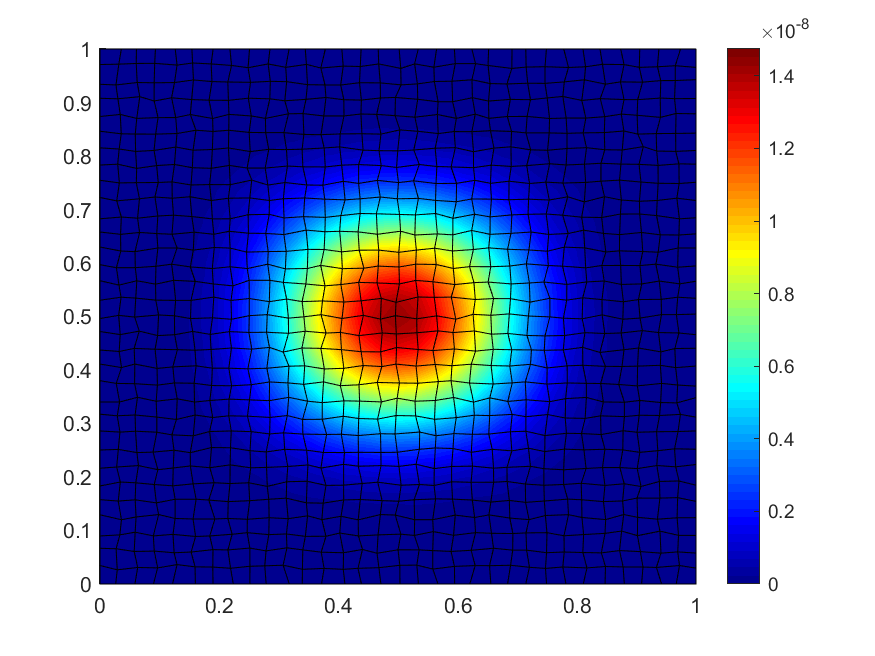}}\hfill
\subfloat[$u^n_h$ at $t_n=0.5$]{\includegraphics[width=0.3\linewidth,height=0.25\textwidth]{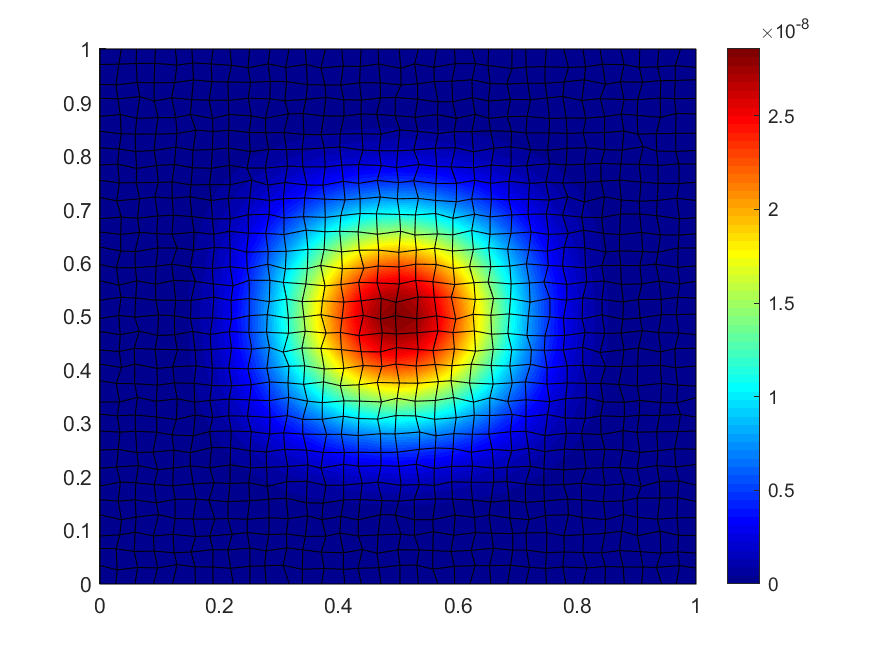}}\hfill
	\subfloat[$u^n_h$ at $t_n=1$]{\includegraphics[width=0.3\linewidth, height=0.25\textwidth]{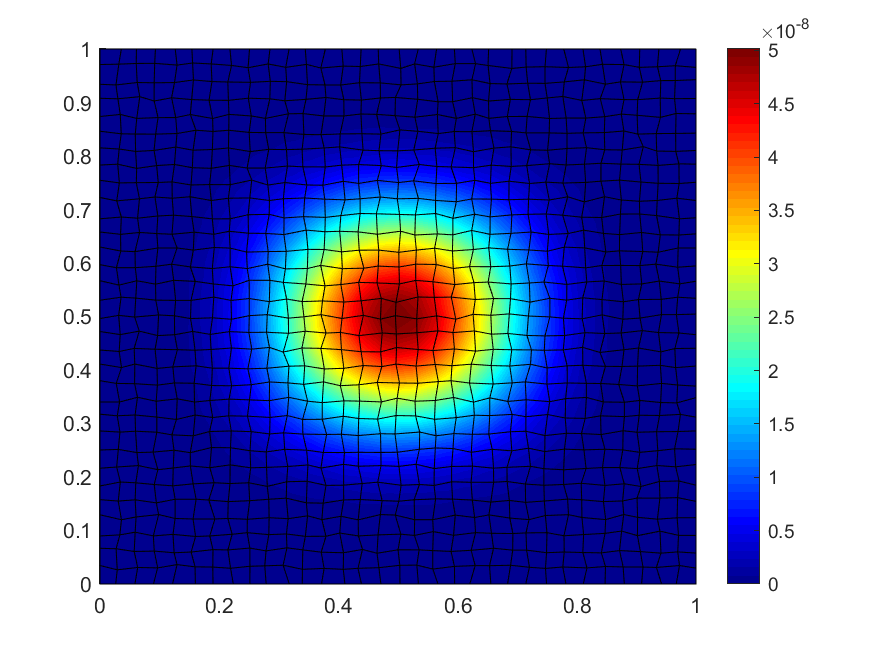}}
	\caption{Test 1. Snapshots of the exact and numerical solutions at different time instants for the Extended Fisher–Kolmogorov equation with $\CP$ BCs, using the fully-discrete VE scheme~\eqref{fully:dis:schm}, $\alpha_1=\alpha_2=1$, and mesh $\O^1_h$, with $h=1/32$ and  $\Delta=10^{-2}$.}
	\label{fig:snapshops:Test1}
\end{figure}

\subsection{Test 2. The semilinear system with Cahn--Hilliard BCs}
In this second experiment, we consider the model problem \eqref{EFK-eq} on the square domain 
$\Omega := (0,1)^{2}$ with time interval $I=(0,2]$. We take the load term and the initial conditions  in such a way that the analytical solution is given  as in Test~1.

In Table~\ref{table2-1}, we report the errors $e_2(u)$ obtained with the scheme~\eqref{fully:dis:schm} and using the polygonal meshes $\O^2_h$ and $\O^4_h$. For this boundary condition, the results also exhibit a consistent behavior across both meshes, confirming the expected linear convergence predicted by Theorem~\ref{converg:fully:discre:schm}, which highlights the accuracy of our method. 
We apply a simultaneous refinement in both time and space variables. For each mesh family, we begin with an initial mesh size and time step of $h_0=\Delta t_0 =1/4$. Then, at each iteration, we refine the discretization by reducing both the spatial and temporal step sizes by half.

As observed along the diagonal of Table~\ref{table2-1}, the error $e_{2}(u)$ decreases linearly under simultaneous refinement of the spatial mesh size $h$ and the time step $\Delta t$. This behavior is consistent with the predicted convergence rate of order $\mathcal O(h+\Delta t)$ for $k=2$, thereby confirming the theoretical estimate of Theorem~\ref{converg:fully:discre:schm}.
\begin{table}[H] 
	\begin{center}
		{\small \begin{tabular}{crrrccccccccc}
				\hline
				\hline\noalign{\smallskip}
				&	&dof &$h$ &$\Delta t_0$ & $\Delta t_0/2$ & $\Delta t_0/4$& $\Delta t_0/8$ &$\Delta t_0/16 $ \\
				\hline  
				\hline
				&	      &49   &$h_0$   &$\fbox{1.249185e-6}$ &1.206800e-6&1.206800e-6 &1.183120e-6 &$\fbox{1.169990e-6}$ \\
				&      	  &193  &$h_0/2$ &6.637615e-7 &$\fbox{6.413787e-7}$&6.290343e-7 &6.224309e-7 &$\fbox{6.188062e-7}$\\
				& $\O^2_h$&769  &$h_0/4$ &3.346403e-7 &3.233576e-7&$\fbox{3.171396e-7}$ &3.138301e-7 &$\fbox{3.121217e-7}$\\
				&		  &3073 &$h_0/8$ &1.529060e-7 &1.477540e-7&1.449247e-7 &$\fbox{1.434575e-7}$ &$\fbox{1.428226e-7}$\\
				&	      &12289&$h_0/16$&7.449027e-8 &7.198374e-8&7.062268e-8 &6.998415e-8          &$\fbox{6.898833e-8}$\\   		
				\\			
				& 	     &67    &$h_0$   &$\fbox{1.061331e-6}$ &1.025607e-6&1.005996e-6         &9.957346e-7         &$\fbox{9.903387e-7}$ \\
				& 	     &291   &$h_0/2$ &5.630875e-7 &$\fbox{5.441219e-7}$&5.336944e-7         &5.281830e-7         &$\fbox{5.253043e-7}$\\
				&$\O^4_h$&1226  &$h_0/4$ &2.749379e-7 &2.656807e-7         &$\fbox{2.605982e-7}$&2.579360e-7         &$\fbox{2.566174e-7}$\\
				&		 &5005  &$h_0/8$ &1.426001e-7 &1.378002e-7         &1.351724e-7         &$\fbox{1.338292e-7}$&$\fbox{1.333015e-7}$\\
				&	   	 &20247 &$h_0/16$&7.051774e-8 &6.811375e-8         &6.681314e-8         &6.521678e-8     &$\fbox{6.427102e-8}$\\  
				\hline  
				\hline
		\end{tabular}}
	\end{center}
	\caption{Test~2. Errors in the discrete $e_2(u)$ norm obtained with $k=2$, $\alpha_1=\alpha_2=1$ and $\CH$ BCs.}
	\label{table2-1}
\end{table}

\subsection{Test 3: Non-smooth solution on an $\Gamma$-shaped domain}
In this numerical test, we are interested in examining the accuracy of the discrete scheme \eqref{fully:dis:schm} with an exact solution having less regularity (in space) on a nonconvex $\Gamma$-shaped domain to justify our new finding in
Section~\ref{SECTION:ERROR:UNIFIED} for small values $\alpha$, i,e. $\alpha \in (0,1/2)$. 

We consider the time-dependent fourth order problem~\eqref{EFK-eq} with $\alpha_1=1$, $\alpha_2=0$ and with the $\CP$  BCs \eqref{Campled:BCs}  on the  $\Gamma$-shaped domain $\Omega_{\Gamma} :=(-1,1)^2 \setminus \left([0,1) \times (-1,0]\right)$, on the time interval $I=(0,1]$. We consider the load term $g$, boundary and initial conditions in such a way that the analytical solution is given by
\begin{equation*}
	u(\bx,t)=\sin(t)\:  r^{4/3}\sin\left(\frac{4\theta}{3}\right),
\end{equation*}
where $(r, \theta)$ are the polar coordinates.	We observe that this solution presents
a singularity at the re-entrant corner of $\Omega_{\Gamma}$ (the origin $(0,0)$); we have $u(\bx,t)  \in H^{7/3 - \varepsilon}(\Omega )$ for all  $\varepsilon >0$, and for all $t \in I$. 	

In order to see the sub-linear trend of the error~$e_2(u)$ \eqref{error:quant}, predicted by Theorem~\ref{converg:fully:discre:schm}, we fixed the time step $\Delta t =0.01$ and refine in space. In Table~\ref{table3-1}, we report the results obtained by using a mesh with distorted squares elements (as in $\Omega^2_h$) for $k=2$, with mesh refinements $h = 1/4, 1/8, 1/16, 1/32, 1/64$. According to the regularity of solution $u$, we expect an order of convergence in the energy norm  as
$\mathcal{O}(h^{1/3})$  which is predicted by Theorem~\ref{converg:fully:discre:schm}.
\begin{table}[h!]
	\setlength{\tabcolsep}{5.2pt}
	\begin{center}
		{\begin{tabular}{crrccc}
				\hline
				&$h$ &dofs &${e}_2(u)$ &${\tt r}_2(u)$ \\
				\hline  
				\hline
				&1/2   &21   &7.0023e-3 &--- \\
				&1/4   &113  &5.4869e-3 &0.35\\
				$\Omega^2_{h}$	&1/8   &513  &4.2896e-3 &0.35\\
				&1/16  &2177 &3.3852e-3 &0.34\\
				&1/32  &8961 &2.6822e-3 &0.33\\ 
				&1/64  &36353&2.1279e-3 &0.33 \\  
				\hline			
		\end{tabular}}
		\caption{Test~3. Errors in the discrete $e_2(u)$ norm obtained with $k=2$, $\alpha_1=1$ and $\alpha_2=0$ and $\CP$ BCs on the $\Gamma$-shaped domain.}
		\label{table3-1}
	\end{center}
\end{table}
\setcounter{equation}{0}
\setcounter{equation}{0}	
\section{Concluding remarks}\label{SECTION:CONCLUSIONS}
In this work, we have developed a novel unified framework for the analysis of high-order nonconforming virtual element methods applied to fourth-order nonlinear reaction–diffusion problems with clamped, Navier, and Cahn--Hilliard-type boundary conditions. The proposed approach encompasses $C^0$-nonconforming and Morley-type VEMs, and is applicable for arbitrary polynomial degrees $k \geq 2$.

A key contribution of our study is the introduction of a class of Companion operators, which enables the construction of new Ritz-type projections. This allows us to establish optimal error estimates under minimal spatial regularity assumptions on the weak solution, even in non-convex domains and for problems with different boundary conditions. Such low-regularity extend the applicability of NCVEM for fourth-order problems beyond what has been reported in the existing literature.

The theoretical findings are supported by numerical experiments on polygonal meshes, which confirm the optimal convergence rates and the robustness of the proposed schemes for the three types of boundary conditions. The versatility and accuracy demonstrated suggest that this framework can be successfully applied to a broad class of fourth order nonlinear problems.
Finally, we emphasize that the proposed analysis can be extended to more complex nonlinear time-dependent fourth-order problems. In particular, considering the Cahn--Hilliard equation or the Navier–Stokes equations in stream function formulation would constitute interesting and challenging directions for future research.


\small
\section*{Acknowledgments}
DA was partially supported by the National funding agency,  Anusandhan National Research Foundation for Research and Development in India, through project No : RJF/ 2025/000114. 

The research of DM was partially supported by the National Agency
for Research and Development, ANID-Chile through FONDECYT project 1220881,
and by project Centro de Modelamiento Matem\'atico (CMM), FB210005, BASAL funds for centers of excellence.
\setcounter{equation}{0}


\begin{thebibliography}{99}
		
	\bibitem{Adak2023_M3AS}
	\textsc{D. Adak, V. Anaya, M. Bendahmane, and D. Mora}, 
	\textit{Conforming and nonconforming virtual element methods for fourth-order nonlocal reaction diffusion equation}, 
	Math. Models Methods Appl. Sci., {\bf 33} (2023), no.~10, pp.~2035--2083.
		
	
	\bibitem{Adak2024_CMAME}
	\textsc{D. Adak, D. Mora, and A. Silgado},
	\textit{The Morley-type virtual element method for the Navier--Stokes equations in stream-function form}, 
	Comput. Methods Appl. Mech. Engrg., {\bf 419} (2024), p. 116573.
	
	\bibitem{Adak2019_NMPDE}
	\textsc{D. Adak, E. Natarajan, and S. Kumar},
	\textit{Convergence analysis of virtual element methods for semilinear parabolic problems on polygonal meshes}, 
	Numer. Methods Partial Differ. Equ., {\bf 35} (2019), pp. 222–245.
	
	\bibitem{Adams2003_Sobolev}
	\textsc{R. A. Adams and J. J. Fournier}, 
	\textit{Sobolev Spaces}, Elsevier, (2003).	
	
	\bibitem{Ahmad2013_CMA}
	\textsc{B. Ahmad, A. Alsaedi, F. Brezzi, L. D. Marini, and A. Russo}, 
	\textit{Equivalent projectors for virtual element methods}, 
	Comput. Math. Appl., {\bf 66} (2013), pp. 376–391.
	
		\bibitem{al2024finite}
	\textsc{Al-Musawi, A. Ghufran, and A J Harfash},
	\textit{Finite element analysis of extended {Fisher-Kolmogorov} equation with Neumann boundary conditions}, Appl. Numer. Math., \textbf{201}, (2024), pp. 41--71.
	
	
	\bibitem{Book_VEM2022}
	\textsc{P.F. Antonietti, L. Beir\~ao da Veiga and G. Manzini},
	The Virtual Element Method and	its Applications, SEMA SIMAI Springer Series,
	Springer, Cham, Vol. 31, 2022.  
	
	
	\bibitem{Antonietti2016_SINUM}
	\textsc{P. F. Antonietti, L. Beir\~ao da Veiga, S. Scacchi, and M. Verani}, 
	\textit{A $C^1$ virtual element method for the Cahn-Hilliard equation with polygonal meshes}, 
	SIAM J. Numer. Anal., {\bf 54} (2016), pp. 34–56.
		
	\bibitem{Antonietti2018_M3AS}
	\textsc{P. F. Antonietti, G. Manzini, and M. Verani}, 
	\textit{The fully nonconforming virtual element method for biharmonic problems}, 
	Math. Models Methods Appl. Sci., {\bf 28} (2018), pp. 387–407.
	
	
	\bibitem{ALKM2016}
	\textsc{B. Ayuso de Dios, K. Lipnikov and G. Manzini},
	\textit{The nonconforming virtual element method.}
	ESAIM Math. Model. Numer. Anal., {\bf 50}(3) (2016), pp. 879--904.
	
	
	\bibitem{Beirao2013_M3AS}
	\textsc{L. Beir\~ao da Veiga, F. Brezzi, A. Cangiani, G. Manzini, L. D. Marini, and A. Russo}, 
	\textit{Basic principles of virtual element methods}, 
	Math. Models Methods Appl. Sci., {\bf 23} (2013), pp. 199–214.
	
	
	\bibitem{Beirao2023_ActaNumer}
	\textsc{L. Beirão da Veiga, F. Brezzi, L. D. Marini, and A. Russo}, 
	\textit{The virtual element method}, 
	Acta Numer., {\bf 32} (2023), pp. 123--202.
	
	
		\bibitem{BR80}
	\textsc{H. Blum and R. Rannacher},
	\textit{On the boundary value problem of the
		biharmonic operator on domains with angular corners},
	Math. Methods Appl. Sci., \textbf{2}(4) (1980), pp. 556--581.
	
	
	\bibitem{Brenner2012_SINUM}
	\textsc{S. C. Brenner, S. Gu, T. Gudi, and L.-Y. Sung}, 
	\textit{A quadratic $C^0$ interior penalty method for linear fourth-order boundary value problems with boundary conditions of the Cahn--Hilliard type}, 
	SIAM J. Numer. Anal., {\bf 50} (2012), pp.~2088--2110.
	
	
	
	\bibitem{Brenner-Monk-Sun2014}
	\textsc{S. C. Brenner, P. Monk and J. Sun},
	\textit{$C^0$ interior penalty Galerkin method for biharmonic eigenvalue problems},
	Spectral and high order methods for partial differential equations, ICOSAHOM 2014, pp. 3--15.
	Lect. Notes Comput. Sci. Eng., 106, Springer, Cham, 2015
		
	\bibitem{BSZZ2013}
	\textsc{S.C. Brenner SC, L. Sung, H. Zhang and Y. Zhang},
	\textit{A Morley finite element method for the displacement
		obstacle problem of clamped Kirchhoff plates},
	J. Comput. Appl. Math., {\bf 254} (2013), pp. 31--42.
	
	
	
	\bibitem{Brezzi2013_CMAME}
	\textsc{F. Brezzi and L. D. Marini}, 
	\textit{Virtual element methods for plate bending problems}, 
	Comput. Methods Appl. Mech. Engrg., {\bf 253} (2013), pp. 455--462.
	
	
	
	\bibitem{CMS2016}
	\textsc{A. Cangiani, G. Manzini and O.J. Sutton},
	\textit{Conforming and nonconforming virtual element methods for elliptic problems},
	IMA J. Numer. Anal., {\bf 37} (2017), pp. 1317--1354.
	
	\bibitem{CKP-SINUM2023}
	\textsc{C. Carstensen, R. Khot and A.K. Pani},
	\textit{Nonconforming virtual elements for the biharmonic
		equation with Morley degrees of freedom on polygonal meshes},
	SIAM J. Numer. Anal., {\bf 61}(5) (2023), pp. 2460--2484.
	
	\bibitem{CH2020}
	\textsc{L. Chen and X. Huang},
	\textit{Nonconforming virtual element method for $2m$th order partial differential
		equations in $\R^n$}, Math. Comp., {\bf 89}(324) (2020), pp. 1711--1744.
	
	\bibitem{Chinosi2016_CMA}
	\textsc{C. Chinosi and L. D. Marini}, 
	\textit{Virtual element method for fourth order problems: $L^2$-estimates}, 
	Comput. Math. Appl., {\bf 72} (2016), pp. 1959–1967.
	
	\bibitem{Ciarlet2002_FEM}
	\textsc{P. G. Ciarlet}, 
	\textit{The Finite Element Method for Elliptic Problems}, 
	SIAM, (2002).
	
	\bibitem{Das2024_CMA}
	\textsc{A. Das, B. P. Lamichhane, and N. Nataraj}, 
	\textit{A unified mixed finite element method for fourth-order time-dependent problems using biorthogonal systems}, 
	Comput. Math. Appl., {\bf 165} (2024), pp. 52--69.
	
	\bibitem{DNR2025_JSC}
		\textsc{A. Das,  N. Nataraj and G.~C. Remesan}, 
\textit{Semi and fully discrete analysis of extended Fisher-Kolmogorov equation with nonstandard FEMs for space discretisation},
 J. Sci. Comput. {\bf 104}(1) (2025), Paper No. 14, 44 pp. 
	
	
	\bibitem{Danumjaya-Pani_2006}
	\textsc{P. Danumjaya and A. Pani}, 
	\textit{Numerical methods for the extended Fisher-Kolmogorov
		(EFK) equation}, Int. J. Numer. Anal. Model., {\bf 3} (2) (2006), pp. 186--210.
	
	\bibitem{Danumjaya2021_CMA}
	\textsc{P. Danumjaya, A. K. Pany, and A. K. Pani}, 
	\textit{Morley FEM for the fourth-order nonlinear reaction-diffusion problems}, 
	Comput. Math. Appl., {\bf 99} (2021), pp. 229--245.
	
	
	
	\bibitem{Dedner2022_IMA}
	\textsc{A. Dedner and A. Hodson}, 
	\textit{Robust nonconforming virtual element methods for general fourth-order problems with varying coefficients}, 
	IMA J. Numer. Anal., {\bf 42} (2022), pp. 1364–1399.
	
	\bibitem{Dedner2024_JSC}
	\textsc{A. Dedner and A. Hodson}, 
	\textit{A higher order nonconforming virtual element method for the Cahn–Hilliard equation}, 
	J. Sci. Comput., {\bf 101} (2024), p. 81.
	

	
	\bibitem{DE2022}
	\textsc{Z. Dong and A. Ern},
	\textit{Hybrid high-order and weak Galerkin methods for the biharmonic problem}, 
	SIAM J. Numer. Anal., {\bf 60}(5) (2022), pp. 2626--2656.


   \bibitem{DE_IMA2024}
   \textsc{Z. Dong and A. Ern},
    \textit{$C^0$-hybrid high-order methods for biharmonic problems}, 
    IMA J. Numer. Anal., {\bf 44}(1) (2024), pp. 24-52.
	
	\bibitem{Dong2023_JSC}
	\textsc{Z. Dong and L. Mascotto},
	\textit{$hp$-optimal interior penalty discontinuous Galerkin methods for the biharmonic problem},
	J. Sci. Comput., \textbf{96} (2023), no.~1, Paper No.~30, 32~pp.
	
	
	\bibitem{Elliott1989_SINUM}
	\textsc{C. M. Elliott and D. A. French}, 
	\textit{A nonconforming finite-element method for the two-dimensional Cahn--Hilliard equation}, 
	SIAM J. Numer. Anal., {\bf 26} (1989), pp. 884--903.
	
	
	\bibitem{Feng2007_MathComp}
	\textsc{X. Feng, Y. He, and C. Liu},
	\textit{Analysis of finite element approximations of a phase field model for two-phase fluids},
	Math. Comp., {\bf 76} (2007), pp.~539--571.
	
	
	\bibitem{GR}
	\textsc{V. Girault and P.A. Raviart},
	\textit{Finite Element Methods for Navier-Stokes Equations},
	Springer-Verlag, Berlin, 1986.
	
	
	\bibitem{Georgoulis2009_IMA}
	\textsc{E. H. Georgoulis and P. Houston}, 
	\textit{Discontinuous Galerkin methods for the biharmonic problem}, 
	IMA J. Numer. Anal., {\bf 29} (2009), no. 3, pp. 573--594.
	
	
	\bibitem{Gudi-Gupta_camwa2013}
	\textsc{T. Gudi and H. Gupta}, 
	\textit{A fully discrete $C^0$ interior penalty Galerkin approximation of the
		extended Fisher-Kolmogorov equation}, J. Comput. Appl. Math., {\bf 247} (2013), pp. 1--16.
	
			\bibitem{HR1990}
	\textsc{J. Heywood and R. Rannacher}, 
	\textit{Finite element approximation of the nonstationary Navier-
		Stokes problem, IV. Error analysis for second-order time discretization},
	SIAM J. Numer. Anal., {\bf 19}, (1990), pp. 275–311.
	
	\bibitem{Huang2021_JCAM}
	\textsc{J. Huang and Y. Yu}, 
	\textit{A medius error analysis for nonconforming virtual element methods for Poisson and biharmonic equations}, 
	J. Comput. Appl. Math., {\bf 386} (2021).
	
	
	
	\bibitem{Kay2009_SINUM}
	\textsc{D. Kay, V. Styles, and E. Süli}, 
	\textit{Discontinuous Galerkin finite element approximation of the Cahn--Hilliard equation with convection}, 
	SIAM J. Numer. Anal., {\bf 47} (2009), no. 4, pp. 2660--2685.
	
	
	
	\bibitem{Khot2025_MathComp}
	\textsc{R. Khot, D. Mora, and R. Ruiz-Baier}, 
	\textit{Virtual element methods for Biot-Kirchhoff poroelasticity}, 
	Math. Comp., {\bf 94} (2025), no. 353, pp. 1101--1146.
	
	\bibitem{Li2023_IMA}
	\textsc{H. Li, P. Yin, and Z. Zhang}, 
	\textit{A $C^0$ finite element method for the biharmonic problem with Navier boundary conditions in a polygonal domain}, 
	IMA J. Numer. Anal., {\bf 43} (2023), no.~3, pp.~1779--1801.
	
	\bibitem{Li2021_IMA}
	\textsc{M. Li, J. Zhao, C. Huang, and S. Chen}, 
	\textit{Conforming and nonconforming VEMs for the fourth-order reaction-subdiffusion equation: A unified framework}, 
	IMA J. Numer. Anal., {\bf 10} (2021), no.~3, pp.~2238--2300.

	
	\bibitem{Lovadina2021_M2AN}
	\textsc{C. Lovadina, D. Mora, and I. Vel\'asquez}, 
	\textit{A virtual element method for the von K\'arm\'an equations}, 
	ESAIM Math. Model. Numer. Anal., {\bf 55} (2021), no.~2, pp.~533--560.
	
	
	\bibitem{Mora2018_M2AN}
	\textsc{D. Mora, G. Rivera, and I. Vel\'asquez}, 
	\textit{A virtual element method for the vibration problem of Kirchhoff plates}, 
	ESAIM Math. Model. Numer. Anal., {\bf 52} (2018), no.~4, pp.~1437--1456.
	
	\bibitem{Mora2021_IMA}
	\textsc{D. Mora, C. Reales, and A. Silgado}, 
	\textit{A $C^1$-virtual element method of high order for the Brinkman equations in stream function formulation with pressure recovery}, 
	IMA J. Numer. Anal., {\bf 42} (2021), no~4,~pp.~3632–3674.
	
	
	\bibitem{Mora2025_SISC}
	\textsc{D. Mora and A. Silgado}, 
	\textit{Stream virtual elements for the Navier-Stokes system: nonstandard boundary conditions and variable recovery algorithms}, 
	SIAM J. Sci. Comput., {\bf 47} (2025), no.~1, pp.~A207--A237.
	
	\bibitem{Morley}
	\textsc{L.S.D. Morley},
	\textit{The triangular equilibrium element in the solution of plate bending problems}, 
	Aero. Quart., {\bf 19} (1968), pp. 149--169.
	
	
	\bibitem{Mozolevski2003_CMAM}
	\textsc{I. Mozolevski and E. S\"uli}, 
	\textit{A priori error analysis for the $hp$-version of the discontinuous Galerkin finite element method for the biharmonic equation}, 
	Comput. Methods Appl. Math., {\bf 3} (2003), pp.~596--607.
	
	\bibitem{Pei2023_CMA}
	\textsc{L. Pei, C. Zhang, and M. Li}, 
	\textit{Dissipative nonconforming virtual element method for the fourth-order nonlinear extended Fisher-Kolmogorov equation}, 
	Comput. Math. Appl., {\bf 152} (2023), pp. 28–45.
	
	
	\bibitem{Shylaja2024_ACM}
	\textsc{D. Shylaja and S. Kumar}, 
	\textit{Morley type virtual element method for von Kármán equations}, 
	Adv. Comput. Math., {\bf 50} (2024), no.~5, Paper No.~94, 31 pp.
	
	
	\bibitem{Vacca_parabolico}
	\textsc{G. Vacca and L. Beir\~ao da Veiga}, 
	\textit{Virtual element methods for parabolic problems on polygonal meshes},
	Numer. Methods Partial Differential Equation, {\bf 31} (2015), pp. 2110--2134.
	
	
	\bibitem{Zhao2016_M3AS}
	\textsc{J. Zhao, S. Chen, and B. Zhang}, 
	\textit{The nonconforming virtual element method for plate bending problems}, 
	Math. Models Methods Appl. Sci., {\bf 26} (2016), pp. 1671–1687.

	
	\bibitem{Zhao2018_JSC}
	\textsc{J. Zhao, B. Zhang, S. Chen, and S. Mao}, 
	\textit{The Morley-type virtual element for plate bending problems}, 
	J. Sci. Comput., {\bf 76} (2018), pp. 610–629.
	
	\bibitem{Zhao2023_MathComp}
	\textsc{J. Zhao, S. Mao, B. Zhang, and F. Wang},
	\textit{The interior penalty virtual element method for the biharmonic problem},
	Math. Comp., {\bf 92} (2023), no.~342, pp.~1543--1574.
	
	\bibitem{Zhao2019_parabolic}		
	\textsc{J. Zhao, B. Zhang, and X. Zhu}, 
	\textit{The nonconforming virtual element method for parabolic problems},
	Appl. Numer. Math. {\bf 143} (2019), pp. 97--113.	
	
	
	
\end{thebibliography}
\end{document}